\documentclass[12pt]{amsart}  

\usepackage{graphicx}
\usepackage{here} 
\usepackage{array} 

\usepackage{vmargin}
\usepackage{amssymb}
\usepackage{mathrsfs}
\usepackage[all]{xy}
\usepackage[usenames,dvipsnames]{color}
\usepackage{soul}
\RequirePackage[colorlinks,linkcolor=blue,citecolor=LimeGreen,urlcolor=red]{hyperref} 
\usepackage{amsmath}

\newtheorem{theorem}{Theorem}[section]

\newtheorem{lemma}[theorem]{Lemma}

\newtheorem{prop}[theorem]{Proposition}
\newtheorem{remark}[theorem]{Remark}

\numberwithin{equation}{section}

\newcommand{\R}{\mathbb{R}}

\newcommand{\Q}{\mathbb{Q}}
\newcommand{\N}{\mathbb{N}}

\newcommand{\T}{\mathbb{T}}
\renewcommand{\S}{{\mathbb S}}

\newcommand{\func}[3]{#1 : #2 \longrightarrow #3}

\newcommand{\disp}{\displaystyle}
\newcommand{\abs}[1]{\left|#1\right|}
\newcommand{\eps}{\varepsilon}
\newcommand{\norm}[1]{\left\|#1\right\|}

\renewcommand{\leq}{\leqslant}
\renewcommand{\geq}{\geqslant}
\renewcommand{\bar}{\overline}
\renewcommand{\tilde}{\widetilde}
\newcommand{\pa}[1]{\left(#1\right)}
\newcommand{\cro}[1]{\left[#1\right]}
\newcommand{\br}[1]{\left\{#1\right\}}
\newcommand\restr[2]{{
  \left.\kern-\nulldelimiterspace 
  #1 
  \right|_{#2} 
  }}

\def\signmb{\bigskip \begin{center} {\sc
Marc Briant\par\vspace{3mm}
Sorbonne Universit\'es, UPMC Univ. Paris 06/ CNRS\par
UMR 7598, Laboratoire Jacques-Louis Lions,\par
F-75005, Paris, France\par
\vspace{3mm}
e-mail:} \tt{briant.maths@gmail.com} \end{center}}

\def\signed{\bigskip \begin{center} {\sc
Esther S. Daus\par\vspace{3mm}
Vienna University of Technology\par
Institute for Analysis and Scientific Computing\par
Wiedner Hauptstrasse 8-10,
1040 Vienna, Austria\par
\vspace{3mm}
e-mail:} \tt{esther.daus@tuwien.ac.at} \end{center}}

\begin{document} 

\title[Multi-species Boltzmann equation]{The Boltzmann equation for a multi-species mixture close to global equilibrium}
\author{Marc Briant}
\author{Esther S. Daus}
\thanks{The first author was partly supported by the $150^{th}$ Anniversary Postdoctoral Mobility Grant of the London Mathematical Society and the Division of Applied Mathematics at Brown University.}
\thanks{The second author acknowledges partial support from the Austrian Science Fund (FWF), grants
P24304, P27352, and W1245, and the Austrian-French Program of the Austrian Exchange Service (\"OAD)}

\begin{abstract}
We study the Cauchy theory for a multi-species mixture, where the different species can have different masses, in a perturbative setting on the $3$-dimensional torus. The ultimate aim of this work is to obtain existence, uniqueness and exponential trend to equilibrium of solutions to the multi-species Boltzmann equation in $L^1_vL^\infty_x(m)$, where $m\sim (1+\abs{v}^k)$ is a polynomial weight. We prove the existence of a spectral gap for the linear multi-species Boltzmann operator allowing different masses, and then we establish a semigroup property thanks to a new explicit coercive estimate for the Boltzmann operator. Then we develop an $L^2-L^\infty$ theory \textit{\`a la Guo} for the linear perturbed equation. Finally, we combine the latter results with a decomposition of the multi-species Boltzmann equation in order to deal with the full equation. We emphasize that dealing with different masses induces a loss of symmetry in the Boltzmann operator which prevents the direct adaptation of standard mono-species methods (\textit{e.g.} Carleman representation, Povzner inequality). Of important note is the fact that all methods used and developed in this work are constructive. Moreover, they do not require any Sobolev regularity and the $L^1_vL^\infty_x$ framework is dealt with for any $k>k_0$, recovering the optimal physical threshold of finite energy $k_0=2$ in the particular case of a multi-species hard spheres mixture with same masses.
\end{abstract}

\maketitle

\vspace*{10mm}


%

\smallskip
\textbf{Acknowledgements:} The second author wants to thank Ansgar J\"{u}ngel for his valuable help.

\tableofcontents

\section{Introduction} \label{sec:intro}
The present work establishes existence, uniqueness, positivity and exponential trend to equilibrium for the multi-species Boltzmann equation close to equilibrium, which is used in physics and biology to model the evolution of a dilute gaseous mixture with different masses. The physically most relevant space for such a Cauchy theory is the space of density functions that only have finite mass and energy, which are the first and second moments in the velocity variable. This present article proves the result in the space $L^1_vL^\infty_x(1+\abs{v}^{k})$  for any $k>k_0$, where $k_0$ is an explicit threshold depending heavily on the differences of the masses, recovering the physically optimal threshold $k_0=2$ when all the masses of the mixture are the same and the particles are approximated to be hard spheres.

\par We are thus interested in the evolution of a dilute gas on the torus $\T^3$ composed of $N$ different species of chemically non-reacting mono-atomic particles, which can be modeled by the following system of Boltzmann equations, stated on $\R^+\times\T^3\times\R^3$,

\begin{equation}\label{multiBE}
\forall\: 1\leq i \leq N, \quad \partial_tF_i(t,x,v) + v\cdot \nabla_x F_i(t,x,v) = Q_i(\mathbf{F})(t,x,v)
\end{equation}
with initial data
$$\forall\: 1\leq i \leq N,\:\forall (x,v)\in \T^3\times\R^3, \quad F_i(0,x,v) = F_{0,i}(x,v).$$
Note that the distribution function of the system is given by the vector $\mathbf{F} = (F_1,\dots,F_N),$ with $F_i$ describing the $i^{th}$ species at time $t$, position $x$ and velocity $v$.

\bigskip

The Boltzmann operator $\mathbf{Q}(\mathbf{F})=(Q_1(\mathbf{F}),\ldots, Q_N(\mathbf{F}))$ is given for all $i$ by 
$$Q_i(\mathbf{F}) = \sum\limits_{j=1}^N Q_{ij}(F_i,F_j),$$
where $Q_{ij}$ describes interactions between particles of either the same ($i=j$) or of different ($i\neq j$) species and are local in time and space.
$$Q_{ij}(F_i,F_j)(v) =\int_{\R^3\times \mathbb{S}^{2}}B_{ij}\left(|v - v_*|,\mbox{cos}\:\theta\right)\left[F_i'F_j^{'*} - F_iF_j^*\right]dv_*d\sigma,$$
where we used the shorthands $F_i'=F_i(v')$, $F_i=F_i(v)$, $F_j^{'*}=F_j(v'_*)$ and $F_j^*=F_j(v_*)$. 
$$\left\{ \begin{array}{rl} \displaystyle{v'} & \displaystyle{=\frac{1}{m_i+m_j}\pa{m_iv+m_jv_* +  m_j|v-v_*|\sigma}} \vspace{2mm} \\ \vspace{2mm} \displaystyle{v' _*}&\displaystyle{=\frac{1}{m_i+m_j}\pa{m_iv+m_jv_* -m_i  |v-v_*|\sigma}} \end{array}\right., \: \mbox{and} \quad \mbox{cos}\:\theta = \left\langle \frac{v-v_*}{\abs{v-v_*}},\sigma\right\rangle .$$
Note that these expressions imply that we deal with gases where only binary elastic collisions occur (the mass $m_i$ of all molecules of species $i$ remains the same, since there is no reaction). Indeed, $v'$ and $v'_*$ are the velocities of two molecules of species $i$ and $j$ before collision giving post-collisional velocities $v$ and $v_*$ respectively, with conservation of momentum and kinetic energy:
\begin{equation}\label{elasticcollision}
\begin{split}
m_iv + m_jv_* &= m_iv' + m_jv'_*,
\\\frac{1}{2}m_i\abs{v}^2 + \frac{1}{2}m_j\abs{v_*}^2 &= \frac{1}{2}m_i\abs{v'}^2 + \frac{1}{2}m_j\abs{v'_*}^2.
\end{split}
\end{equation}
\par The collision kernels $B_{ij}$ are nonnegative, moreover they contain all the information about the interaction between two particles and are determined by physics. We mention at this point that one can derive this type of equations from Newtonian mechanics at least formally in the case of single species \cite{Ce}\cite{CIP}. The rigorous validity of the mono-species Boltzmann equation from Newtonian laws is known for short times (Landford's theorem \cite{La} or more recently \cite{GST}\cite{PSS}).  

\bigskip


\subsection{The perturbative regime and its motivation}\label{subsec:perturbregime}
Using the standard changes of variables $(v,v_*) \mapsto (v',v'_*)$ and $(v,v_*) \mapsto (v_*,v)$ (note the lack of symmetry between $v'$ and $v'_*$ compared to $v$ for the second transformation due to different masses) together with the symmetries of the collision operators (see \cite{Ce}\cite{CIP}\cite{Vi2} among others and \cite{DMS}\cite{DJMZ} and in particular \cite{BGS} for multi-species specifically), we recover the following weak forms:
$$\int_{\R^3}Q_{ij}(F_i,F_j)(v)\psi_i(v)\:dv =\int_{\R^6}\int_{\S^2}B_{ij}(|v-v_*|,\cos(\theta))F_iF_j^* \left(\psi_i'-\psi_i\right)\:d\sigma dvdv_*$$
and
\begin{equation}\label{symmetry property Qij}
\begin{split}
& \int_{\R^3}Q_{ij}(F_i,F_j)(v)\psi_i(v)\:dv + \int_{\R^3}Q_{ji}(F_j,F_i)(v)\psi_j(v)\:dv=\\
&\quad - \frac{1}{2}\int_{\R^6}\int_{\S^2}B_{ij}(|v-v_*|,\cos(\theta))\left(F_i'F_j^* - F_iF_j^*\right)\left(\psi_i' + \psi_j'^* - \psi_i - \psi_j^*\right)\:d\sigma dvdv_*.
\end{split}
\end{equation}

Thus 
\begin{equation}\label{invariantsQij}
\sum\limits_{i,j=1}^N\int_{\R^3} Q_{ij}(F_i,F_j)(v)\psi_i(v)\:dv =0
\end{equation}
if and only if $\psi(v)$ belongs to $\mbox{Span}\br{\mathbf{e_1},\dots,\mathbf{e_N},v_1\mathbf{m},v_2\mathbf{m},v_3\mathbf{m},\abs{v}^2\mathbf{m}}$, where $\mathbf{e_k}$ stands for the $k^{th}$ unit vector in $\R^N$ and $\mathbf{m} = (m_1,\dots,m_N)$.
The fact that we need to sum over $i$ has interesting consequences and implies a fundamental difference compared with the single-species Boltzmann equation. In particular it implies conservation of the total number density $c_{\infty,i}$ of each species, of the total momentum of the gas $\rho_\infty u_\infty$ and its total energy $3\rho_\infty\theta_\infty /2$:
\begin{equation}\label{conservationlaws}
\begin{split}
\forall t\geq 0,\quad& c_{\infty,i} = \int_{\T^3\times\R^3} F_i(t,x,v)\:dxdv \quad (1\leq i \leq N)
\\& u_{\infty} = \frac{1}{\rho_\infty}\sum\limits_{i=1}^N\int_{\T^3\times\R^3} m_ivF_i(t,x,v)\:dxdv
\\&\theta_{\infty} = \frac{1}{3\rho_\infty}\sum\limits_{i=1}^N\int_{\T^3\times\R^3} m_i\abs{v-u_\infty}^2F_i(t,x,v)\:dxdv,
\end{split}
\end{equation}
where $\rho_\infty = \sum_{i=1}^Nm_ic_{\infty,i}$ is the global density of the gas. Note that this already shows intricate interactions between each species and the total mixture itself.
\par The operator $\mathbf{Q}=(Q_1,\dots,Q_N)$ also satisfies a multi-species version of the classical H-theorem \cite{DMS} which implies that any local equilibrium, i.e. any function $\mathbf{F}=(F_1,\dots,F_N)$ being the maximum of the Boltzmann entropy, has the form of a local Maxwellian, that is
$$\forall\:1\leq i \leq N,\:\: F_i(t,x,v) = c_{\mbox{\scriptsize{loc}},i}(t,x)\pa{\frac{m_i}{2\pi k_B \theta_{\mbox{\scriptsize{loc}}}(t,x)}}^{3/2}\mbox{exp}\cro{-m_i\frac{\abs{v-u_{\mbox{\scriptsize{loc}}}(t,x)}^2}{2k_B\theta_{\mbox{\scriptsize{loc}}}(t,x)}}.$$
Here $k_B$ is the Boltzmann constant and, denoting the total local mass density by $\rho_{\mbox{\scriptsize{loc}}} = \sum_{i=1}^N m_ic_{\mbox{\scriptsize{loc}},i}$, we used the following local definitions
$$\forall \:1\leq i \leq N, \quad c_{\mbox{\scriptsize{loc}},i}(t,x) = \int_{\R^3}F_i(t,x,v)\:dv,$$
$$u_{\mbox{\scriptsize{loc}}}(t,x)=\frac{1}{\rho_{\mbox{\scriptsize{loc}}}}\sum\limits_{i=1}^N\int_{\R^3}m_ivF_i\:dv,\quad \theta_{\mbox{\scriptsize{loc}}}(t,x) = \frac{1}{3\rho_{\mbox{\scriptsize{loc}}}}\sum\limits_{i=1}^N\int_{\R^3}m_i\abs{v-u_{\mbox{\scriptsize{loc}}}}^2F_i\:dv.$$
On the torus, this multi-species H-theorem also implies that the global equilibrium, i.e. a stationary solution $\mathbf{F}$ to $\eqref{multiBE}$, associated to the initial data $\mathbf{F_0}(x,v) =(F_{0,1},\dots,F_{0,N})$ is uniquely given by the global Maxwellian 
$$\forall\:1\leq i \leq N,\quad F_i(t,x,v) = F_i(v)= c_{\infty,i}\pa{\frac{m_i}{2\pi k_B \theta_{\infty}}}^{3/2}\mbox{exp}\cro{-m_i\frac{\abs{v-u_{\infty}}^2}{2k_B\theta_{\infty}}}.$$
By translating and rescaling the coordinate system we can always assume that $u_\infty=0$ and $k_B\theta_\infty=1$ so that the only global equilibrium is the normalized Maxwellian
\begin{equation}\label{mui}
\boldsymbol\mu =\pa{\mu_i}_{1\leq i \leq N} \quad\mbox{with}\quad \mu_i(v) = c_{\infty,i}\pa{\frac{m_i}{2\pi}}^{3/2}e^{-m_i\frac{\abs{v}^2}{2}}.
\end{equation}

\bigskip
The aim of the present article is to construct a Cauchy theory for the multi-species Boltzmann equation $\eqref{multiBE}$ around the global equilibrium $\boldsymbol\mu$. In other terms we study the existence, uniqueness and exponential decay of solutions of the form $F_i(t,x,v) = \mu_i(v) + f_i(t,x,v)$ for all $i$.
\par Under this perturbative regime, the Cauchy problem amounts to solving the perturbed multi-species Boltzmann system of equations
\begin{equation}\label{perturbedmultiBE}
\partial_t \mathbf{f} + v\cdot\nabla_x\mathbf{f} = \mathbf{L}(\mathbf{f}) + \mathbf{Q}(\mathbf{f}), 
\end{equation}
or equivalently in the non-vectorial form
$$\forall \:1\leq i\leq N,\quad \partial_t f_i + v\cdot\nabla_x f_i = L_i(\mathbf{f}) + Q_i(\mathbf{f}),$$
where $\mathbf{f}=(f_1,\dots,f_N)$ and the operator $\mathbf{L} =(L_1,\dots,L_N)$ is the linear Boltzmann operator given for all $1\leq i \leq N$ by
$$L_i(\mathbf{f}) = \sum\limits_{j=1}^N L_{ij}(f_i,f_j),$$
with
$$L_{ij}(f_i,f_j) = Q_{ij}(\mu_i,f_j)+Q_{ij}(f_i,\mu_j).$$
\par Since we are looking for solutions $\mathbf{F}$ preserving individual mass, total momentum and total energy $\eqref{conservationlaws}$ we have the equivalent perturbed conservation laws for $\mathbf{f} = \mathbf{F}-\boldsymbol\mu$ which are given by
\begin{equation}\label{perturbedconservationlaws}
\begin{split}
\forall t\geq 0,\quad& 0 = \int_{\T^3\times\R^3} f_i(t,x,v)\:dxdv \quad (1\leq i \leq N)
\\& 0 = \sum\limits_{i=1}^N\int_{\T^3\times\R^3} m_ivf_i(t,x,v)\:dxdv
\\&0 = \sum\limits_{i=1}^N\int_{\T^3\times\R^3} m_i\abs{v}^2f_i(t,x,v)\:dxdv.
\end{split}
\end{equation}

\bigskip


\subsection{Notations and assumptions on the collision kernel}\label{subsec:notations}

First, to avoid any confusion, vectors and vector-valued operators in $\R^N$ will be denoted by a bold symbol, whereas their components by the same indexed symbol. For instance, $\mathbf{W}$ represents the vector or vector-valued operator $(W_1,\dots,W_N)$. 
\par We define the Euclidian scalar product in $\R^N$ weighted by a vector $\mathbf{W}$ by
$$\langle \mathbf{f},\mathbf{g}\rangle_{\mathbf{W}}= \sum\limits_{i=1}^Nf_ig_i W_i.$$
In the case $\mathbf{W}=\mathbf{1}=(1,\dots,1)$ we may omit the index $\mathbf{1}$.

\bigskip
\textbf{Function spaces.} 
We define the following shorthand notation
$$\langle v \rangle = \sqrt{1+\abs{v}^2}.$$
\par The convention we choose is to index the space by the name of the concerned variable, so we have for $p$ in $[1,+\infty]$
$$L^p_{[0,T]} = L^p\pa{[0,T]},\quad L^p_{t} = L^p \left(\R^+\right),\quad L^p_x = L^p\left(\T^3\right), \quad L^p_v = L^p\left(\R^3\right).$$
\par For $\func{\mathbf{W}=(W_1, \ldots, W_N)}{\R^3}{\R^+}$ a strictly positive measurable function in $v$, we will use the following vector-valued weighted Lebesgue spaces defined by their norms 
\begin{equation*}\label{norm}
     \begin{array}{ll}
 \norm{f}_{L^2_{v}\pa{\mathbf{W}}} = \left(\sum\limits_{i=1}^N \norm{f_i}^2_{L^2_{v}(W_i)}\right)^{1/2},     &    \norm{f_i}_{L^2_{v}(W_i)}=\left\|f_i W_i(v)\right\|_{L^2_v},\\
 \norm{f}_{L^2_{x,v}\pa{\mathbf{W}}} = \left(\sum\limits_{i=1}^N \norm{f_i}^2_{L^2_{x,v}(W_i)}\right)^{1/2}, &    \norm{f_i}_{L^2_{x,v}(W_i)}=\left\|\|f_i\|_{L^2_x}W_i(v)\right\|_{L^2_v},\\
 \norm{f}_{L^{\infty}_{x,v}\pa{\mathbf{W}}} = \sum\limits_{i=1}^N \norm{f_i}_{L^{\infty}_{x,v}(W_i)},        &    \norm{f_i}_{L^{\infty}_{x,v}\pa{W_i}} = \sup\limits_{(x,v)\in \T^3 \times \R^3}\big(\left|f_i(x,v)\right|W_i(v)\big),\\
 \norm{f}_{L^1_vL^{\infty}_x\pa{\mathbf{W}}} = \sum\limits_{i=1}^N \norm{f_i}_{L^1_vL^{\infty}_x\pa{W_i}},   &    \norm{f_i}_{L^1_vL^{\infty}_x\pa{W_i}}=\left\|\sup\limits_{x \in \T^3}\left|f_i(x,v)\right|W_i(v)\right\|_{L^1_v}.
        \end{array}
 \end{equation*}
Note that $L^2_v(\mathbf{W})$ and $L^2_{x,v}(\mathbf{W})$ are Hilbert spaces with respect to the scalar products 
\begin{eqnarray*}
\langle \mathbf{f},\mathbf{g}\rangle_{L^2_{v}(\mathbf{W})}&=&\sum\limits_{i=1}^N \langle f_i,g_i \rangle_{L^2_{v}(W_i)}=\sum\limits_{i=1}^N \int_{\R^3} f_i g_i W_i^2 dv,
\\\langle\mathbf{f},\mathbf{g}\rangle_{L^2_{x,v}(\mathbf{W})}&=&\sum\limits_{i=1}^N \langle f_i,g_i \rangle_{L^2_{x,v}(W_i)}=\sum\limits_{i=1}^N \int_{\T^3 \times \R^3} f_i g_i W_i^2 dxdv.
\end{eqnarray*}

\bigskip
\textbf{Assumptions on the collision kernel.}

We will use the following assumptions on the collision kernels $B_{ij}$. 

\renewcommand{\labelenumi}{(H\theenumi)}
\begin{enumerate}
\item The following symmetry holds
$$B_{ij}(|v-v_*|,\cos\theta) = B_{ji}(|v-v_*|,\cos\theta)\quad\mbox{for }1\le i,j\le N.$$
\item The collision kernels decompose into the product
$$ B_{ij}(|v-v_*|,\cos\theta) = \Phi_{ij}(|v-v_*|)b_{ij}(\cos\theta),
\quad 1\le i,j\le N,$$
where the functions $\Phi_{ij}\ge 0$ are called kinetic part and $b_{ij}\ge 0$ angular part. This is a common assumption as it is technically more convenient and also covers a wide range of physical applications.
\item The kinetic part has the form of hard or Maxwellian ($\gamma=0$) potentials, \textit{i.e.}
$$\Phi_{ij}(|v-v_*|)=C_{ij}^{\Phi}|v-v_*|^{\gamma}, \quad C_{ij}^{\Phi}>0,~\:\gamma\in[0,1], \quad \forall\: 1 \leq i,j\leq N.$$
\item For the angular part, we assume a strong form of Grad's angular cutoff (first introduced in \cite{Gr1}), that is: there exist constants $C_{b1}$, $C_{b2}>0$ such that
for all $1\le i,j\le N$ and $\theta\in[0,\pi]$,
$$  0<b_{ij}(\cos\theta)\le C_{b1}|\sin\theta|\,|\cos\theta|, \quad b'_{ij}(\cos\theta)\le C_{b2}.$$
Furthermore, 
$$  C^b := \min_{1\le i\le N}\inf_{\sigma_1,\sigma_2\in\S^2}\int_{\S^2}\min\big\{	b_{ii}(\sigma_1\cdot\sigma_3),b_{ii}(\sigma_2\cdot\sigma_3)\big\}\:d\sigma_3 > 0. $$
\end{enumerate}

\noindent We emphasize here that the important cases of Maxwellian molecules ($\gamma=0$ and $b=1$) and of hard spheres ($\gamma=b=1$) are included in our study. We shall use the standard shorthand notations
\begin{equation}\label{constantsbij}
b_{ij}^\infty = \norm{b_{ij}}_{L^\infty_{[-1,1]}} \quad\mbox{and}\quad l_{b_{ij}} = \norm{b\circ \cos}_{L^1_{\S^2}}.
\end{equation}
\bigskip


\subsection{Novelty of this article}\label{subsec:novelty}
As mentioned previously, the present work proves the existence, uniqueness, positivity and exponential trend to equilibrium for the full nonlinear multi-species Boltzmann equation $\eqref{multiBE}$ in $L^1_vL^\infty_x\pa{\langle v \rangle^{k}}$ with the explicit threshold $k>k_0$ defined in Lemma \ref{lem:controlB}, when the initial data $\mathbf{F_0}$ is close enough to the global equilibrium $\boldsymbol\mu$. This is equivalent to solving the perturbed equation $\eqref{perturbedmultiBE}$ for small $\mathbf{f_0}$. This perturbative Cauchy theory for gaseous mixtures is completely new.
\par Moreover, one of the  major contributions of the present article is to combine and adapt several very recent strategies, combined with new hypocoercivity estimates, in order to develop a new constructive approach that allows to deal with polynomial weights without requiring any spatial Sobolev regularity. This is new even in the mono-species case even though the final result we obtain has recently been proved for the mono-species hard sphere model \cite{GMM}) (which we therefore also extend to more general hard and Maxwellian potential kernels.).
\par Also, as a by-product, we prove explicitly that the linear operator $\mathbf{L}-v\cdot\nabla_x$ generates a strongly continuous semigroup with exponential decay both in $L^2_{x,v}\pa{\boldsymbol\mu^{-1/2}}$ and in $L^\infty_{x,v}\pa{\langle v \rangle^\beta\boldsymbol\mu^{-1/2}}$; such constructive and direct results on the torus are new to our knowledge, even for the single-species Boltzmann equation.
\par At last, we derive new estimates in order to deal with different masses and the multi-species cross-interaction operators, and we also extend recent mono-species estimates to more general collision kernels. Note that the asymmetry of the elastic collisions requires to derive a new description of Carleman's representation of the Boltzmann operator as well as new Povzner-type inequalities suitable for this lack of symmetry.

\bigskip


\subsection{State of the art and strategy}\label{subsec:strategy}
Very little is known about any rigorous Cauchy theory for multi-species gases with different masses. We want to mention \cite{BGPS}, where a compactness result for the linear operator $\mathbf{K}:=\mathbf{L}+\boldsymbol{\nu}$ was proved in $L^2_v(\boldsymbol{\mu}^{-1/2})$. For multi-species gases with same masses, the recent work \cite{DJMZ} proved that the operator $\mathbf{L}$ has a spectral gap in $L^2_v\pa{\boldsymbol{\mu}^{-1/2}}$ and obtained an \textit{a priori} exponential convergence to equilibrium for the perturbed equation $\eqref{perturbedmultiBE}$ in $H^1_{x,v}\pa{\boldsymbol\mu^{-1/2}}$. We emphasize here that \cite{DJMZ} only studied the case of same masses $m_i = m_j$ for all $i$, $j$. On the contrary, the single-species Boltzmann equation in the perturbative regime around a global Maxwellian has been extensively studied over the past fifty years (see \cite{UkYa} for an exhaustive review). Starting with Grad \cite{Gr}, the Cauchy problem has been tackled in $L^2_vH^s_x\left(\mu^{-1/2}\right)$ spaces \cite{Uk}, in $H^s_{x,v}\left(\mu^{-1/2}(1+\abs{v})^k\right)$ \cite{Gu1}\cite{Yu} was then extended to $H^s_{x,v}\left(\mu^{-1/2}\right)$ where an exponential trend to equilibrium has also been obtained \cite{MN}\cite{Gu4}. Recently, \cite{GMM} proved existence and uniqueness for single-species Boltzmann equation in more the general spaces $\left(W^{\alpha,1}_v \cap W^{\alpha,q}_v\right)W^{\beta,p}_x\left((1+\abs{v})^k\right)$ for $\alpha \leq \beta$ and $\beta$ and $k$ large enough with explicit thresholds. The latter paper thus includes $L^1_vL^\infty_x\pa{\langle v \rangle^{k}}$. All the results presented above hold in the case of the torus for hard and Maxwellian potentials. We refer the reader interested in the Cauchy problem to the review \cite{UkYa}.
\par All the works mentioned above involve to working in spaces with derivatives in the space variable $x$ (we shall discuss some of the reasons later) with exponential weight. The recent breakthrough \cite{GMM} gets rid of both the Sobolev regularity and the exponential weight but uses a new extension method which still requires to have a well-established linear theory in $H^s_{x,v}\pa{\mu^{-1/2}}$. 

\bigskip

\par Our strategy can be decomposed into four main steps and we now describe each of them and their link to existing works.

\bigskip

\textbf{Step 1: Spectral gap for the linear operator in $L^2_v\pa{\mathbf{\boldsymbol\mu^{-1/2}}}$.} It has been known for long that the single-species linear Boltzmann operator $L$ is a self-adjoint non positive linear operator in the space $L^2_v\left(\mu^{-1/2}\right)$. Moreover it has a spectral gap $\lambda_0$. This has been proved in \cite{Ca2}\cite{Gr1}\cite{Gr2} with non constructive methods for hard potential with cutoff and in \cite{Bob1}\cite{Bob2} in the Maxwellian case. These results were made constructive in \cite{BM}\cite{Mo1} for more general collision operators. One can easily extend this spectral gap to Sobolev spaces $H^s_v\left(\mu^{-1/2}\right)$ (see for instance \cite{GMM} Section $4.1$).
\par Recently, \cite{DJMZ} proved the existence of an explicit spectral gap for the operator $\mathbf{L}$ for multi-species mixtures where all the masses are the same ($m_i=m_j$). Our constructive spectral gap estimate in $L^2_v\pa{\boldsymbol\mu^{-1/2}}$ closely follows their methods that consist in proving that the cross-interactions between different species do not perturb too much the spectral gap that is known to exist for the diagonal operator $L_{ii}$ (single-species operators). We emphasize here that not only we adapt the methods of \cite{DJMZ} to fit the different masses framework but we also derive estimates on the collision frequencies that allow us to get rid of their strong requirement on the collision kernels: $B_{ij}\leq \beta B_{ii}$ for all $i$, $j$. The latter assumption is indeed physically irrelevant in our framework.

\bigskip
\textbf{Step 2: $L^2_{x,v}\pa{\boldsymbol\mu^{-1/2}}$ theory for the full perturbed linear operator.} The next step is to prove that the existence of a spectral gap for $\mathbf{L}$ in the sole velocity variable can be transposed to $L^2_{x,v}\pa{\boldsymbol\mu^{-1/2}}$ when one adds the skew-symmetric transport operator $-v\cdot\nabla_x$. In other words, we prove that $\mathbf{G} = \mathbf{L} - v\cdot\nabla_x$  generates a strongly continuous semigroup in $L^2_{x,v}\pa{\boldsymbol\mu^{-1/2}}$ with exponential decay.
\par One thus wants to derive an exponential decay for solutions to the linear perturbed Boltzmann equation
$$\partial_t \mathbf{f} + v\cdot\nabla_x \mathbf{f} = L\pa{\mathbf{f}}.$$
A straightforward use of the spectral gap $\lambda_L$ of $\mathbf{L}$ shows for such a solution
$$\frac{d}{dt}\norm{\mathbf{f}}^2_{L^2_{x,v}\pa{\boldsymbol\mu^{-1/2}}} \leq -2\lambda_L \norm{\mathbf{f}-\pi_\mathbf{L}\pa{\mathbf{f}}}^2_{L^2_{x,v}\pa{\boldsymbol\mu^{-1/2}}},$$
where $\pi_\mathbf{L}$ stands for the orthogonal projection in $L^2_v\pa{\boldsymbol\mu^{-1/2}}$ onto the kernel of the operator $\mathbf{L}$. This inequality exhibits the hypocoercivity of $\mathbf{L}$. Roughly speaking, the exponential decay in $L^2_{x,v}\pa{\boldsymbol{\mu}^{-1/2}}$ would follow for solutions $\mathbf{f}$ if the microscopic part $\pi_\mathbf{L}^\bot(\mathbf{f})= \mathbf{f} -\pi_\mathbf{L}(\mathbf{f})$ controls the fluid part which has the following form (see Section \ref{sec:spectralgap})
$$\forall 1\leq i \leq N,\quad \pi_\mathbf{L}(\mathbf{f})_i(t,x,v) = \cro{a_i(t,x) + b(t,x)\cdot v + c(t,x)\frac{|v|^2-3m_i^{-1}}{2}}m_i\mu_i(v),$$
where $a_i(t,x), c(t,x) \in \R$ and $b(t,x) \in \R^3$ are the coordinates of an orthogonal basis.
\par The standard strategies in the case of the single-species Boltzmann equation are based on higher Sobolev regularity either from hypocoercivity methods \cite{MN} or elliptic regularity of the coefficients $a$, $b$ and $c$ \cite{Gu3}\cite{Gu4}. Roughly speaking one has \cite{Gu3}\cite{Gu4}
\begin{equation}\label{fluidmicroderivative}
\Delta \pi_L(f) \sim \partial^2 \pi_L^\bot f + \mbox{higher order terms},
\end{equation}
which can be combined with elliptic estimates to control the fluid part by the microscopic part in Sobolev spaces $H^s$. Our main contribution to avoid involving high regularity is based on an adaptation of the recent work \cite{EGKM} (dealing with the single-species Boltzmann equation with diffusive boundary conditions). The key idea consists in integrating against test functions that contains a weak version of the elliptic regularity of $\mathbf{a}(t,x)$, $b(t,x)$ and $c(t,x)$. Basically, the elliptic regularity of $\pi_{\mathbf{L}}\pa{\mathbf{f}}$ will be recovered thanks to the transport part applied to these test functions while on the other side $\mathbf{L}$ encodes the control by $\pi_{\mathbf{L}}^\bot\pa{\mathbf{f}}$.
\par It has to be emphasized that thanks to boundary conditions, \cite{EGKM} only needed the conservation of mass whereas in our case this ``weak version'' of estimates $\eqref{fluidmicroderivative}$ strongly relies on all the conservation laws. The choice of test functions thus has to take into account the delicate interaction between each species and the total mixture we already pointed out. This leads to intricate technicalities since for each species we need to deal with different reference rates of decay $m_i$. Finally, our proof also involves elliptic regularity in negative Sobolev spaces to deal with $\partial_t \mathbf{a}$, $\partial_t b$ and $\partial_t c$.

\bigskip
\textbf{Step 3: $L^\infty_{x,v}\pa{\langle v \rangle^\beta\boldsymbol\mu^{-1/2}}$ theory for the full nonlinear equation.} Thanks to the first two steps we have a satisfactory $L^2$ semigroup theory for the full linear operator. Unfortunately, as it is already the case for the single-species Boltzmann equation (see \cite{Ce}\cite{CIP} or \cite{Vi2} for instance), the underlying $L^2_{x,v}$-norm is not an algebraic norm for the nonlinear operator $\mathbf{Q}$ whereas the $L^\infty_{x,v}$-norm is.
\par The key idea of proving a semigroup property in $L^\infty$ is thanks to an $L^2-L^\infty$ theory ``\`a la Guo'' \cite{Gu6}, where the $L^\infty$-norm will be controlled by the $L^2$-norm along the characteristics. As we shall see, each component $L_i$ can be decomposed into $L_i= K_i - \nu_i$ where $\nu_i(f)=\nu_i(v)f_i$ is a multiplicative operator. If we denote by $\mathbf{S}_\mathbf{G}(t)$ the semigroup generated by $\mathbf{G} = \mathbf{L}-v\cdot\nabla_x$, we have the following implicit Duhamel representation of its $i^{th}$ component along the characteristics
$$S_\mathbf{G}(t)_i = e^{-\nu_i(v)t} + \int_0^t e^{-\nu_i(v)(t-s)}K_i\cro{\mathbf{S}_\mathbf{G}(s)} \:ds.$$
Following the idea of Vidav \cite{Vid} and later used in \cite{Gu6}, an iteration of the above should yield a certain compactness property. Hiding here all the cross-interactions, we end up with
\begin{equation*}
\begin{split}
\mathbf{S}_\mathbf{G}(t) =& e^{-\boldsymbol\nu(v)t} + \int_0^t e^{-\boldsymbol\nu(v)(t-s)}\mathbf{K}e^{-\boldsymbol\nu(v)s}\:ds 
\\&+ \int_0^t\int_0^s e^{-\boldsymbol\nu(v)(t-s)}\mathbf{K}e^{-\boldsymbol\nu(v)(s-s_1)}\mathbf{K}\cro{\mathbf{S}_\mathbf{G}(s_1)} \:ds_1ds.
\end{split}
\end{equation*}
We shall prove that $\mathbf{K}$ is compact and is a kernel operator. The first two terms will be easily estimated and the last term will be roughly of the form
$$\int_0^t\int_0^s\int_{v_1,v_2 \:\mbox{\scriptsize{bounded}}}\abs{\mathbf{S}_\mathbf{G}(s_1, x-(t-s)v-(s-s_1)v_1,v_2}\:dv_2dv_1ds_1ds.$$ 
The double integration implies that $v_1$ and $v_2$ are independent and we can thus perform a change of variables which changes the integral in $v_1$ into an integral over $\T^3$ that we can bound thanks to the previous $L^2$ theory. For integrability reasons, this third step actually proves that $\mathbf{G}$ generates a strongly continuous semigroup with exponential decay in $L^\infty\pa{\langle v \rangle^\beta\boldsymbol\mu^{-1/2}}$ for $\beta > 3/2$.
\par Our work provides two key contributions to prove the latter result. First, to prove the desired pointwise estimate for the kernel of $\mathbf{K}$, we need to give a new representation of the operator in terms of the parameters $(v',v'_*)$ instead of $(v_*,\sigma)$. In the single-species case, such a representation is the well-known Carleman representation \cite{Ca2} and requires integration onto the so-called Carleman hyperplanes $\langle v'-v , v'_* -v \rangle =0$. However, when particles have different masses, the lack of symmetry between $v'$ and $v'_*$ compared to $v$ obliges us to derive new Carleman admissible sets (some become spheres). Second, the decay of the exponential weight differs from one species to the other. To obtain estimates that are similar to the case of single-species we exhibit the property that $\mathbf{K}$ mixes the exponential rate of decay among the cross-interaction between species. This enables us to close the $L^\infty$ estimate for the first two terms of the iterated Duhamel representation.

\bigskip
\textbf{Step 4: Extension to polynomial weights and $L^1_vL^\infty_x$ space.} To conclude the present study, we develop an analytic and nonlinear version of the recent work \cite{GMM}, also recently adapted in a nonlinear setting \cite{Bri6}. The main strategy is to find a decomposition of the full linear operator $\mathbf{G}$ into $\mathbf{G_1}+\mathbf{A}$. We shall prove that $\mathbf{G_1}$ acts like a small perturbation of the operator $\mathbf{G_{\boldsymbol\nu}} = -v\cdot\nabla_x - \boldsymbol\nu(v)$ and is thus hypodissipative, and that $\mathbf{A}$ has a regularizing effect. The regularizing property of the operator $\mathbf{A}$ allows us to decompose the perturbative equation $\eqref{perturbedmultiBE}$ into a system of differential equations
\begin{eqnarray*}
\partial_t \mathbf{f_1} &=& \mathbf{G_1}\pa{\mathbf{f_1}} + \mathbf{Q}(\mathbf{f_1}+\mathbf{f_2},\mathbf{f_1}+\mathbf{f_2})
\\\partial_t \mathbf{f_2} + v\cdot\nabla_x \mathbf{f_2} &=& \mathbf{L}\pa{\mathbf{f_2}} + \mathbf{A}\pa{\mathbf{f_1}} \end{eqnarray*}
The first equation is solved in $L^\infty_{x,v}\pa{m}$ or $L^1_vL^\infty_x\pa{m}$ with the initial data $\mathbf{f_0}$ thanks to the hypodissipativity of $\mathbf{G_1}$. The regularity of $\mathbf{A}\pa{\mathbf{f_1}}$ allows us to use Step 3 and thus solve the second equation with null initial data in $L^\infty_{x,v}\pa{\langle v \rangle^\beta\boldsymbol\mu^{-1/2})}$. First, the existence of a solution to the system having exponential decay is obtained thanks to an iterative scheme combined with new estimates on the multi-species operators $\mathbf{G_1}$ and $\mathbf{A}$. Then uniqueness follows a new stability estimate in an equivalent norm (proposed in \cite{GMM}), that fits the dissipativity of the semigroup generated by $\mathbf{G}$. Finally, positivity of the unique solution comes from a different iterative scheme.
\par In the case of the single-species Boltzmann equation, the less regular weight $m(v)$ one can achieve with this method is determined by the hypodissipative property of $\mathbf{G_1}$ and gives $m=\langle v \rangle^k$ with $k>2$, which is indeed obtained also in the multi-species framework of same masses. In the general case of different masses, the threshold $k_0$ is more intricate (see Theorem \ref{theo:fullcauchy}), since it also depends on the different masses $m_i$.

\bigskip


\subsection{Organisation of the paper}\label{subsec:organisation}

The paper follows exactly the four steps described above.
\par Section \ref{sec:mainresults} gives a precise statement of the main theorems that will be proved in this work and the rest of the article is dedicated to the proof of these theorems.
\par Section \ref{sec:spectralgap} deals with the spectral gap of $\mathbf{L}$. The semigroup property in $L^2_{x,v}\pa{\boldsymbol\mu^{-1/2}}$ is treated in Section \ref{sec:L2theory}. This property is then passed on to $L^\infty_{x,v}\pa{\langle v \rangle^\beta\boldsymbol\mu^{-1/2}}$ in Section \ref{sec:Linftytheory}.
\par At last, we work out the Cauchy problem for the full nonlinear equation in Section \ref{sec:fullcauchy}.

\bigskip

\section{Main results} \label{sec:mainresults}

As explained in the introduction, the ultimate goal of this article is a full perturbative Cauchy theory for the multi-species Boltzmann equation $\eqref{multiBE}$. Along the way, we shall also prove the following important results about the linear perturbed operator $\mathbf{L}-v\cdot\nabla_x$.

\bigskip
\begin{theorem}\label{theo:alllinear}
Let the collision kernels $B_{ij}$ satisfy assumptions $(H1) - (H4)$. Then the following holds.
\begin{enumerate}
\item[(i)] The operator $\mathbf{L}$ is a closed self-adjoint operator in $L^2_v\pa{\boldsymbol\mu^{-1/2}}$ and there exists $\lambda_L >0$ such that
$$\forall \mathbf{f} \in L^2_v\pa{\boldsymbol\mu^{-1/2}}, \quad \left\langle \mathbf{f}, \mathbf{L}\pa{\mathbf{f}} \right\rangle_{L^2_v\pa{\boldsymbol\mu^{-1/2}}} \leq -\lambda_L \norm{\mathbf{f} - \pi_\mathbf{L}\pa{\mathbf{f}}}^2_{L^2_v\pa{\langle v \rangle^{\gamma/2}\boldsymbol\mu^{-1/2}}};$$
\item[(ii)] Let $E=L^2_{x,v}\pa{\boldsymbol\mu^{-1/2}}$ or $E=L^\infty_{x,v}\pa{\langle v \rangle^\beta\boldsymbol\mu^{-1/2}}$ with $\beta>3/2$. The linear perturbed operator $\mathbf{G}=\mathbf{L}-v\cdot\nabla_x$ generates a strongly continuous semigroup $S_{\mathbf{G}}(t)$ on $E$ and there exist $C_E$, $\lambda_E>0$ such that
$$\forall t \geq 0, \quad \norm{S_{\mathbf{G}}(t)\pa{\mbox{Id}-\Pi_{\mathbf{G}}}}_{E} \leq C_E e^{-\lambda_E t},$$
\end{enumerate}
where $\pi_\mathbf{L}$ is the orthogonal projection onto $\mbox{Ker}(\mathbf{L})$ in $L^2_v\pa{\boldsymbol\mu^{-1/2}}$ and $\Pi_{\mathbf{G}}$ is the orthogonal projection onto $\mbox{Ker}(\mathbf{G})$ in $L^2_{x,v}\pa{\boldsymbol\mu^{-1/2}}$.
\\ The constants $\lambda_L$, $C_E$ and $\lambda_E$ are explicit and depend on $N$, $E$, the different masses $m_i$ and the collision kernels.
\end{theorem}
\bigskip

We now state the results we obtain for the full nonlinear equation.

\bigskip
\begin{theorem}\label{theo:fullcauchy}
Let the collision kernels $B_{ij}$ satisfy assumptions $(H1) - (H4)$ and let $E=L^1_vL^\infty_x\pa{\langle v \rangle^k}$ with $k>k_0$, where $k_0$ is the minimal integer such that
\begin{equation}\label{Ck}
C_k =\frac{2}{k+2}\frac{1-\cro{\max\limits_{i,j}\frac{\abs{m_i-m_j}}{m_i+m_j}}^{\frac{k+2}{2}}+\cro{1-\pa{\max\limits_{i,j}\frac{\abs{m_i-m_j}}{m_i+m_j}}}^{\frac{k+2}{2}}}{1-\max\limits_{i,j}\frac{\abs{m_i-m_j}}{m_i+m_j}}\max\limits_{i,j}\frac{4\pi b_{ij}^\infty}{l_{b_{ij}}}<1.
\end{equation}
where $l_{b_{ij}}$ and $b_{ij}^\infty$ are angular kernel constants $\eqref{constantsbij}$.
\\Then there exist $\eta_E$, $C_E$ and $\lambda_E >0$ such that for any $\mathbf{F_0} = \boldsymbol\mu + \mathbf{f_0} \geq 0$ satisfying the conservation of mass, momentum and energy $\eqref{conservationlaws}$ with $u_\infty=0$ and $\theta_\infty=1$, if
$$\norm{\mathbf{F_0} - \boldsymbol\mu}\leq \eta_E$$
then there exists a unique solution $\mathbf{F}=\boldsymbol\mu + \mathbf{f}$ in $E$ to the multi-species Boltzmann equation $\eqref{multiBE}$ with initial data $\mathbf{f_0}$. Moreover, $\mathbf{F}$ is non-negative, satisfies the conservation laws and
$$\forall t\geq 0, \quad \norm{\mathbf{F}-\boldsymbol\mu}_E \leq C_E e^{-\lambda_E t}\norm{\mathbf{F_0}-\boldsymbol\mu}_E.$$
The constants are explicit and only depend on $N$, $k$, the different masses $m_i$ and the collision kernels.
\end{theorem}
\bigskip

\begin{remark}\label{rem:mainresults}
We make a few comments about the theorem above.
\begin{enumerate}
\item[(1)] As mentioned in the introduction, $\boldsymbol\mu$ can be replaced by any global equilibrium $\mathbf{M}(c_{i,\infty},u_\infty,\theta_\infty)$. Moreover, as we shall see in Section \ref{sec:fullcauchy}, the natural weight for this theory is the one associated to the conservation of individual masses and total energy: $(1+m_i^{k/2}\abs{v}^k)_{1\leq i \leq N}$. This weight is equivalent to $\langle v \rangle^k$ and we keep the latter weight to work without vector-valued masses outside Subsection \ref{subsubsec:cauchyE}.
\item[(2)] The uniqueness has to be understood in a perturbative regime, that is among the solutions that can be written under the form $\mathbf{F} = \boldsymbol\mu +\mathbf{f}$. We do not give a global uniqueness in $L^1_vL^\infty_x\pa{\langle v \rangle^k}$ (as proved in \cite{GMM} for the single-species Boltzmann equation).
\item[(3)] As a by-product of the proof of uniqueness, we prove that the spectral-gap estimate of Theorem \ref{theo:alllinear} also holds for $E=L^1_vL^\infty_x\pa{\langle v \rangle^k}$ with $k>k_0$.
\item[(4)]In the case of identical masses and hard sphere collision kernels ($b=1$) we recover $C_k = 4/(k+2)$ and thus $k_0=2$ which has recently been obtained in the mono-species case \cite{GMM}.
\end{enumerate}
\end{remark}
\bigskip

\section{Spectral gap for the linear operator in $L^2_v\pa{\boldsymbol\mu^{-1/2}}$}\label{sec:spectralgap}


\subsection{First properties of the linear multi-species Boltzmann operator}\label{subsec:toolboxL}

We start by describing some properties of the linear multi-species Boltzmann operator $\mathbf{L}=\pa{L_i}_{1\leq i \leq N}$. First recall
$$L_i(\mathbf{f}) = \sum_{j=1}^N L_{ij}(f_i,f_j), \quad 1\le i\le N,$$
with
\begin{equation}\label{Lij}
\begin{split}
  L_{ij}(f_i,f_j) &=Q_{ij}\left(\mu_i,f_j\right)
	+ Q_{ij}\left(f_i,\mu_j\right) \nonumber \\
	&= \int_{\R^3\times\S^2}B_{ij}(|v-v_*|, \cos (\theta))\left(\mu_j'^*f_i' + \mu_i'f_j'^* - \mu_j^*f_i - \mu_i f_j^* \right)\:dv_*d\sigma,
\end{split}
\end{equation}
where we have used $\mu_i^{'*} \mu_j' = \mu_i^* \mu_j$ for any
$i$, $j$, which follows from the laws of elastic collisions $\eqref{elasticcollision}$.

\bigskip
Some results about the kernel of $\mathbf{L}$ have recently been obtained \cite{DJMZ} in the case of multi-species having same mass ($m_i=m_j$). Their proofs are directly applicable in the case of different masses, and we therefore refer to their work for detailed proofs. 
\par $\mathbf{L}$ is a self-adjoint operator in $L^{2}_v\left(\boldsymbol\mu^{-1/2}\right)$ with $\langle \mathbf{f}, \mathbf{L}(\mathbf{f})\rangle_{L^2_v\pa{\boldsymbol\mu^{-1/2}}}=0$ if and only if $\mathbf{f}$ belongs to $\mbox{Ker}(\mathbf{L})$. 
$$\mbox{Ker}\left(\mathbf{L}\right) = \mbox{Span}\left\{\boldsymbol\phi_1(v),\dots,\boldsymbol\phi_{N+4}(v)\right\},$$
where $\pa{\boldsymbol\phi_i}_{1\leq i\leq N+4}$ is an orthonormal basis of $\mbox{Ker}\left(\mathbf{L}\right)$ in $L^2_v\left(\boldsymbol\mu^{-1/2}\right)$. More precisely, if we denote $\pi_{\mathbf{L}}$ the orthogonal projection onto $\mbox{Ker}\left(\mathbf{L}\right)$ in $L^2_v\left(\boldsymbol\mu^{-1/2}\right)$:
$$\pi_{\mathbf{L}}(\mathbf{f}) = \sum\limits_{k=1}^{N+4} \pa{\int_{\R^3} \langle\mathbf{f}(v),\boldsymbol\phi_k(v)\rangle_{\boldsymbol\mu^{-1/2}}\:dv} \boldsymbol\phi_k(v),$$
and 
$$\mathbf{e_k} = \pa{\delta_{ik}}_{1\leq i \leq N},$$ we can write
\begin{equation}\label{piL}
\left\{\begin{array}{l} \disp{\boldsymbol\phi_k(v)=\frac{1}{\sqrt{c_{\infty,k}}}\:\mu_k\mathbf{e_k},\quad 1\leq k \leq N}
\\\vspace{2mm}  \disp{ \boldsymbol\phi_k(v) =\frac{v_{k-N}}{\pa{\sum\limits_{i=1}^N m_ic_{\infty,i}}^{1/2}}\:\pa{m_i\mu_i}_{1\leq i \leq N},\quad N+1\leq k \leq N+3.}
\vspace{2mm}\\\vspace{2mm} \disp{\boldsymbol\phi_{N+4}(v)=\frac{1}{\pa{\sum\limits_{i=1}^N c_{\infty,i}}^{1/2}}\:\pa{\frac{\abs{v}^2-3m_i^{-1}}{\sqrt{6}}m_i\mu_i}_{1\leq i \leq N}.}\end{array}\right.
\end{equation}
\par Finally, we denote $\pi_{\mathbf{L}}^\bot = \mbox{Id} - \pi_{\mathbf{L}}$. The projection $\pi_{\mathbf{L}}(\mathbf{f}(t,x,\cdot))(v)$ of $\mathbf{f}(t,x,v)$ onto the kernel of $\mathbf{L}$ is called its fluid part whereas $\pi_{\mathbf{L}}^\bot(\mathbf{f})$ is its microscopic part.

\bigskip
$\mathbf{L}$ can be written under the following form
\begin{equation}\label{LLambdaK}
\mathbf{L} = -\boldsymbol\nu(v) + \mathbf{K},
\end{equation}
where $\boldsymbol\nu = \pa{\nu_i}_{1\leq i\leq N}$ is a multiplicative operator called the collision frequency
\begin{equation}\label{nu}
\nu_i(v) = \sum\limits_{j=1}^N \nu_{ij}(v),
\end{equation}
with
$$\nu_{ij}(v) = C_{ij}^{\Phi}\int_{\R^3\times\mathbb{S}^{2}} b_{ij}\left(\mbox{cos}\:\theta\right)\abs{v-v_*}^\gamma \mu_j(v_*)\:d\sigma dv_*.$$
Each of the $\nu_{ij}$ could be seen as the collision frequency $\nu(v)$ of a single-species Boltzmann kernel with kernel $B_{ij}$. It is well-known (for instance \cite{Ce}\cite{CIP}\cite{Vi2}\cite{GMM}) that under our assumptions: $\nu(v) \sim (1+\abs{v}^\gamma)\sim \langle v \rangle^\gamma$. This means that for all $i$, $j$ there exist $\nu_{ij}^{(0)},\:\nu_{ij}^{(1)} >0$ (they are explicit, see the references above) such that
\begin{equation*}
\forall v \in \R^3,\quad \nu_{ij}^{(0)}\left(1 + |v|^\gamma\right)\leq \nu_{ij}(v)\leq \nu_{ij}^{(1)}\left(1 + |v|^{\gamma}\right),
\end{equation*}
Every constant being strictly positive, the following lemma follows straightforwardly.

\bigskip
\begin{lemma}
There exists a constant $\beta>0$, and for all $i$ in $\br{1,\dots,N}$ there exist $\nu_{i}^{(0)},\:\nu_{i}^{(1)} >0$ such that
\begin{equation}\label{nu0nu1}
\forall v \in \R^3,\quad \nu_{i}^{(0)}\left(1 + |v|^\gamma\right)\leq \nu_{i}(v)\leq \nu_{i}^{(1)}\left(1 + |v|^\gamma\right).
\end{equation}
Thus, we get the following relation between the collision frequencies
\begin{equation}\label{relatenu}
\:\forall v \in \R^3,\quad \nu_{i}(v) \leq \beta \nu_{ii}(v).
\end{equation}
\end{lemma}
\begin{remark}
Estimate \eqref{relatenu}  is a crucial step in the proof of Lemma \ref{lem.specLm}.
In \cite{DJMZ} the additional assumption $B_{ij}\leq C B_{ii}$ for a constant $C>0$ has been used in order to get $\eqref{relatenu}$. 
We want to point out that despite of even having  different masses to handle, we manage to get rid of this assumption.
The prize we have to pay is a slightly more restrictive assumption on the collision kernel $B$ in assumption $(H3)$. \label{rem:spectralgap}
\end{remark}
\bigskip


Next we decompose the operator $\mathbf{L}$ into its mono-species part $\mathbf{L^m}=(L^m_i)_{1\leq i \leq N}$ and its  bi-species part $\mathbf{L^b}=(L^b_i)_{1\leq i \leq N}$ according to 
\begin{align}\label{Lmb}
 \mathbf{L}=\mathbf{L^m}+\mathbf{L^b}, \quad 
 L_i^m(f_i) = L_{ii}(f_i,f_i), \quad L_i^b(f) = \sum_{j\neq i}L_{ij}(f_i,f_j).
\end{align}
Thus $\mathbf{f}$ can be written as
\begin{equation}\label{fdecomp}
  \mathbf{f} = \pi_{\mathbf{L^m}}(\mathbf{f}) + \pi^{\perp}_{\mathbf{L^m}}(\mathbf{f}), 
\end{equation}
where $\pi_{\mathbf{L^m}}$ is the orthogonal projection on $\mbox{Ker}(\mathbf{L^m})$ with respect to $L^2_{v}\pa{\boldsymbol\mu^{-1/2}}$, and $$\pi_{\mathbf{L^m}}^{\perp}:=\left(1-\pi_{\mathbf{L^m}}\right).$$

\bigskip

By employing the standard change of variables, the Dirichlet forms of $\mathbf{L^m}$ and $\mathbf{L^b}$ have the form
\begin{align}
	\left\langle\mathbf{f},\mathbf{L^m}(\mathbf{f})\right\rangle_{L^2_{v}\pa{\boldsymbol\mu^{-1/2}}} &= -\frac14\sum\limits_{i=1}^N\int_{\R^6\times\S^2}B_{ii}\mu_i\mu_i^*\left(A_{ii}\left[f_i\mu_i^{-1},f_i\mu_i^{-1}\right]\right)^2,\\\label{Aij} 
	\left\langle\mathbf{f},\mathbf{L^b}(\mathbf{f})\right\rangle_{L^2_{v}\pa{\boldsymbol\mu^{-1/2}}} &= -\frac14\sum\limits_{i=1}^N\sum\limits_{j\neq i}\int_{\R^6\times\S^2}B_{ij}\mu_i\mu_j^*\left(A_{ij}\left[f_i\mu_i^{-1},f_j\mu_j^{-1}\right]\right)^2,
\end{align}
with the shorthands
  \begin{align}\label{shorthandA}
	A_{ij}\left[f_i\mu_i^{-1},f_j\mu_j^{-1}\right]:= \left(f_i\mu_i^{-1}\right)'+\left(f_j\mu_j^{-1}\right)'^*-\left(f_i\mu_i^{-1}\right)-\left(f_j\mu_j^{-1}\right)^*.
\end{align}
\bigskip

Since $\mathbf{L^m}$ describes a multi-species operator when all the cross-interactions are null, 
\begin{align}
&\pi_{\mathbf{L^m}}(\mathbf{f})_i=m_i\mu_i(v)(a_i(t,x)+u_i(t,x)\cdot v+e_i(t,x)|v|^2), \quad 1\leq i \leq N, \label{kerLm}
\end{align}
where $a_i \in \R, u_i \in \R^3$ and $e_i \in \R$ are the coordinates of $\pi_{\mathbf{L^m}}(\mathbf{f})$ with respect to a $5N$-dimensional basis, while
\begin{align}
&\pi_{\mathbf{L}}(\mathbf{f})_i=m_i\mu_i(v)(a_i(t,x)+u(t,x)\cdot v+e(t,x)|v|^2) \quad 1\leq i \leq N, \label{kerL} 
\end{align}
where $a_i \in \R, u \in \R^3$ and $e \in \R$ are the coordinates of $\pi_{\mathbf{L}}(\mathbf{f})$ with respect to an $(N+4)$-dimensional basis. 
\bigskip

\par Finally, since
\begin{align}\label{max1}
&\int_{\R^3}\mu_i\:dv = c_{i}, \quad \int_{\R^3}\mu_i |v|^2 \:dv = 3c_{i}m_i^{-1},\quad \int_{\R^3}\mu_i|v|^4 \:dv = 15c_{i}m_i^{-2},\
\end{align}
the following moment identities hold for $a_i, u_i, e_i$ defined in \eqref{kerLm} 
\begin{align}\label{max2}
&\int_{\R^3}f_i \:dv = c_{i}(m_ia_i+3e_i), \nonumber \\&\int_{\R^3}f_i v\:dv = c_{i}u_i, \\& \int_{\R^3}f_i|v|^2 \:dv = c_{i}(3a_i+15e_im_i^{-1}).\nonumber
\end{align}
\bigskip

\subsection{Explicit spectral gap}\label{subsec:spectralgap}

This subsection is devoted to the proof of the following constructive spectral gap estimate for the multi-species linear operator $\mathbf{L}$ with different masses. 
\bigskip
\begin{theorem}\label{theo:spectralgapL}
Let the collision kernels $B_{ij}$ satisfy assumptions (H1)-(H4).
Then there exists an explicit constant $\lambda_L>0$ such that
\begin{equation*}
\left\langle \mathbf{f},\mathbf{L(f)}\right\rangle_{L^2_v\pa{\boldsymbol\mu^{-1/2}}} \leq -\lambda_L\|\mathbf{f}-\pi_{\mathbf{L}}(\mathbf{f})\|_{L^2_v\pa{\langle v \rangle^{\gamma/2}\boldsymbol\mu^{-1/2}}}^2\quad \forall \mathbf{f} \in \mbox{Dom}(\mathbf{L}),
\end{equation*}
where $\lambda_L$ depends on the properties of the collision kernel, the number of species $N$ and the different masses.
\end{theorem}
\bigskip

The next two lemmas are crucial for the proof of Theorem \ref{theo:spectralgapL}, generalizing the strategy of \cite{DJMZ} to the case of different masses.
The key idea is to decompose $\mathbf{L}$ into $\mathbf{L}=\mathbf{L^m}+\mathbf{L^b}$ (see \eqref{Lmb}), and to derive separately a spectral-gap estimate for the mono-species part $\mathbf{L^m}$ on its domain $\mbox{Dom}(\mathbf{L^m})$ (see Lemma \ref{lem.specLm}), and a spectral-gap type estimate for the bi-species part $\mathbf{L^b}$ on $\mbox{Ker}(\mathbf{L^m})$ (see Lemma \ref{sec.remain}) measured in terms of the following functional 
\begin{align*}
\mathcal{E}:\: \mbox{Ker}(\mathbf{L^m})\to \R^+, \quad
\mathcal{E}(\mathbf{f}):=\sum\limits_{i,j=1}^N \left(\left|u^{(f)}_i - u^{(f)}_j\right|^2 + \left(e^{(f)}_i-e^{(f)}_j\right)^2 \right),
\end{align*}
where for a fixed $\mathbf{f} \in \mbox{Ker}(\mathbf{L^m})$, $u^{(f)}_i$ and $e^{(f)}_i$  describe the coordinates of the $i^{th}$ component of $\mathbf{f}$ with respect to the basis defined in \eqref{kerLm}.
To lighten computations, we introduce the following Hilbert space $\mathcal H:=L^2_v\left(\boldsymbol{\nu}^{1/2}\boldsymbol\mu^{-1/2}\right)$, which is equivalent to  $L^2_v\pa{\langle v \rangle^{\gamma/2}\boldsymbol\mu^{-1/2}}$:
\begin{equation}
\begin{split} \label{H}
\mathcal H &= \bigg\{f\in L^2_v(\boldsymbol{\mu^{-1/2}}):\|\mathbf{f}\|_\mathcal H^2 = \sum_{i=1}^N\int_{\R^3}f_i^2\nu_i \mu_i^{-1} \:dv < \infty\bigg\}.
\end{split}
\end{equation}
\bigskip

\begin{lemma}\label{lem.specLm}
For all $\mathbf{f}$ in $\mbox{Dom}(\mathbf{L^m})$ there exists an explicit constant $C_1>0$, such that
$$\left\langle\mathbf{f},\mathbf{L^m(f)}\right\rangle_{L^2_v\pa{\boldsymbol\mu^{-1/2}}} \leq - C_1 \|\mathbf{f}-\pi_{\mathbf{L^m}}(\mathbf{f})\|_{L^2_v\pa{\langle v \rangle^{\gamma/2}\boldsymbol\mu^{-1/2}}}^2,$$
where $C_1$ depends on the properties of the collision kernel, the number of species $N$ and the different masses.
\end{lemma}
\begin{proof}
By \cite[Theorem 1.1 and Remark 1 below it]{Mo1} together with the shorthand introduced in \eqref{shorthandA}, 
\begin{equation*}
  \frac14\int_{\R^6\times\S^2}B_{ii} \left(A_{ii}\left[f_i\mu_i^{-1}, f_i\mu_i^{-1}\right]\right)^2 \mu_i \mu_i^*
	\:dvdv_*d\sigma \ge 
	\lambda_m c_{\infty,i}\int_{\R^3}(f_i-\pi_{\mathbf{L^m}}(\mathbf{f})_i)^2\nu_{ii}\mu_i^{-1}\:dv,
\end{equation*}
where $\lambda_m>0$ depends on the properties of the collision kernel, the number of species $N$ and the different masses.
Summing this estimate over $i=1,\ldots,N$ and employing \eqref{Aij} yields
\begin{equation}\label{fLmf}
  -\left\langle\mathbf{f},\mathbf{L^m(f)}\right\rangle_{L^2_v\pa{\boldsymbol\mu^{-1/2}}} \ge \lambda^m \sum_{i=1}^N c_{\infty,i}\int_{\R^3}
	(f_i-\pi_{\mathbf{L^m}}(f_i))^2\frac{\nu_{ii}}{\mu_i}\:dv.
\end{equation}
Now we can estimate $\nu_{ii}$ in terms of $\nu_i$ by using $\eqref{relatenu}$, and plugging this bound into \eqref{fLmf} together with the fact that $\mathcal H$ is equivalent to  $L^2_v\pa{\langle v \rangle^{\gamma/2}\boldsymbol\mu^{-1/2}}$ finishes the proof.
\end{proof}
\bigskip

\begin{lemma}\label{sec.remain}
For all $\mathbf{f}$ in $\mbox{Ker}(\mathbf{L^m})\cap\mbox{Dom}(\mathbf{\mathbf{L^b}})$ there exists an explicit $C_2>0$ such that
\begin{align*}
 \left\langle \mathbf{f}, \mathbf{L^b(f)} \right\rangle_{L^2_v\pa{\boldsymbol\mu^{-1/2}}} \leq -C_2 \:\mathcal{E}(\mathbf{f}),  
\end{align*}
with the functional  $\mathcal{E}$ defined by
\begin{align}\label{energynorm}
\mathcal{E}:\: \mbox{Ker}(\mathbf{L^m})\to \R^+, \quad
\mathcal{E}(\mathbf{f}):=\sum\limits_{i,j=1}^N \left(\left|u^{(f)}_i - u^{(f)}_j\right|^2 + \left(e^{(f)}_i-e^{(f)}_j\right)^2 \right),
\end{align}
where for fixed $\mathbf{f} \in \mbox{Ker}(\mathbf{L^m})$ it holds that $u^{(f)}_i$, $e^{(f)}_i$ describe the coordinates of the $i^{th}$ component of $\mathbf{f}$ with respect to the basis defined in \eqref{kerLm}, and $C_2>0$ is defined in \eqref{D}.
\end{lemma}
\bigskip

\begin{remark}
 Note that for $\mathbf{f}$ in $\mbox{Ker}(\mathbf{L^m})$ it holds that
 $$\mathcal{E}(\mathbf{f}) = 0  \quad \Leftrightarrow \quad \mathbf{f} \in \mbox{Ker}(\mathbf{L^b}),$$
 since $\mbox{Ker}(\mathbf{L})=\mbox{Ker}(\mathbf{L^m})\cap \mbox{Ker}(\mathbf{L^b})$. This fact together with a multi-species version of the H-theorem show that the left-hand side of the estimate in Lemma \ref{sec.remain} is null if and only  if the right-hand side is null.
 \end{remark}

\begin{proof}
Let $\mathbf{f} \in \mbox{Ker}(\mathbf{L^m})\cap\mbox{Dom}(\mathbf{\mathbf{L^b}})$. Writing $\mathbf{f}$ in the form \eqref{kerLm} and applying the microscopic conservation laws \eqref{elasticcollision} yields  
\begin{align*}
A_{ij}[f_i \mu_i^{-1}, f_j \mu_j^{-1}] = m_i(u_i-u_j)\cdot(v'-v) + m_i(e_i-e_j)(|v'|^2-|v|^2),
\end{align*}
and thus
\begin{align*}
  -\langle\mathbf{f},&\mathbf{L^b}(\mathbf{f})\rangle_{L^2_v\pa{\boldsymbol\mu^{-1/2}}}\\
		&= \frac14\sum_{\overset{i,j=1}{j\neq i}}^Nm_i^2\int_{\R^6\times\S^2}B_{ij}
	\big[(u_i-u_j)\cdot(v'-v) + (e_i-e_j)(|v'|^2-|v|^2)\big]^2 \mu_i\mu_j^*.
\end{align*}
Using the symmetry of $B_{ij}$ and of $\mu_i\mu_j^*$ 
together with the  oddity of the function $G(v,v_*,\sigma)=B_{ij}(u_i-u_j)\cdot(v'-v)(|v'|^2-|v|^2)$
with respect to $(v,v_*,\sigma)$ yields that the mixed term in the square of the integral above vanishes. Thus we obtain
\begin{align}\label{fLf3}
  -\langle \mathbf{f}, &\mathbf{L^b}(\mathbf{f})\rangle_{L^2_v\pa{\boldsymbol\mu^{-1/2}}}= \frac14\sum_{\overset{i,j=1}{j\neq i}}^N m_i^2\int_{\R^6\times\S^2}B_{ij}\\\nonumber
	&\times\big(|(u_i-u_j)\cdot(v'-v)|^2 + (e_i-e_j)^2(|v'|^2-|v|^2)^2\big)\mu_i\mu_j^*\:dvdv_*d\sigma.
\end{align}

We claim that the following holds
$$
  \int_{\R^6\times\S^2}B_{ij}((u_i-u_j)\cdot(v'-v))^2 \mu_i\mu_j^*\:dvdv_*d\sigma
	= \frac{|u_i-u_j|^2}{3}\int_{\R^6\times\S^2}B_{ij}|v-v'|^2\mu_i\mu_j^*\:dvdv_*d\sigma.
$$
To prove this identity, we write $u_{i,k}$ and $v_k$ for the $k$th component of
the vectors $u_i$ and $v$, respectively. The change of variables $(v_k,v_k^*,\sigma_k)
\mapsto -(v_k,v_k^*,\sigma_k)$ for fixed $k$ leaves $B_{ij}$, $\mu_i$, and $\mu_j^*$
unchanged but $v_k'\mapsto -v_k'$, such that
$$
  \int_{\R^6\times\S^2}B_{ij}v_k'v_\ell \mu_i\mu_j^*\:dvdv_*d\sigma = 0
	\quad\mbox{for }\ell\neq k.
$$
Moreover,
$$
  \int_{\R^6\times\S^2}B_{ij}v_kv_\ell \mu_i\mu_j^*\:dvdv_*d\sigma = 0
	\quad\mbox{for }\ell\neq k,
$$
since the integrand is odd. Thus,
\begin{align*}
  \int_{\R^6\times\S^2} & B_{ij}((u_i-u_j)\cdot(v'-v))^2 \mu_i\mu_j^*\:dvdv_*d\sigma \\
	&= \sum_{k,\ell=1}^3 (u_{i,k}-u_{j,k})(u_{i,\ell}-u_{j,\ell})
	\int_{\R^6\times\S^2}B_{ij}(v_k'-v_k)(v_\ell'-v_\ell)\mu_i\mu_j^*\:dvdv_*d\sigma \\
	&= \sum_{k=1}^3 (u_{i,k}-u_{j,k})^2\int_{\R^6\times\S^2}B_{ij}
	(v_k-v_k')^2 \mu_i\mu_j^*\:dvdv_*d\sigma.
\end{align*}
Since the integral is independent of $k$, we get
\begin{align*}
  \int_{\R^6\times\S^2} & B_{ij}((u_i-u_j)\cdot(v'-v))^2 \mu_i\mu_j^*\:dvdv_*d\sigma \\
	&= \frac13\sum_{k=1}^3 (u_{i,k}-u_{j,k})^2\int_{\R^6\times\S^2}B_{ij}
	|v-v'|^2 \mu_i\mu_j^*\:dvdv_*d\sigma,
\end{align*}
which proves the claim.

This implies that for all $\mathbf{f}$ in $\mbox{Ker}(\mathbf{L^m})\cap\mbox{Dom}(\mathbf{\mathbf{L^b}})$ it holds that
$$
  \langle\mathbf{f},\mathbf{L^b}(\mathbf{f})\rangle_{L^2_v\pa{\boldsymbol\mu^{-1/2}}}
	\leq -C_2 \: \mathcal{E}(\mathbf{f}),
$$
where $\mathcal{E}(\cdot)$ is defined in \eqref{energynorm} and
\begin{equation}\label{D}
  C_2 = \frac{1}{4}\min_{1\le i,j\le n}\int_{\R^6\times\S^2}m_i^2B_{ij}
	\min\left\{\frac13|v-v'|^2,(|v'|^2-|v|^2)^2\right\}\mu_i\mu_j^* \:dvdv_*d\sigma.
\end{equation}
The last part is to prove that $C_2>0$. For this we note that the integrand of \eqref{D} vanishes
if and only if $|v'|=|v|$. However, the set
$$
  X = \{(v,v_*,\sigma)\in\R^3\times\R^3\times\S^2: |v'|=|v|\}
$$
is closed since it is the pre-image of $\{0\}$ of the function 
$H(v,v_*,\sigma)=|v'|^2-|v|^2$ which is continuous. Now $X^c$ is open and nonempty
and thus has positive Lebesgue measure, and since the integrand in \eqref{D} is
positive on $X^c$, we get that $C_2>0$, which finishes the proof.
\end{proof}
\bigskip

\begin{proof}[Proof of Theorem \ref{theo:spectralgapL}]
The proof will be performed in $4$ steps. To lighten notation, we will use the following shorthands for $\mathbf{f} \in \mbox{Dom}(\mathbf{L})$:
\begin{align}\label{shorthands}
\mathbf{f^\parallel}=\pi_{\mathbf{L^m}}(\mathbf{f}), \quad \mathbf{f^{\perp}}=\mathbf{f}-\mathbf{f^{\parallel}}, \quad 
h_i^\parallel=\mu_i^{-1}f_i^\parallel, \quad h_i^{\perp}=\mu_i^{-1}h_i^\perp.
\end{align}
\bigskip

\textbf{Step $1:$ Absorption of the orthogonal part.}\label{sec.ortho}

\par The nonnegativity of  $-\langle\mathbf{f},\mathbf{L^b}(\mathbf{f})\rangle_{L^2_v\pa{\boldsymbol\mu^{-1/2}}}\ge 0$ and Lemma \ref{lem.specLm} imply that
\begin{eqnarray}
  -\langle(\mathbf{f},\mathbf{L(f)}\rangle_{L^2_v\pa{\boldsymbol\mu^{-1/2}}} \ge C_1\|\mathbf{f}-\mathbf{f^\parallel}\|_\mathcal H^2 - \eta\langle\mathbf{f},\mathbf{L^b}(\mathbf{f})\rangle_{L^2_v\pa{\boldsymbol\mu^{-1/2}}},\label{fLf}
\end{eqnarray}
where $\eta \in (0,1]$ and $C_1>0$ was defined in Lemma \ref{lem.specLm}. 
Now it holds that
\begin{align*}
  A_{ij}[h_i,h_j]^2 
	&\ge \frac12A_{ij}[h_i^\parallel,h_j^\parallel]^2
	- A_{ij}[h_i^\perp,h_j^\perp]^2,
\end{align*}
and plugging this into \eqref{Aij} and \eqref{fLf} implies
\begin{align}
  -\langle\mathbf{f},\mathbf{L(f)})\rangle_{L^2_v\pa{\boldsymbol\mu^{-1/2}}} &\ge C_1\|f^\perp\|_\mathcal H^2
	+ \frac{\eta}{8}\sum_{i=1}^N\sum_{j\neq i}\int_{\R^6\times\S^2}
	B_{ij}A_{ij}[h_i^\parallel,h_j^\parallel]^2\mu_i\mu_j^*\:dvdv_*d\sigma \nonumber \\
	&\phantom{xx}{}- \frac{\eta}{4}\sum_{i=1}^N\sum_{j\neq i}^N\int_{\R^6\times\S^2}
	B_{ij}A_{ij}[h_i^\perp,h_j^\perp]^2 \mu_i\mu_j^*\:dvdv_*d\sigma. \label{fLf2}
\end{align}

Now we prove that (up to a small factor) the last term on the right-hand side can be estimated from below
by $\|\mathbf{f^\perp}\|_\mathcal H^2$. For this we perform the standard change of variables $(v,v_*) \to (v_*,v)$ together with $i \leftrightarrow j$ and $(v,v_*)\to(v',v'_*)$, and by using the 
identity $\mu_i\mu_j^*=\mu_i'\mu_j'^*$ we obtain
\begin{align*}
  &\sum_{i=1}^N\sum_{j\neq i}\int_{\R^6\times\S^2}  B_{ij}A_{ij}[h_i^\perp,h_j^\perp]^2 \mu_i\mu_j^*\:dvdv_*d\sigma \\
	&\quad\quad\quad\le 4\sum_{i=1}^N\sum_{j\neq i}\int_{\R^6\times\S^2}B_{ij}\big(((h_i^\perp)')^2 + ((h_j^\perp)'^*)^2
	+ (h_i^\perp)^2 + ((h_j^\perp)^*)^2\big)\mu_i\mu_j^* \:dvdv_*d\sigma \\
	&\quad\quad\quad\le 16\sum_{i=1}^N\sum_{j\neq i}\int_{\R^6\times\S^2}B_{ij}(h_i^\perp)^2 \mu_i\mu_j^* \:dvdv_*d\sigma.
	\end{align*}
Taking into account the definition \eqref{nu} of $\nu_i$, we get for the last term on the right-hand side of \eqref{fLf2} 
\begin{align*}
  - &\frac{\eta}{4} \sum_{i=1}^N\sum_{j\neq i} \int_{\R^6\times\S^2}
	B_{ij}A_{ij}[h_i^\perp,h_j^\perp]^2 \mu_i\mu_j^*\:dvdv_*d\sigma \\
	&\quad\quad\quad\ge -4\eta\sum_{i,j=1}^N\int_{\R^6\times\S^2} B_{ij}
	(f_i^\perp)^2 \mu_j^*\mu_i^{-1}\:dvdv_*d\sigma \\
	&\quad\quad\quad\ge -4\eta\sum_{i=1}^N\int_{\R^3}(f_i^\perp)^2\nu_i\mu_i^{-1} \:dv
	= -4\eta\|f^\perp\|_\mathcal H^2.
\end{align*}
Finally \eqref{fLf2} yields
\begin{equation*}
\begin{split}
  -\langle\mathbf{f},\mathbf{L(f)}\rangle_{L^2_v\pa{\boldsymbol\mu^{-1/2}}} \ge& \left(C_1-4\eta\right)\left\|\mathbf{f}-\mathbf{f^\parallel}\right\|_\mathcal H^2
	\\ &+ \frac{\eta}{8}\sum_{i=1}^N\sum_{j\neq i}\int_{\R^6\times\S^2}
	B_{ij}A_{ij}\left[h_i^\parallel,h_j^\parallel\right]^2\mu_i\mu_j^*\:dvdv_*d\sigma.
\end{split}
\end{equation*}
Thus 
\begin{align}\label{lem.ortho}
  \langle \mathbf{f},\mathbf{L(f)}\rangle_{L^2_v\pa{\boldsymbol\mu^{-1/2}}} \leq -(C_1-4\eta)\left\|\mathbf{f}-\mathbf{f^\parallel}\right\|_\mathcal H^2
	+\frac{\eta}{2}\langle \mathbf{f^\parallel},\mathbf{L^b(f^\parallel})\rangle_{L^2_v\pa{\boldsymbol\mu^{-1/2}}},
\end{align}
where $0<\eta\leq\min\{1,C_1/8\}$.
\bigskip

\textbf{Step $2:$ Estimate for for the remaining part.}
Due to Lemma \ref{sec.remain} there exists an explicit $C_2>0$ such that
\begin{align*}
\left\langle\mathbf{f^\parallel},\mathbf{L^b}(\mathbf{f^\parallel})\right\rangle_{L^2_v\pa{\boldsymbol\mu^{-1/2}}} \leq - C_2 \: \mathcal{E}\left(\mathbf{f^\parallel}\right).
\end{align*}

\textbf{Step $3:$ Estimate for the momentum and energy differences.}\label{sec.diff}

\par We need to find a relation between $\mathcal{E}(\mathbf{f^\parallel}), \left\|\mathbf{f}-\mathbf{f^{\parallel}}\right\|$ and $\left\|\mathbf{f}-\pi_{\mathbf{L}}(\mathbf{f})\right\|$ respectively. To this end, we decompose $\mathbf{f}=\mathbf{f^\parallel}+\mathbf{f^\perp}$ recalling that $\mathbf{f^\parallel}=\pi_{\mathbf{L^m}}(\mathbf{f})$ and
$\mathbf{f^\perp}=\mathbf{f}-\mathbf{f^\parallel}$. Using an arbitrary orthonormal basis $\left(\boldsymbol{\psi}_k\right)_{1\leq k\leq 5N}$ of $\mbox{Ker}(\mathbf{L}^m)$ in $L^2_v\pa{\boldsymbol\mu^{-1/2}}$, we first show that 
\begin{equation}\label{aux2}
  \|\mathbf{f}-\pi_{\mathbf{L}}(\mathbf{f})\|_\mathcal H^2 \le 2\|\mathbf{f}^\perp\|_\mathcal H^2 
	+ k_0\pa{\|\mathbf{f}^\parallel\|_{L^2_v\pa{\boldsymbol\mu^{-1/2}}}^2 - \|\pi_{\mathbf{L}}(\mathbf{f})\|_{L^2_v\pa{\boldsymbol\mu^{-1/2}}}^2},
\end{equation}
where $k_0=10N\max_{1\le k,\ell\le 5N}|\langle\boldsymbol{\psi_k},\boldsymbol{\psi_\ell}\rangle_\mathcal H|$.

To this end, we start with
\begin{equation}\label{aux1}
  \|\mathbf{f}-\pi_{\mathbf{L}}(\mathbf{f})\|_\mathcal H^2 \le 2\big(\|\mathbf{f}^\perp\|_\mathcal H^2 
	+ \|\mathbf{f}^\parallel-\pi_{\mathbf{L}}(\mathbf{f})\|_\mathcal H^2\big).
\end{equation}
Denoting the last term by $\mathbf{g}:=\mathbf{f}^\parallel-\pi_{\mathbf{L}}(\mathbf{f})
\in\mbox{Ker}(\mathbf{L}^m)$ (note that $\mbox{Ker}(\mathbf{L})\subset\mbox{Ker}(\mathbf{L^m})$) 
and using Young's inequality implies
\begin{align*}
  \|\mathbf{g}\|_\mathcal H^2 &= \sum_{i=1}^N\int_{\R^3}\left|\sum_{k=1}^{5N}\langle\mathbf{g},\boldsymbol{\psi_k}\rangle_{L^2_v\pa{\boldsymbol\mu^{-1/2}}}
	\psi_{k,i}\right|^2\nu_i(v)\:dv\\
	&= \sum_{k,\ell=1}^{5N}\langle\mathbf{g},\boldsymbol{\psi}_k\rangle_{L^2_v\pa{\boldsymbol\mu^{-1/2}}}\langle\mathbf{g},\boldsymbol{\psi}_\ell\rangle_{L^2_v\pa{\boldsymbol\mu^{-1/2}}}
	\langle\boldsymbol{\psi}_k,\boldsymbol{\psi}_\ell\rangle_\mathcal H \\
	&\le \frac12\max_{1\le k,\ell\le 5N}\abs{\langle\boldsymbol{\psi_k},\boldsymbol{\psi}_\ell\rangle_\mathcal H}
	\sum_{k,\ell=1}^{5N}\pa{\langle\mathbf{g},\boldsymbol{\psi_k}\rangle^2_{L^2_v\pa{\boldsymbol\mu^{-1/2}}}+\langle\mathbf{g},\boldsymbol{\psi}_\ell\rangle^2_{L^2_v\pa{\boldsymbol\mu^{-1/2}}}} \\
	&= 5N\max_{1\le k,\ell\le 5N}\abs{\langle\boldsymbol{\psi_k},\boldsymbol{\psi_\ell}\rangle_\mathcal H}\|\mathbf{g}\|^2_{L^2_v\pa{\boldsymbol\mu^{-1/2}}}.
\end{align*}
Thus, \eqref{aux1} implies
$$
  \|\mathbf{f}-\pi_{\mathbf{L}}(\mathbf{f})\|_\mathcal H^2 \le 2\|\mathbf{f}^\perp\|_\mathcal H^2 
	+ 10N\max_{1\le k,\ell\le 5N}\abs{\langle\boldsymbol{\psi_k},\boldsymbol{\psi_\ell}\rangle_\mathcal H}
	\|\mathbf{f}^\parallel-\pi_{\mathbf{L}}(\mathbf{f})\|_{L^2_v\pa{\boldsymbol\mu^{-1/2}}}^2.
$$
Now $\mbox{Ker}(\mathbf{L})\subset\mbox{Ker}(\mathbf{L^m})$ implies $\pi_{\mathbf{L^m}}\pi_{\mathbf{L}}=\pi_{\mathbf{L}}$, thus
\begin{align*}
  \|\mathbf{f}^\parallel-\pi_{\mathbf{L}}(\mathbf{f})\|^2_{L^2_v\pa{\boldsymbol\mu^{-1/2}}} = \|\mathbf{f}^\parallel\|_{L^2_v\pa{\boldsymbol\mu^{-1/2}}}^2 - \|\pi_{\mathbf{L}}(\mathbf{f})\|_{L^2_v\pa{\boldsymbol\mu^{-1/2}}}^2,
\end{align*}
which indeed yields \eqref{aux2}.
\bigskip

Now the moment identities \eqref{max1} and \eqref{max2} yield
\begin{align*}
  \|\mathbf{f}^\parallel\|_{L^2_v\pa{\boldsymbol\mu^{-1/2}}}^2 
	&= \sum_{i=1}^N c_{\infty,i}(m_i^2a_i^2 + m_i|u_i|^2 +15e_i^2 + 6m_ia_ie_i),
\end{align*}
and
\begin{equation*}
\begin{split}
  \|\pi_{\mathbf{L}}(\mathbf{f})\|_{L^2_v\pa{\boldsymbol\mu^{-1/2}}}^2
	&= \sum_{j=1}^{N+4}\langle\mathbf{f},\boldsymbol{\phi_{j}}\rangle_{L^2_v\pa{\boldsymbol\mu^{-1/2}}}^2 
	\\&= \sum_{i=1}^N c_{\infty,i}(m_ia_i+3e_i)^2 
	+ \rho_\infty\left|\sum_{i=1}^N\frac{\rho_{\infty,i}}{\rho_\infty}u_i\right|^2
	+ 6c_\infty\left(\sum_{i=1}^N\frac{c_{\infty,i}}{c_{\infty}}e_i\right)^2,
\end{split}
\end{equation*}
where $\left(\boldsymbol{\phi}_j\right)_{1\leq j\leq N+4}$ is the orthonormal basis of $\mbox{Ker}(\mathbf{L})$ in $L^2_v(\boldsymbol{\mu}^{-1/2})$ introduced in \eqref{piL}.
\bigskip

Inserting these expressions into \eqref{aux2}, we conclude that
\begin{align*}
  \|\mathbf{f}-\pi_{\mathbf{L}}(\mathbf{f})\|_\mathcal H^2
	&\le 2\|\mathbf{f}-\mathbf{f^{\parallel}}\|_\mathcal H^2 
	+ k_0\rho_\infty\left(\sum_{i=1}^N\frac{\rho_{\infty,i}}{\rho_\infty}|u_i|^2 
	- \left|\sum_{i=1}^N\frac{\rho_{\infty,i}}{\rho_\infty}
	u_i\right|^2\right) \\
	&\phantom{xx}{}+ 6k_0c_\infty\left(\sum_{i=1}^N
	\frac{c_{\infty,i}}{c_\infty} e_i^2
	- \left(\sum_{i=1}^N\frac{c_{\infty,i}}{c_\infty}e_i\right)^2\right).
\end{align*}
The next step is to prove that the following estimates hold:
\begin{align}
  I_1:=\sum_{i=1}^N \frac{\rho_{\infty,i}}{\rho_\infty}|u_i|^2 
	- \left|\sum_{i=1}^N\frac{\rho_{\infty,i}}{\rho_\infty}u_i\right|^2
	&\le \sum_{i,j=1}^N|u_i-u_j|^2, \label{aux.ineq1} \\
  I_2:=\sum_{i=1}^N\frac{c_{\infty,i}}{c_\infty} e_i^2 
	- \left(\sum_{i=1}^N\frac{c_{\infty,i}}{c_\infty}e_i\right)^2
	&\le \sum_{i,j=1}^N(e_i-e_j)^2. \label{aux.ineq2}
\end{align}
Note that we only need to prove the estimate for $I_1$, since the arguments for $I_2$ are exactly the same.
In order to handle the expression $I_1$, we define for $\mathbf{u}=(u_i)_{1\leq i \leq N}$ and $\mathbf{v}=(v_i)_{1\leq i \leq N} \in \R^{3N}$
the following scalar product on $\R^{3N}$ with corresponding norm
$$
  \langle\mathbf{u},\mathbf{v}\rangle_\rho = \sum_{i=1}^N\frac{\rho_{\infty,i}}{\rho_\infty}u_i\cdot v_i, \quad \|\mathbf{u}\|_\rho = \langle\mathbf{u},\mathbf{u}\rangle_\rho^{1/2},
$$
where $u_i\cdot v_i$ denotes the standard Euclidean scalar product in $\R^3$. Note that the vector
${\mathbf 1}=(1,\ldots,1)\in\R^{3N}$ satisfies $\|{\mathbf 1}\|_\rho=1$. Now we use the following elementary identity
$$
  \|\mathbf{u}\|_\rho^2 - \langle\mathbf{u},{\mathbf 1}\rangle_\rho^2 = \|\mathbf{u}-\langle\mathbf{u},{\mathbf 1}\rangle_\rho{\mathbf 1}\|_\rho^2,
$$
which can be written as
$$
  I_1 = \sum_{i=1}^N \frac{\rho_{\infty,i}}{\rho_\infty}|u_i|^2 
	- \left|\sum_{i=1}^N\frac{\rho_{\infty,i}}{\rho_\infty}u_i\right|^2
	= \sum_{i=1}^N\frac{\rho_{\infty,i}}{\rho_\infty}\left|u_i 
	- \sum_{j=1}^N\frac{\rho_{\infty,j}}{\rho_\infty}u_j
	\right|^2.
$$
By using the fact that $\sum_{j=1}^N\rho_{\infty,j}=\rho_\infty$, we get
\begin{align*}
  I_1 &= \sum_{i=1}^N\frac{\rho_{\infty,i}}{\rho_\infty}
	\left|\left(1-\frac{\rho_{\infty,i}}{\rho_\infty}\right)u_i
	- \sum_{j\neq i}\frac{\rho_{\infty,j}}{\rho_\infty}u_j\right|^2
	= \sum_{i=1}^N\frac{\rho_{\infty,i}}{\rho_\infty}
	\left|\sum_{j\neq i}\frac{\rho_{\infty,j}}{\rho_\infty}
	(u_i-u_j)\right|^2. 
\end{align*}
Inserting the additional factor $(\sum\limits_{j \neq i}\rho_{\infty,k}/\rho_{\infty})^2$ leads to a convex combination of $\lambda_j$ such that  $\sum_{j\neq i}\lambda_j=1$:
\begin{align*}
	I_1&= \sum_{i=1}^N\frac{\rho_{\infty,i}}{\rho_\infty}
	\left(\sum_{k\neq i}\frac{\rho_{\infty,k}}{\rho_\infty}\right)^2
  \left|\frac{\sum_{j\neq i}(\rho_{\infty,j}/\rho_\infty)
	(u_i-u_j)}{\sum_{k\neq i}\rho_{\infty,k}/\rho_\infty}  \right|^2 \\
  &= \sum_{i=1}^N\frac{\rho_{\infty,i}}{\rho_\infty}
	\left(\sum_{k\neq i}\frac{\rho_{\infty,k}}{\rho_\infty}\right)^2
	\left|\sum_{j\neq i}\lambda_j(u_i-u_j)\right|^2,
\end{align*}
where $\lambda_j=(\rho_{\infty,j}/\rho_\infty)(\sum_{k\neq i}
(\rho_{\infty,k}/\rho_\infty))^{-1}$. Thus we can apply Jensen's inequality to this convex
combination and obtain
\begin{align*}
  I_1 &=\sum_{i=1}^N\frac{\rho_{\infty,i}}{\rho_\infty}
	\left(\sum_{k\neq i}\frac{\rho_{\infty,k}}{\rho_\infty}\right)^2
	\left|\sum_{j\neq i}\lambda_j(u_i-u_j)\right|^2
	\le \sum_{i=1}^N\frac{\rho_{\infty,i}}{\rho_\infty}
	\left(\sum_{k\neq i}\frac{\rho_{\infty,k}}{\rho_\infty}\right)^2
	\sum_{j\neq i}\lambda_j|u_i-u_j|^2 .
\end{align*}
Finally, we can estimate the right-hand side easily by using the definition of the $\lambda_j$ and that $\rho_{\infty,j}\le\rho_\infty$ to obtain
\begin{align*}
  I_1 &\leq \sum_{i=1}^N\frac{\rho_{\infty,i}}{\rho_\infty}
	\left(1-\frac{\rho_{\infty,i}}{\rho_\infty}\right)
	\sum_{j\neq i}\frac{\rho_{\infty,j}}{\rho_\infty}|u_i-u_j|^2
	\le \sum_{i,j=1}^N|u_i-u_j|^2.
\end{align*}
For $I_2$ in \eqref{aux.ineq2} exactly the same calculations hold.
This implies that
\begin{align}\label{lem.diff}
- \mathcal{E}(\mathbf{f^{\parallel}}) \leq -C_3 \big(\|\mathbf{f}-\pi_{\mathbf{L}}(\mathbf{f})\|_\mathcal H^2 - 2\|\mathbf{f}-\mathbf{f^{\parallel}}\|_\mathcal H^2\big),
\end{align}
where $C_3=1/C_k>0$, with
\begin{equation*}
  C_k = 10N \max_{1\le k,\ell\le 5N}\left|\sum_{i=1}^N\int_{\R^3}\psi_{k,i}\psi_{\ell,i}\nu_i \:dv\right|\max\left\{\rho_{\infty},6c_{\infty}\right\},
\end{equation*}
recalling that $\left(\boldsymbol{\psi}_k\right)_{1\leq k\leq 5N}$ is an arbitrary orthonormal basis of $\mbox{Ker}(\mathbf{L^m})$ in $L^2_v\pa{\boldsymbol\mu^{-1/2}}$.
\bigskip

\textbf{Step $4$: End of the proof.}

\par Putting together \eqref{lem.ortho}, Lemma \ref{sec.remain}, and \eqref{lem.diff} yields
\begin{align*}
  \langle\mathbf{f},\mathbf{L}(\mathbf{f})\rangle_{L^2_v\pa{\boldsymbol\mu^{-1/2}}}
	&\leq -(C_1-4\eta)\|\mathbf{f}-\mathbf{f^{\parallel}}\|_\mathcal H^2
	- C_2/2\:\mathcal{E}(\mathbf{f^{\parallel}}) \\
	&\leq -\left(C_1-4\eta-C_2C_3\eta\right)\|\mathbf{f}-\mathbf{f^{\parallel}}\|_\mathcal H^2 
	-(C_2 C_3\eta)/2\|\mathbf{f}-\pi_{\mathbf{L}}(\mathbf{f})\:\|_\mathcal H^2.
\end{align*}
The first term on the right-hand side is nonnegative if we choose
$$0<\eta \leq \min\br{1, C_1/(4+C_2C_3)},$$
and the desired spectral-gap estimate follows with $\lambda_L=(C_2 C_3C_4\eta)/2$, where the additional constant $C_4>0$ takes care of the fact that $\mathcal{H}$ is equivalent to $L^2_v\pa{\langle v \rangle^{\gamma/2}\boldsymbol\mu^{-1/2}}$.
\end{proof}
\bigskip

\begin{remark}
\begin{enumerate}
\item[(1)] We obtain the following relation between the spectral-gap constant $\lambda$ derived for same masses $m_i=m_j ~~\textnormal{for} ~~1\leq i,j \leq N$ in \cite[Theorem 3]{DJMZ} and our new constant $\lambda_L$ 
for different masses in Theorem \ref{theo:spectralgapL} :
\begin{align*}
 \lambda_L=\lambda\min_{1\leq i \leq N}m_i^2\frac{6 \rho_{\infty}}{\max\{\rho_{\infty}, 6 c_{\infty} \}},
\end{align*}
where $\rho_{\infty}=\sum_{i=1}^N m_ic_{\infty,i}$ and $c_{\infty}=\sum_{i=1}^N c_{\infty,i}$. 
Thus, increasing the difference between the masses $m_i$ makes the 
the spectral-gap constant $\lambda_L$ smaller, 
while in the special case of identical masses the two spectral-gap constants $\lambda$ and $\lambda_L$  are equal.

\item[(2)] Furthermore, the spectral-gap result of Theorem $\ref{theo:spectralgapL}$ only holds for a finite number of species $1\leq N<\infty$, since for $N \to \infty$ we get that $\lambda_L \to \infty$.
It remains an open problem whether or not it is possible to extend the result of Theorem $\ref{theo:spectralgapL}$ to the limit $N \to \infty$.
\end{enumerate}
\end{remark}
\bigskip

\section{$L^2$ theory for the linear part with maxwellian weight}\label{sec:L2theory}

This section is devoted to the study of the linear perturbed operator $\mathbf{G}=\mathbf{L}-v\cdot\nabla_x$ in $L^2_{x,v}\pa{\boldsymbol\mu^{-1/2}}$, which is the natural space for $\mathbf{L}$. We shall show that $\mathbf{G}$ generates a strongly continuous semigroup on this space.

\bigskip
\begin{theorem}\label{theo:semigroupL2}
We assume that assumptions $(H1) - (H4)$ hold for the collision kernel. Then the linear perturbed operator $\mathbf{G}=\mathbf{L}-v\cdot\nabla_x$ generates a strongly continuous semigroup $S_{\mathbf{G}}(t)$ on $L^2_{x,v}\pa{\boldsymbol\mu^{-1/2}}$ which satisfies
$$\forall t \geq 0, \quad \norm{S_{\mathbf{G}}(t)\pa{\mbox{Id}-\Pi_{\mathbf{G}}}}_{L^2_{x,v}\pa{\boldsymbol\mu^{-1/2}}} \leq C_G e^{-\lambda_G t},$$
where $\Pi_{\mathbf{G}}$ is the orthogonal projection onto $\mbox{Ker}(\mathbf{G})$ in $L^2_{x,v}\pa{\boldsymbol\mu^{-1/2}}$.
\\The constants $C_G$, $\lambda_G>0$ are explicit and depend on $N$, the different masses $m_i$ and the collision kernels.
\end{theorem}
\bigskip

Let us first make an important remark about $\Pi_\mathbf{G}$. Note that $\mathbf{G}(\mathbf{f}) = 0$ means
$$\forall i\in \br{1,\dots,N},\:\forall (x,v) \in \T^3\times\R^3, \quad v\cdot \nabla_x f_i (x,v) = L_i(\mathbf{f}(x,\cdot))(v)$$
Multiplying by $\mu_i^{-1}(v)f_i(x,v)$ and integrating over $\T^3\times\R^3$ implies
$$0 = \int_{\T^3} \langle L_i(\mathbf{f}(x,\cdot)),f_i(x,\cdot)\rangle_{L^2_v\pa{\mu_i^{-1/2}}}\:dx$$
and therefore by summing over $i$ in $\br{1,\dots,N}$
$$0 = \int_{\T^3} \langle \mathbf{L}(\mathbf{f}(x,\cdot)),\mathbf{f}(x,\cdot)\rangle_{L^2_v\pa{\boldsymbol\mu^{-1/2}}}\:dx.$$
The integrand is nonpositive thanks to the spectral gap of $\mathbf{L}$ and hence
$$\forall x \in \T^3,\:\forall v \in \R^3, \quad \mathbf{f}(x,v) = \pi_{\mathbf{L}} (\mathbf{f}(x,\cdot))(v)$$
and therefore $\mathbf{L}(\mathbf{f}(x,\cdot))= 0$. The latter further implies that $v\cdot\nabla_x \mathbf{f}(x,v) = 0$ which in turn implies that $\mathbf{f}$ does not depend on $x$ \cite[Lemma B.2]{Bri1}.
\par We can thus define the projection in $L^2_{x,v}\pa{\boldsymbol\mu^{-1/2}}$ onto the kernel of $\mathbf{G}$
\begin{equation}\label{PiG}
\Pi_\mathbf{G}(\mathbf{f}) = \sum\limits_{k=1}^{N+4} \pa{\int_{\T^3\times\R^3} \langle\mathbf{f}(x,v),\boldsymbol\phi_k(v)\rangle_{\boldsymbol\mu^{-1/2}}\:dxdv} \boldsymbol\phi_k(v),
\end{equation}
where the $\boldsymbol\phi_k$ were defined in $\eqref{piL}$. Again we define $\Pi_\mathbf{G}^\bot = \mbox{Id} - \Pi_\mathbf{G}$. Note that $\Pi_\mathbf{G}^\bot(\mathbf{f})=0$ amounts to saying that $\mathbf{f}$ satifies the multi-species perturbed conservation laws $\eqref{perturbedconservationlaws}$, \textit{i.e.} null individual mass, sum of momentum and sum of energy.

\bigskip

In Subsection \ref{subsec:controlfluidmicro}, we show the key lemma of the proof that is the \textit{a priori} control of the fluid part of $S_{\mathbf{G}}(t)$ by its orthogonal part, thus recovering some coercivity for $\mathbf{G}$ in the set of solutions to the linear perturbed equation. Subsection \ref{subsec:semigroupL2} is dedicated to the proof of Theorem \ref{theo:semigroupL2}.

\bigskip


\subsection{A priori control of the fluid part by the microscopic part}\label{subsec:controlfluidmicro}

As seen in the previous section, the operator $\mathbf{L}$ is only coercive on the orthogonal part. The key argument is to show that we recover some coercivity for solutions to the differential equation. Namely, that for these specific functions, the microscopic part controls the fluid part. This is the purpose of the next lemma

\bigskip
\begin{lemma}\label{lem:controlfluidmicro}
Let $\mathbf{f_0}(x,v)$ and $\mathbf{g}(t,x,v)$ be in $L^2_{x,v}\pa{\boldsymbol\mu^{-1/2}}$ such that $\Pi_{\mathbf{G}}(\mathbf{f_0})=\Pi_{\mathbf{G}}(\mathbf{g})=0$. Suppose that $\mathbf{f}(t,x,v)$ in $L^2_{x,v}\pa{\boldsymbol\mu^{-1/2}}$ is solution to the equation 
\begin{equation}\label{equationL2lem}
\partial_t \mathbf{f} = \mathbf{L}\pa{\mathbf{f}} - v\cdot\nabla_x \mathbf{f}+\mathbf{g}
\end{equation}
with initial value $\mathbf{f_0}$ and satisfying the multi-species conservation laws.  Then there exist an explicit $C_\bot >0$ and a function $N_\mathbf{f}(t)$ such that for all $t\geq 0$
\begin{itemize}
\item[(i)] $\abs{N_\mathbf{f}(t)}\leq C_\bot \norm{\mathbf{f}(t)}^2_{L^2_{x,v}\pa{\boldsymbol\mu^{-1/2}}}$;
\item[(ii)]  
\begin{equation*}
\begin{split}
\int_0^t\norm{\pi_\mathbf{L}(\mathbf{f})}^2_{L^2_{x,v}\pa{\boldsymbol\mu^{-1/2}}}\:ds \leq& N_\mathbf{f}(t)-N_\mathbf{f}(0) + C_\bot \int_0^t\norm{\pi^\bot_\mathbf{L}(\mathbf{f})}^2_{L^2_{x,v}\pa{\boldsymbol\mu^{-1/2}}}\:ds
\\& +C_\bot\int_0^t\norm{\mathbf{g}}_{L^2_{x,v}\pa{\boldsymbol\mu^{-1/2}}}ds.
\end{split}
\end{equation*}
\end{itemize}
The constant $C_\bot$ is independent of $\mathbf{f}$ and $\mathbf{g}$.
\end{lemma}
\bigskip

The methods of the proof are a technical adaptation of the method proposed in \cite{EGKM} in the case of bounded domain with diffusive boundary conditions. The description of $\mbox{Ker}(\mathbf{L})$ associated with the global equilibrium $\boldsymbol\mu$ is given by orthogonal functions in $L^2_v$ but that are not of norm one. Unlike \cite{EGKM} where only  mass conservation holds but boundary conditions overcome the lack of conservation laws, we strongly need the conservation of mass, momentum and energy.

\bigskip
\begin{proof}[Proof of Lemma \ref{lem:controlfluidmicro}]

We recall $\eqref{piL}$ the definition of $\pi_\mathbf{L}(\mathbf{f}) = \pa{\pi_i(\mathbf{f})}_{1\leq i \leq N}$  and we define $(a_i(t,x))_{1\leq i \leq N}$, $b(t,x)$ and $c(t,x)$ to be the coordinates of $\pi_\mathbf{L}(\mathbf{f})$: 
\begin{equation}\label{piLL2}
\forall \: 1\leq i\leq N, \quad \pi_i(\mathbf{f})(t,x,v) = \cro{a_i(t,x) + b(t,x)\cdot v + c(t,x)\frac{\abs{v^2}-3m_i^{-1}}{2}}m_i\mu_i(v).
\end{equation}
Note that we are working with an orthogonal but not orthonormal basis of $\mbox{Ker}(\mathbf{L})$ in $L^2_{x,v}(\boldsymbol{\mu^{-1/2}})$ in order to lighten computations. We will denote by $\rho_i$ the mass of $m_i\mu_i$.
\par The key idea of the proof is to choose suitable test functions $\boldsymbol\psi = \pa{\psi_i}_{1\leq i\leq N}$ in $H^1_{x,v}$ that will catch the elliptic regularity of $a_i$, $b$ and $c$ and estimate them.
\par For a test function $\boldsymbol\psi=\boldsymbol\psi(t,x,v)$ integrated against the differential equation $\eqref{equationL2lem}$ we have by Green's formula on each coordinate
\begin{eqnarray*}
&&\int_0^t\frac{d}{dt}\int_{\T^3\times\R^3} \langle \boldsymbol\psi, \mathbf{f}\rangle_{\mathbf{1}}\:dxdvds = \int_{\T^3\times\R^3} \langle \boldsymbol\psi(t), \mathbf{f}(t)\rangle_{\mathbf{1}}\:dxdv -\int_{\T^3\times\R^3} \langle \boldsymbol\psi_0, \mathbf{f_0}\rangle_{\mathbf{1}}\:dxdv   
\\&&\quad\quad= \int_0^t \int_{\T^3\times\R^3} \langle \mathbf{f}, \partial_t\boldsymbol\psi\rangle_{\mathbf{1}}\:dxdvds + \int_0^t \int_{\T^3\times\R^3} \langle \mathbf{L}\pa{\mathbf{f}},\boldsymbol\psi\rangle_{\mathbf{1}}\:dxdvds
\\&&\quad\quad\quad+\sum\limits_{i=1}^N\int_0^t \int_{\T^3\times\R^3} f_i v\cdot\nabla_x\psi_i \:dxdvds + \int_0^t\int_{\T^3\times\R^3}\langle \boldsymbol\psi,\mathbf{g} \rangle_{\mathbf{1}}\:dxdvds. 
\end{eqnarray*}
We decompose $\mathbf{f} = \pi_\mathbf{L}(\mathbf{f}) + \pi_\mathbf{L}^\bot(\mathbf{f})$ in the term involving $v\cdot\nabla_x$ and use the fact that $\mathbf{L}\pa{\mathbf{f}} = \mathbf{L}[\pi_\mathbf{L}^\bot(\mathbf{f})]$ to obtain the weak formulation
\begin{equation}\label{eqpsi}
-\sum\limits_{i=1}^N\int_0^t \int_{\T^3\times\R^3} \pi_i(\mathbf{f}) v\cdot\nabla_x\psi_i \:dxdvds = \Psi_1(t) + \Psi_2(t)+\Psi_3(t) + \Psi_4(t)+\Psi_5(t)
\end{equation}
with the following definitions
\begin{eqnarray}
\Psi_1(t) &=& \int_{\T^3\times\R^3} \langle \boldsymbol\psi_0, \mathbf{f_0}\rangle_{\mathbf{1}}\:dxdv -\int_{\T^3\times\R^3} \langle \boldsymbol\psi(t), \mathbf{f}(t)\rangle_{\mathbf{1}}\:dxdv, \label{Psi1}
\\\Psi_2(t) &=& \sum\limits_{i=1}^N\int_0^t \int_{\T^3\times\R^3} \pi^\bot_\mathbf{L}(\mathbf{f})_i v\cdot\nabla_x\psi_i \:dxdvds, \label{Psi2}
\\\Psi_3(t) &=& \sum\limits_{i=1}^N\int_0^t \int_{\T^3\times\R^3} \mathbf{L}\pa{\pi_\mathbf{L}^\bot(\mathbf{f})}_i \psi_i\:dxdvds, \label{Psi3}
\\\Psi_4(t) &=& \sum\limits_{i=1}^N\int_0^t \int_{\T^3\times\R^3} f_i\partial_s\psi_i\:dxdvds, \label{Psi4}
\\\Psi_5(t) &=& \int_0^t\int_{\T^3\times\R^3}\langle \boldsymbol\psi,\mathbf{g} \rangle_{\mathbf{1}}\:dxdvds. \label{Psi5}
\end{eqnarray}

\par For each of the functions $\mathbf{a} = \pa{a_i}_{1\leq i\leq N}$, $b$ and $c$, we construct a $\boldsymbol\psi$ such that the left-hand side of $\eqref{eqpsi}$ is exactly the $L^2_x$-norm of the function and the rest of the proof is estimating the four different terms $\Psi_i(t)$. Note that $\Psi_1(t)$ is already under the desired form
\begin{equation}\label{Psi1all}
\Psi_1(t) = N_\mathbf{f}(t)-N_\mathbf{f}(0)
\end{equation}
with $\abs{N_\mathbf{f}(s)} \leq C\norm{\mathbf{f}}^2_{L^2_{x,v}\pa{\boldsymbol\mu^{-1/2}}}$ if $\psi_i(x,v)\mu_i^{1/2}(v)$ is in $L^2_{x,v}$ for all $i$ and their norm is controlled by the one of $\mathbf{f}$ (which will be the case in our next choices).

\bigskip
\begin{remark}
The linear perturbed equation $\eqref{equationL2lem}$ and the conservation laws are invariant under standard time mollification. We therefore consider for simplicity in the rest of the proof that all functions are smooth in the variable $t$. Exactly the same estimates can be derived for more general functions and the method would obviously be to study time mollified equation and then take the limit in the smoothing parameter.
\end{remark}
\bigskip

For clarity, every positive constant will be denoted by $C_k$.


\bigskip
\textbf{Estimate for $\mathbf{a}=\pa{a_i}_{1\leq i \leq N}$.} By assumption $\mathbf{f}$ preserves the mass which is equivalent to
$$0 = \int_{\T^3\times\R^3}\mathbf{f}(t,x,v)\:dxdv = \int_{\T^3} \pa{\int_{\R^3}\langle\mathbf{f}(t,x,v),\boldsymbol\mu\rangle_{\boldsymbol\mu^{-1/2}}\:dv}dx =\int_{\T^3} \mathbf{a}(t,x)\:dx,$$
where we used the fact that $\boldsymbol{\mu} \in \mbox{Ker}(\mathbf{G})$, $\mathbf{f_0} \in \mbox{Ker}(\mathbf{G})^{\perp}$ and the orthogonality of the basis defined in 
\eqref{piLL2}.
Define a test function $\boldsymbol\psi_\mathbf{a} = \pa{\psi_i}_{1\leq i\leq N}$ by
$$\psi_i(t,x,v) = \pa{\abs{v}^2-\alpha_i}v\cdot\nabla_x\phi_i(t,x)$$
where
$$-\Delta_x\phi_i(t,x) = a_i(t,x)$$
and $\alpha_i>0$ is chosen such that for all $1\leq k \leq 3$
$$\int_{\R^3} \pa{\abs{v}^2-\alpha_i}\frac{\abs{v}^2-3m_i^{-1}}{2}v_k^2\mu_i(v)\:dv = 0.$$
The integral over $\T^3$ of $a_i(t,\cdot)$ is null and therefore standard elliptic estimate \cite{Eva} yields:
\begin{equation}\label{phiaH2}
\forall t \geq 0,\quad\norm{\phi_i(t)}_{H^2_x}\leq C_0\norm{a_i(t)}_{L^2_x}.
\end{equation}
The latter estimate provides both the control of $\Psi_1 = N^{(a)}_\mathbf{f}(t) - N^{(a)}_\mathbf{f}(0)$, as discussed before, and the control of $\eqref{Psi5}$, using Cauchy-Schwarz and Young's inequality,
\begin{eqnarray}
 \abs{\Psi_5(t)} &\leq& C\sum\limits_{i=1}^N\int_0^t \norm{\sqrt{\rho_i}\phi_i}_{L^2_x}\norm{g_i}_{L^2_{x,v}\pa{\mu_i^{-1/2}}}\:ds\nonumber
 \\&\leq& \frac{C_1}{4} \int_0^t\norm{\mathbf{a}}^2_{L^2_x\pa{\boldsymbol\rho^{1/2}}}\:ds + C_5\int_0^t \norm{\mathbf{g}}^2_{L^2_{x,v}\pa{\boldsymbol\mu^{-1/2}}}\:ds,\label{Psi5a}
\end{eqnarray}
where $C_1>0$ is given in $\eqref{RHSa}$ below and where we defined $\boldsymbol\rho = \pa{\rho_i}_{1\leq i \leq N}$ the vector of the masses associated to $(m_i\mu_i)_{1\leq i \leq N}$.

\bigskip
Firstly, we compute the term on the left-hand side of $\eqref{eqpsi}$.
\begin{eqnarray*}
&&-\sum\limits_{i=1}^N\int_0^t \int_{\T^3\times\R^3} \pi_i(\mathbf{f}) v\cdot\nabla_x\psi_i \:dxdvds 
\\&& = -\sum\limits_{i=1}^N\sum\limits_{1\leq j,k \leq 3}\int_0^t \int_{\T^3} a_i(s,x)\pa{\int_{\R^3}\pa{\abs{v}^2-\alpha_i}v_jv_km_i\mu_i(v)\:dv}\partial_{x_j}\partial_{x_k}\phi_i\:dxds
\\&&\quad -\sum\limits_{i=1}^N\sum\limits_{1\leq j,k \leq 3}\int_0^t \int_{\T^3} b(s,x)\cdot \pa{\int_{\R^3}v\pa{\abs{v}^2-\alpha_i}v_jv_km_i\mu_i(v)\:dv}\partial_{x_j}\partial_{x_k}\phi_i\:dxds
\\&&\quad -\sum\limits_{i=1}^N\sum\limits_{1\leq j,k \leq 3}\int_0^t \int_{\T^3} c(s,x)\pa{\int_{\R^3}\pa{\abs{v}^2-\alpha_i}\frac{\abs{v}^2-3m_i^{-1}}{2}v_jv_km_i\mu_idv}\partial_{x_j}\partial_{x_k}\phi_i.
\end{eqnarray*}
The second term is null as well as the first and last ones when $j\neq k$ thanks to the oddity in $v$. In the last term when $j=k$ we recover our choice of $\alpha_i$ which makes the last term being null too. It remains the first term when $k=j$. In this case, the integral in $v$ gives a constant $C_1$ independent of $i$ times $\rho_i$. Direct computations give $\alpha_i=10/m_i$ and $C_1>0$. It follows
\begin{eqnarray}
-\sum\limits_{i=1}^N\int_0^t \int_{\T^3\times\R^3} \pi_i(\mathbf{f}) v\cdot\nabla_x\psi_i \:dxdvds &=& -C_1\sum\limits_{i=1}^N \int_0^t\int_{\T^3}a_i(s,x)\rho_i\Delta_x \phi_i(s,x)\:dxds\nonumber
\\&=& C_1 \sum\limits_{i=1}^N\int_0^t \int_{\T^3}a_i^2\rho_i\:ds \nonumber
\\&=& C_1 \int_0^t \norm{\mathbf{a}(s)}^2_{L^2_{x}\pa{\boldsymbol\rho^{1/2}}}\:ds. \label{RHSa}
\end{eqnarray}

\bigskip
We recall $\mathbf{L}=-\boldsymbol\nu(v)+\mathbf{K}$ where $\mathbf{K}$ is a bounded operator in $L^2_v\pa{\boldsymbol\mu^{-1/2}}$. Moreover, the $H^2_x$-norm of $\phi_i(t,x)$ is bounded by the $L^2_x$-norm of $a_i(t,x)$. Multiplying by $\mu_i^{1/2}(v)\mu_i(v)^{-1/2}$ inside the $i^{th}$ integral of $\Psi_2$ $\eqref{Psi2}$ and of $\Psi_3$ $\eqref{Psi3}$ a mere Cauchy-Schwarz inequality yields
\begin{equation}\label{Psi23a}
\begin{split}
\forall k\in\br{2,3}, \quad \abs{\Psi_k(t)} &\leq C \sum\limits_{i=1}^N\int_0^t\norm{\sqrt{\rho_i}a_i}_{L^2_{x}}\norm{\pi_i^\bot(\mathbf{f})}_{L^2_{x,v}\pa{\mu_i^{-1/2}}}\:ds
\\&\leq  \frac{C_1}{4} \int_0^t\norm{\mathbf{a}}^2_{L^2_{x}\pa{\boldsymbol\rho^{1/2}}}\:ds + C_2\int_0^t\norm{\pi_\mathbf{L}^\bot(\mathbf{f})}^2_{L^2_{x,v}\pa{\boldsymbol\mu^{-1/2}}}\:ds.
\end{split}
\end{equation}
We used Young's inequality for the last inequality, with $C_1$ defined in $\eqref{RHSa}$.

\bigskip
It remains to estimate the term with time derivatives $\eqref{Psi4}$. It reads
\begin{eqnarray*}
\Psi_4(t) &=&\sum\limits_{i=1}^N\int_0^t \int_{\T^3\times\R^3}f_i\pa{\abs{v}^2-\alpha_i} v\cdot\cro{\partial_t\nabla_x\phi_i}\:dxdvds
\\&=& \sum\limits_{i=1}^N\sum\limits_{k=1}^3\int_0^t \int_{\T^3\times\R^3}\pi_i(\mathbf{f})\pa{\abs{v}^2-\alpha_i} v_k\partial_t\partial_{x_k}\phi_i\:dxdvds 
\\ &\:&+ \sum\limits_{i=1}^N\int_0^t \int_{\T^3\times\R^3}\pi_i^\bot(\mathbf{f})\pa{\abs{v}^2-\alpha_i} v\cdot\cro{\partial_t\nabla_x\phi_i}\:dxdvds
\end{eqnarray*}
Using oddity properties for the first integral on the right-hand side and then Cauchy-Schwarz with the following bound
$$\int_{\R^3}\pa{\abs{v}^2-\alpha_i}^2 \abs{v}^2 \mu_i(v)\:dv = C\rho_i<+\infty$$
we get
\begin{equation}\label{Psi4astart}
\abs{\Psi_4(t)}\leq C\sum\limits_{i=1}^N\int_0^t\cro{\sum\limits_{k=1}^3\norm{\rho_ib_k}_{L^2_{x}}+\norm{\pi_i^\bot(\mathbf{f})}_{L^2_{x,v}\pa{\mu_i^{-1/2}}}}\norm{\partial_t\nabla_x\phi_i}_{L^2_x}\:ds.
\end{equation}

\par Estimating $\norm{\partial_t\nabla_x\phi_a}_{L^2_x}$ will come from elliptic estimates in negative Sobolev spaces. We use the decomposition of the weak formulation $\eqref{eqpsi}$ between $t$ and $t+\eps$ (instead of between $0$ and $t$) with $\boldsymbol\psi(t,x,v) = \phi(x)\mathbf{e}_i \in H^1_x$, where $\mathbf{e}_i = \pa{\delta_{ji}}_{1\leq j \leq N}$. We furthermore require that $\phi(x)$ has a null integral over $\T^3$. $\boldsymbol\psi$ only depends on $x$ and therefore $\Psi_4(t)=0$. Moreover, multiplying by $\mu_i(v)\mu_i^{-1}(v)$ in the $i^{th}$ integral of $\Psi_3$ yields
$$\Psi_3(t)= \int_t^{t+\eps} \int_{\T^3} \langle \mathbf{L}(\mathbf{f}), \mu_i\mathbf{e}_i\rangle_{L^2_v\pa{\boldsymbol\mu^{-1/2}}}\phi(x) \:dxdvds  = 0,$$
by definition of $\mbox{Ker}(\mathbf{L})$.
\par  From the weak formulation $\eqref{eqpsi}$ it therefore remains
\begin{equation*}
\begin{split}
&\int_{\T^3\times\R^3} \phi(x)\langle \mathbf{e}_i, \mathbf{f}(t+\eps)\rangle_{\mathbf{1}}\:dxdv -\int_{\T^3\times\R^3} \phi(x)\langle \mathbf{e}_i, \mathbf{f}(t)\rangle_{\mathbf{1}}\:dxdv 
\\&\quad= \int_{t}^{t+\eps} \int_{\T^3\times\R^3}\pi_i(\mathbf{f})v\cdot\nabla_x\phi(x)\:dxdvds + \int_{t}^{t+\eps} \int_{\T^3\times\R^3}\pi^\bot_i(\mathbf{f})v\cdot\nabla_x\phi(x)\:dxdvds
\\&\quad\quad + \int_t^{t+\eps} \int_{\Omega\times\R^3}g_i(s,x,v)\phi(x)\:dxdvds
\end{split}
\end{equation*}
which is equal to
\begin{eqnarray*}
\int_{\T^3} \rho_i\cro{a_i(t+\eps)-a_i(t)}\phi(x)\:dx &=& C\int_t^{t+\eps}\int_{\T^3} \rho_i b(s,x)\cdot \nabla_x\phi(x)\:dxds
\\&&+ \int_{t}^{t+\eps} \int_{\T^3\times\R^3}\pi^\bot_i(\mathbf{f})\mu_i(v)^{-1/2}\mu_i(v)^{1/2}v\cdot\nabla_x\phi(x)
\\&&+\int_t^{t+\eps} \int_{\Omega\times\R^3}g_i(s,x,v)\phi(x)\:dxdvds,
\end{eqnarray*}
where $C$ does not depend on $i$.
\par Dividing by $\rho_i\eps$ and taking the limit as $\eps$ goes to $0$ yields, after a mere Cauchy-Schwarz inequality on the right-hand side
\begin{eqnarray*}
\abs{\int_{\T^3}\partial_ta_i(s,x)\phi(x)\:dx} &\leq& C\cro{\norm{b(t,x)}_{L^2_x}+\norm{\pi_i^\bot(\mathbf{f})}_{L^2_{x,v}\pa{\mu_i^{-1/2}}}}\norm{\nabla_x\phi(x)}_{L^2_x}
\\&&+ C\norm{g_i}_{L^2_{x,v}\pa{\mu_i^{-1/2}}}\norm{\phi}_{L^2_x}
\\&\leq& C\cro{\norm{b(t,x)}_{L^2_x}+\norm{\pi_i^\bot(\mathbf{f})}_{L^2_{x,v}\pa{\mu_i^{-1/2}}}+\norm{g_i}_{L^2_{x,v}\pa{\mu_i^{-1/2}}}}
\\&&\times\norm{\nabla_x\phi(x)}_{L^2_x}.
\end{eqnarray*}
We used Poincar\'e inequality since $\phi(x)$ has a null integral over $\T^d$. The latter inequality is true for all $\phi$ in $H^1_x$ with a null integral and therefore implies for all $t\geq 0$
\begin{equation}\label{H1*a}
\norm{\partial_ta_i(t,x)}_{\pa{\mathcal{H}^1_x}^*} \leq C\cro{\norm{b(t,x)}_{L^2_x}+\norm{\pi_i^\bot(\mathbf{f})}_{L^2_{x,v}\pa{\mu_i^{-1/2}}}+ \norm{g_i}_{L^2_{x,v}\pa{\mu_i^{-1/2}}}}
\end{equation}
where $\pa{\mathcal{H}^1_x}^*$ is the dual of the set of functions in $H^1_x$ with null integral.
\par Thanks to the conservation of mass we have that $\partial_ta_i(t,x)$ have a zero integral on the torus and we can construct $\Phi_i(t,x)$ such that
$$-\Delta_x\Phi_i(t,x) = \partial_ta_i(t,x)$$
and by standard elliptic estimate \cite{Eva}:
$$\norm{\Phi_i}_{\mathcal{H}^1_x} \leq \norm{\partial_t a_i}_{\pa{\mathcal{H}^1_x}^*} \leq C\cro{ \norm{b(t,x)}_{L^2_x}+\norm{\pi_i^\bot(\mathbf{f})}_{L^2_{x,v}\pa{\mu_i^{-1/2}}} + \norm{g_i}_{L^2_{x,v}\pa{\mu_i^{-1/2}}}},$$
where we used $\eqref{H1*a}$. Combining this estimate with 
$$\norm{\partial_t\nabla_x\phi_i}_{L^2_x} = \norm{\nabla_x\Delta^{-1}\partial_t a_i}_{L^2_x} \leq  \norm{\Delta^{-1}\partial_t a_i}_{H^1_x} = \norm{\Phi_i}_{H^1_x}$$
we can further control $\Psi_4$ in $\eqref{Psi4astart}$ using $\rho_i = \sqrt{\rho_i}\sqrt{\rho_i}$
\begin{equation}\label{Psi4a}
\abs{\Psi_4(t)} \leq C_5\int_0^t\pa{\sum\limits_{i=1}^N\norm{\sqrt{\rho_i}b}^2_{L^2_{x}} + \norm{\pi_i^\bot(\mathbf{f})}^2_{L^2_{x,v}\pa{\mu_i^{-1/2}}}+ \norm{g_i}^2_{L^2_{x,v}\pa{\mu_i^{-1/2}}}}\:ds.
\end{equation}

\bigskip
We now plug $\eqref{RHSa}$, $\eqref{Psi1all}$, $\eqref{Psi23a}$, $\eqref{Psi4a}$ and $\eqref{Psi5a}$ into $\eqref{eqpsi}$
\begin{equation}\label{afinal}
\begin{split}
\int_0^t \norm{\mathbf{a}}^2_{L^2_{x}\pa{\boldsymbol\rho^{1/2}}}\:ds \leq & N^{(a)}_\mathbf{f}(t)-N^{(a)}_\mathbf{f}(0) +C_{a,b}\int_0^t \norm{b}^2_{L^2_{x}\pa{\boldsymbol\rho^{1/2}}}\:ds
\\ &+ C_a\int_0^t \cro{\norm{\pi_\mathbf{L}^\bot(\mathbf{f})}^2_{L^2_{x,v}\pa{\boldsymbol\mu^{-1/2}}}+\norm{\mathbf{g}}^2_{L^2_{x,v}\pa{\boldsymbol\mu^{-1/2}}}}\:ds.
\end{split}
\end{equation}


\bigskip
\textbf{Estimate for $b$.} The choice of function to integrate against to deal with the $b$ term is more involved technically.
We emphasize that $b(t,x)$ is a vector $\pa{b_1(t,x),b_2(t,x),b_3(t,x)}$,  
thus we used the obvious short-hand notation
$$\norm{b}^2_{L^2_{x}\pa{\boldsymbol\rho^{1/2}}} = \sum\limits_{i=1}^N\sum\limits_{k=1}^3\norm{\sqrt{\rho_i}b_k}^2_{L^2_x}.$$
\par Fix $J$ in $\br{1,2,3}$ and the conservation of momentum implies that for all $t\geq 0$
$$\int_{\T^3}b_J(t,x)\:dx = 0.$$
Define $\boldsymbol\psi_{b_J}(t,x,v) = \pa{\psi_{iJ}(t,x,v)}_{1\leq i \leq N}$ with
$$\psi_{iJ}(t,x,v) = \sum\limits_{j=1}^3\varphi^{(J)}_{ij}(t,x,v),$$
with
$$\varphi^{(J)}_{ij}(t,x,v) = \left\{\begin{array}{l} \disp{\abs{v}^2v_jv_J\partial_{x_j}\phi_J(t,x) - \frac{7}{2m_i}\pa{v_j^2-m_i^{-1}}\partial_{x_J}\phi_J(t,x), \quad\mbox{if}\:\:j\neq J} \vspace{2mm}\\\vspace{2mm} \disp{\frac{7}{2m_i}\pa{v_J^2-m_i^{-1}}\partial_{x_J}\phi_J(t,x), \quad\mbox{if}\:\: j=J.} \end{array}\right.$$
where
$$-\Delta_x\phi_J(t,x) = b_J(t,x).$$
Since it will be important, we emphasize here that for all $j \neq k$
\begin{equation}\label{contribution0}
\int_{\R^3}\pa{v_j^2-m_i^{-1}}\mu_i(v)\:dv =0 \quad\mbox{and}\quad \int_{\R^3}\pa{v_j^2-m_i^{-1}}v_k^2\mu_i(v)\:dv =0.
\end{equation}
The null integral of $b_J$ implies by standard elliptic estimate \cite{Eva}
\begin{equation}\label{phibH2}
\forall t \geq 0,\quad\norm{\phi_J(t)}_{H^2_x}\leq C_0\norm{b_J(t)}_{L^2_x}.
\end{equation}
Again, this estimate provides the control of $\Psi_1(t) = N^{(J)}_{\mathbf{f}}(t) - N^{(J)}_{\mathbf{f}}(0)$ and of $\Psi_5(t)$ as in $\eqref{Psi5a}$:
\begin{equation}\label{Psi5b}
 \abs{\Psi_5(t)} \leq \frac{C_1}{4} \int_0^t\norm{b_J}^2_{L^2_x\pa{\boldsymbol\rho^{1/2}}}\:ds + C_5\int_0^t \norm{\mathbf{g}}^2_{L^2_{x,v}\pa{\boldsymbol\mu^{-1/2}}}\:ds,
\end{equation}
where $C_1>0$ is given in $\eqref{RHSb}$ below.

\bigskip
We start by the left-hand side of $\eqref{eqpsi}$. By oddity, there is neither contribution from any of the $a_i(s,x)$ nor from $c(s,x)$. Hence, for all $i$ in $\br{1,\dots,N}$
\begin{eqnarray*}
&&-\int_0^t \int_{\Omega\times\R^3} \pi_i(\mathbf{f}) v\cdot\nabla_x\psi_{iJ} \:dxdvds 
\\&& = -\sum\limits_{1\leq k,l\leq 3}\sum\limits_{\overset{j=1}{j\neq J}}^3\int_0^t \int_{\Omega}  b_l(s,x)\pa{\int_{\R^3}\abs{v^2}v_lv_kv_jv_Jm_i\mu_i(v)\:dv}\partial_{x_k}\partial_{x_j}\phi_J(s,x)\:dxds
\\&&\quad + \frac{7}{2m_i}\sum\limits_{1\leq k,l\leq 3}\sum\limits_{\overset{j=1}{j\neq J}}^3\int_0^t \int_{\Omega} b_l(s,x) \pa{\int_{\R^3}\pa{v_j^2-m_i^{-1}}v_lv_km_i\mu_idv}\partial_{x_k}\partial_{x_J}\phi_J\:dxds
\\&&\quad -\frac{7}{2m_i}\sum\limits_{1\leq k,l\leq 3}\int_0^t \int_{\Omega} b_l(s,x)\pa{\int_{\R^3}\pa{v_J^2-m_i^{-1}}v_lv_km_i\mu_i(v)\:dv}\partial_{x_k}\partial_{x_J}\phi_J\:dxds.
\end{eqnarray*}
The last two integrals on $\R^3$ are zero if $l\neq k$. Moreover, when $l=k$ and $l\neq J$ it is also zero by $\eqref{contribution0}$. We compute directly for $l= J$
$$\int_{\R^3}\pa{v_J^2-m_i^{-1}}v_J^2m_i\mu_i(v)\:dv = \frac{2}{m_i^2}\rho_i.$$
The first term is composed by integrals in $v$ of the form
$$\int_{\R^3}\abs{v}^2v_kv_jv_lv_J\mu_i(v)\:dv$$
which is always null unless two indices are equals to the other two. Therefore if $j=l$ then $k=J$ and if $j\neq l$ we only have two options: $k=j$ and $l=J$ or $k=l$ and $j=J$. Hence, for all $i$ in $\br{1,\dots,N}$
\begin{eqnarray*}
&&-\int_0^t \int_{\Omega\times\R^3} \pi_i(\mathbf{f}) v\cdot\nabla_x\psi_J \:dxdvds
\\&& =  -\sum\limits_{\overset{j=1}{j\neq J}}^3\int_0^t\int_{\Omega}b_J(s,x)\partial_{x_jx_j}\phi_J \pa{\int_{\R^3}\abs{v}^2v_j^2v_J^2m_i\mu_i(v)\:dv}dxds
\\&&\quad -\sum\limits_{\overset{j=1}{j\neq J}}^3\int_0^t\int_{\Omega} b_j(s,x)\partial_{x_jx_J}\phi_J \pa{\int_{\R^3}\abs{v}^2v_j^2v_J^2m_i\mu_i(v)\:dv}dxds
\\&&\quad + \frac{7}{m_i^3} \sum\limits_{\overset{j=1}{j\neq J}}^3\int_0^t\int_{\Omega} \rho_i b_j(s,x)\partial_{x_jx_J}\phi_J\:dxds
- \frac{7}{m_i^3}\int_0^t \int_{\Omega} \rho_i b_J(s,x)\partial_{x_J}\partial_{x_J}\phi_J(s,x)\:dxds.
\end{eqnarray*} 
To conlude we compute for $j\neq J$
$$\int_{\R^3}\abs{v^2}v_j^2v_J^2 m_i\mu_i(v)\:dv = \frac{7}{m_i^3}\rho_i$$
and it thus only remains the following equality for all $i$ in $\br{1,\dots,N}$.
\begin{eqnarray*}
-\int_0^t \int_{\Omega\times\R^3} \pi_i(\mathbf{f}) v\cdot\nabla_x\psi_J \:dxdvds &=& -\frac{7}{m_i^3} \int_0^t\int_{\Omega}\rho_ib_J(s,x)\Delta_x \phi_J(s,x)\:dxds
\\&=& \frac{7}{m_i^3} \int_0^t\norm{\sqrt{\rho_i}b_J}^2_{L^2_x}\:ds. 
\end{eqnarray*}
Summing over $i$ yields
\begin{equation}\label{RHSb}
-\sum\limits_{i=1}^N\int_0^t \int_{\Omega\times\R^3} \pi_j(\mathbf{f}) v\cdot\nabla_x\psi_J = \frac{7}{m_i^3} \int_0^t \norm{b_J}_{L^2_{x}\pa{\boldsymbol\rho^{1/2}}} \:dxdvds.
\end{equation}
We recall $\boldsymbol\rho = \pa{\rho_i}_{1\leq i \leq N}$.

\bigskip
Then the terms $\Psi_2$ and $\Psi_3$ are dealt with as in $\eqref{Psi23a}$
\begin{equation}\label{Psi23b}
\forall k\in\br{2,3}, \quad \abs{\Psi_k(t)} \leq   \frac{7}{4}\int_0^t\norm{b_J}^2_{L^2_x\pa{\boldsymbol\rho^{1/2}}}\:ds + C_2\int_0^t\norm{\pi_\mathbf{L}^\bot(\mathbf{f})}^2_{L^2_{x,v}\pa{\boldsymbol\mu^{-1/2}}}\:ds.
\end{equation}

\bigskip
It remains to estimate $\Psi_4$ which involves time derivative $\eqref{Psi4}$:
\begin{eqnarray*}
\Psi_4(t) &=&\sum\limits_{i=1}^N\sum\limits_{j=1}^3\int_0^t \int_{\Omega\times\R^3}f_i \partial_t\varphi^{(J)}_{ij}(s,x,v)\:dxdvds
\\&=& \sum\limits_{i=1}^N\sum\limits_{j=1}^3\int_0^t \int_{\Omega\times\R^3}\pi_i^\bot(\mathbf{f}) \partial_t\varphi^{(J)}_{ij}(s,x,v)\:dxdvds
\\&\:& +\sum\limits_{i=1}^N\sum\limits_{\overset{j=1}{j\neq J}}^3\int_0^t \int_{\Omega\times\R^3}\pi_i(\mathbf{f})\abs{v}^2v_jv_J\partial_{x_j}\phi_J\:dxdvds
\\&\:& +\sum\limits_{i=1}^N\sum\limits_{j=1}^3 \pm \frac{7}{2m_i}\int_0^t \int_{\Omega\times\R^3} \pi_i(\mathbf{f})\pa{v_j^2-m_j^{-1}}\partial_{x_J}\phi_J\:dxdvds.
\end{eqnarray*}
By oddity arguments, only terms in $a_i(s,x)$ and $c(s,x)$ can contribute to the last two terms on the right-hand side. However, $j\neq J$ implies that the second term is zero as well as the contribution of $a_i(s,x)$ in the third term thanks to $\eqref{contribution0}$. Finally, a Cauchy-Schwarz inequality on both integrals yields as in $\eqref{Psi4astart}$
\begin{equation}\label{Psi4bstart}
\abs{\Psi_4(t)}\leq C\sum\limits_{i=1}^N\int_0^t\cro{\norm{\rho_i c}_{L^2_x}+\norm{\pi_i^\bot(\mathbf{f})}_{L^2_{x,v}\pa{\mu_i^{-1/2}}}}\norm{\partial_t\nabla_x\phi_J}_{L^2_x}\:ds.
\end{equation}

\par To estimate $\norm{\partial_t\nabla_x\phi_J}_{L^2_x}$ we follow the idea developed for $\mathbf{a}(s,x)$ about negative Sobolev regularity. We apply the weak formulation $\eqref{eqpsi}$ to a specific function between $t$ and $t+\eps$. The test function is $\boldsymbol\psi(x,v) = \phi(x)v_J\mathbf{m}$ with $\phi$ in $H^1_x$ with a zero integral over $\T^3$. Note that $\psi$ does not depend on $t$ so $\Psi_4=0$ and multiplying by $\mu_i(v)\mu_i^{-1}(v)$ in the $i^{th}$ integral of $\Psi_3$ yields
$$\Psi_3(t)= \int_0^t \int_{\T^3} \langle \mathbf{L}(\mathbf{f}), v_J(m_i\mu_i)_{1\leq i\leq N}\rangle_{L^2_v\pa{\boldsymbol\mu^{-1/2}}}\partial_{x_k}\phi(x) \:dxdvds  = 0,$$
by definition of $\mbox{Ker}(\mathbf{L})$.
\par It remains
\begin{eqnarray*}
&&C\sum\limits_{i=1}^N\int_{\Omega} \rho_i\cro{b_J(t+\eps)-b_J(t)}\phi(x)\:dx 
\\&&\quad\quad\quad\quad\quad = \sum\limits_{i=1}^N\int_{t}^{t+\eps} \int_{\Omega\times\R^3}\pi_i(\mathbf{f})v_Jv\cdot\nabla_x\phi(x)\:dxdvds
\\&&\quad\quad\quad\quad\quad\quad + \sum\limits_{i=1}^N\int_{t}^{t+\eps} \int_{\Omega\times\R^3}\pi_i^\bot(\mathbf{f})v_Jv\cdot\nabla_x\phi(x)\:dxdvds
\\&&\quad\quad\quad\quad\quad\quad + \sum\limits_{i=1}^N\int_{t}^{t+\eps} \int_{\Omega\times\R^3} g_iv_J \phi(x)\:dxdvds.
\end{eqnarray*}
As for $a_i(t,x)$ we divide by $\eps$ and take the limit as $\eps$ goes to $0$. By oddity, the first integral on the right-hand side only gives terms with $a_i(s,x)$ and $c(s,x)$. The other two integrals are dealt with by a Cauchy-Schwarz inequality and Poincar\'e. This yields
\begin{equation}\label{negativesobolevb}
\begin{split}
&\abs{\int_{\Omega} \partial_t b_J(t,x)\phi(x)\:dx} 
\\&\quad\leq C  \cro{\norm{\mathbf{a}}_{L^2_x\pa{\boldsymbol\rho^{1/2}}}+\norm{c}_{L^2_x\pa{\boldsymbol\rho^{1/2}}}+\norm{\pi_\mathbf{L}^\bot(\mathbf{f})}_{L^2_{x,v}\pa{\boldsymbol\mu^{-1/2}}}+\norm{\mathbf{g}}_{L^2_{x,v}\pa{\boldsymbol\mu^{-1/2}}}}\norm{\nabla_x\phi}_{L^2_x}.
\end{split}
\end{equation}

\par The latter is true for all $\phi(x)$ in $H^1_x$ with a null integral over $\T^3$. We thus fix $t$ and apply the inequality above to 
$$-\Delta_x\phi(t,x) = \partial_tb_J(t,x)$$
which has a zero integral thanks to the conservation of momentum
and obtain
$$\norm{\partial_t\nabla_x\phi_J}_{L^2_x}^2 = \norm{\nabla_x\Delta^{-1}\partial_t b_J}_{L^2_x}^2 = \int_{\Omega}\pa{\nabla_x\Delta^{-1}\partial_t b_J}\nabla_x\phi(x)\:dx.$$
We integrate by parts 
$$\norm{\partial_t\nabla_x\phi_J}_{L^2_x}^2 = \int_{\Omega} \partial_t b_J(t,x)\phi(x)\:dx.$$
At last, we use $\eqref{negativesobolevb}$
\begin{equation*}
\begin{split}
&\norm{\partial_t\nabla_x\phi_J}_{L^2_x}^2 
\\&\leq C  \cro{\norm{\mathbf{a}}_{L^2_x\pa{\boldsymbol\rho^{1/2}}}+\norm{c}_{L^2_x\pa{\boldsymbol\rho^{1/2}}}+\norm{\pi_\mathbf{L}^\bot(\mathbf{f})}_{L^2_{x,v}\pa{\boldsymbol\mu^{-1/2}}}+\norm{\mathbf{g}}_{L^2_{x,v}\pa{\boldsymbol\mu^{-1/2}}}}\norm{\nabla_x\phi}_{L^2_x} 
\\&= C  \cro{\norm{\mathbf{a}}_{L^2_x\pa{\boldsymbol\rho^{1/2}}}+\norm{c}_{L^2_x\pa{\boldsymbol\rho^{1/2}}}+\norm{\pi_\mathbf{L}^\bot(\mathbf{f})}_{L^2_{x,v}\pa{\boldsymbol\mu^{-1/2}}}+\norm{\mathbf{g}}_{L^2_{x,v}\pa{\boldsymbol\mu^{-1/2}}}}\norm{\nabla_x \Delta_x^{-1}\partial_t b_J}_{L^2_x}
\\&= C  \cro{\norm{\mathbf{a}}_{L^2_x\pa{\boldsymbol\rho^{1/2}}}+\norm{c}_{L^2_x\pa{\boldsymbol\rho^{1/2}}}+\norm{\pi_\mathbf{L}^\bot(\mathbf{f})}_{L^2_{x,v}\pa{\boldsymbol\mu^{-1/2}}}+\norm{\mathbf{g}}_{L^2_{x,v}\pa{\boldsymbol\mu^{-1/2}}}}\norm{\partial_t\nabla_x \phi_J}_{L^2_x}.
\end{split}
\end{equation*}

\par Combining this estimate with $\eqref{Psi4bstart}$ and using Young's inequality with any $\eps_b>0$
\begin{equation}\label{Psi4b}
\begin{split}
\abs{\Psi_4(t)} \leq &\eps_b \int_0^t\norm{\mathbf{a}}^2_{L^2_x\pa{\boldsymbol\rho^{1/2}}}\:ds 
\\&+ C_5(\eps_b)\int_0^t\cro{\norm{c}^2_{L^2_x\pa{\boldsymbol\rho^{1/2}}}+ \norm{\pi_\mathbf{L}^\bot(\mathbf{f})}^2_{L^2_{x,v}\pa{\boldsymbol\mu^{-1/2}}}+\norm{\mathbf{g}}^2_{L^2_{x,v}\pa{\boldsymbol\mu^{-1/2}}}}\:ds.
\end{split}
\end{equation}

\bigskip
We now gather $\eqref{RHSb}$, $\eqref{Psi1all}$, $\eqref{Psi23b}$, $\eqref{Psi4b}$ and $\eqref{Psi5b}$
\begin{equation*}
\begin{split}
\int_0^t \norm{b_J}^2_{L^2_x\pa{\boldsymbol\rho^{1/2}}}\:ds \leq & N^{(J)}_{\mathbf{f}}(t)-N^{(J)}_{\mathbf{f}}(0) + \eps_b\int_0^t \norm{a}^2_{L^2_x\pa{\boldsymbol\rho^{1/2}}}\:ds+ C_{J,c}(\eps_b)\int_0^t \norm{c}^2_{L^2_x\pa{\boldsymbol\rho^{1/2}}}
\\&+C_{J}(\eps_b)\int_0^t \cro{\norm{\mathbf{g}}^2_{L^2_{x,v}\pa{\boldsymbol\mu^{-1/2}}}+\norm{\pi_\mathbf{L}^\bot(\mathbf{f})}^2_{L^2_{x,v}\pa{\boldsymbol\mu^{-1/2}}}}\:ds.
\end{split}
\end{equation*}
Finally, summing over all $J$ in $\br{1,2,3}$
\begin{equation}\label{bfinal}
\begin{split}
\int_0^t \norm{b}^2_{L^2_x\pa{\boldsymbol\rho^{1/2}}}\:ds \leq & N^{(b)}_\mathbf{f}(t)-N^{(b)}_\mathbf{f}(0) + \eps_b\int_0^t \norm{\mathbf{a}}^2_{L^2_x\pa{\boldsymbol\rho^{1/2}}}+ C_{b,c}\int_0^t \norm{c}^2_{L^2_x\pa{\boldsymbol\rho^{1/2}}} 
\\&+C_b\int_0^t \cro{\norm{\pi_\mathbf{L}^\bot(\mathbf{f})}^2_{L^2_{x,v}}+\norm{\mathbf{g}}^2_{L^2_{x,v}\pa{\boldsymbol\mu^{-1/2}}}}\:ds,
\end{split}
\end{equation}
with $C_{b,c}$ and $C_b$ depending on $\eps_b$.


\bigskip
\textbf{Estimate for $c$.} The contribution of $c(t,x)$ is really similar to the one of $\mathbf{a}(t,x)$. Since $\mathbf{f}$ preserves mass and energy the following holds
$$\int_{\T^3} c(t,x)\:dx=0.$$
Define the test function $\boldsymbol\psi = \pa{\psi_{ic}(t,x,v)}_{1\leq i\leq N}$ with
$$\psi_{ic}(t,x,v) = \pa{\abs{v}^2-\alpha_{ic}}v\cdot\nabla_x\phi_c(t,x)$$
where
$$-\Delta_x\phi_c(t,x) = c(t,x)$$
and $\alpha_{ic}>0$ is chosen such that for all $1\leq k \leq 3$
$$\int_{\R^3} \pa{\abs{v}^2-\alpha_{ic}}v_k^2\:\mu_i(v)\:dv = 0.$$
Again, the null integral of $c$ and standard elliptic estimate \cite{Eva} show
\begin{equation}\label{phicH2}
\forall t \geq 0,\quad\norm{\phi_c(t)}_{H^2_x}\leq C_0\norm{c(t)}_{L^2_x}.
\end{equation}
Again, this estimate provides the control of $\Psi_1 = N^{(c)}_\mathbf{f}(t) - N^{(c)}_\mathbf{f}(0)$ and of $\Psi_5(t)$ as in $\eqref{Psi5a}$:
\begin{equation}\label{Psi5c}
 \abs{\Psi_5(t)} \leq \frac{C_1}{4} \int_0^t\norm{c}^2_{L^2_x\pa{\boldsymbol\rho^{1/2}}}\:ds + C_5\int_0^t \norm{\mathbf{g}}^2_{L^2_{x,v}\pa{\boldsymbol\mu^{-1/2}}}\:ds,
\end{equation}
where $C_1>0$ is given in $\eqref{RHSc}$ below.

\bigskip
We start by the left-hand side of $\eqref{eqpsi}$.
\begin{eqnarray*}
&&-\sum\limits_{i=1}^N\int_0^t \int_{\T^3\times\R^3} \pi_i(\mathbf{f}) v\cdot\nabla_x\psi_c \:dxdvds 
\\&& = -\sum\limits_{i=1}^N\sum\limits_{1\leq j,k \leq 3}\int_0^t \int_{\T^3} a_i(s,x)\pa{\int_{\R^3}\pa{\abs{v}^2-\alpha_{ic}}v_jv_km_i\mu_i(v)\:dv}\partial_{x_j}\partial_{x_k}\phi_c\:dxds
\\&&\quad -\sum\limits_{i=1}^N\sum\limits_{1\leq j,k \leq 3}\int_0^t \int_{\T^3} b(s,x)\cdot \pa{\int_{\R^3}v\pa{\abs{v}^2-\alpha_{ic}}v_jv_km_i\mu_i(v)\:dv}\partial_{x_j}\partial_{x_k}\phi_c\:dxds
\\&&\quad -\sum\limits_{i=1}^N\sum\limits_{1\leq j,k \leq 3}\int_0^t \int_{\T^3} c(s,x)\pa{\int_{\R^3}\pa{\abs{v}^2-\alpha_{ic}}\frac{\abs{v}^2-3m_i^{-1}}{2}v_jv_km_i\mu_idv}\partial_{x_j}\partial_{x_k}\phi_c.
\end{eqnarray*}
By oddity, the second integral vanishes, as well as all the others if $j\neq k$. Our choice of $\alpha_{ic}$ makes the first integral vanish even for $j=k$. It only remains the last integral with terms $j=k$ and therefore the definition of $\Delta_x\phi_c(t,x)$ gives
\begin{eqnarray}
-\sum\limits_{i=1}^N\int_0^t \int_{\T^3\times\R^3} \pi_i(\mathbf{f}) v\cdot\nabla_x\psi_c \:dxdvds &=& C_1 \int_0^t\sum\limits_{i=1}^N\int_{\T^3}\rho_i c(s,x)^2\:dxds \nonumber
\\&=& C_1 \int_0^t \norm{c(s)}_{L^2_{x}\pa{\boldsymbol\rho^{1/2}}}\:ds. \label{RHSc}
\end{eqnarray}
Again, direct computations give $\alpha_{ic}=5/m_i$ and $C_1>0$.

\bigskip
Then the terms $\Psi_2$ and $\Psi_3$ are dealt with as in $\eqref{Psi23a}$
\begin{equation}\label{Psi23c}
\forall k\in\br{2,3}, \quad \abs{\Psi_k(t)} \leq   \frac{C_1}{4}\int_0^t\norm{c}^2_{L^2_{x}\pa{\boldsymbol\rho^{1/2}}}\:ds + C_2\int_0^t\norm{\pi_\mathbf{L}^\bot(\mathbf{f})}^2_{L^2_{x,v}\pa{\boldsymbol\mu^{-1/2}}}\:ds.
\end{equation}

\bigskip
As for $\mathbf{a}(t,x)$ the estimate on $\Psi_4$ $\eqref{Psi4}$ will follow from elliptic regularity in negative Sobolev spaces. With exactly the same computations as for $\eqref{Psi4astart}$ we have
\begin{equation}\label{Psi4cstart}
\abs{\Psi_4(t)}\leq C\int_0^t\norm{\pi_\mathbf{L}^\bot(\mathbf{f})}_{L^2_{x,v}\pa{\boldsymbol\mu^{-1/2}}}\norm{\partial_t\nabla_x\phi_c}_{L^2_x}\:ds.
\end{equation}
Note that the contribution of $\pi_\mathbf{L}$ was null by oddity for the $\mathbf{a}(t,x)$ and $c(t,x)$ terms and also for the $b(t,x)$ terms thanks to our choice of $\alpha_{ic}$.
\par To estimate $\norm{\partial_t\nabla_x\phi_c}_{L^2_x}$ we use the decomposition of the weak formulation $\eqref{eqpsi}$ between $t$ and $t+\eps$ (instead of between $0$ and $t$) with 
$$\boldsymbol\psi(t,x,v) = \pa{m_i(\abs{v}^2-3m_i^{-1})\phi(x)}_{1\leq i \leq N}$$
 where $\phi$ belongs to $H^1_x$ and has a zero integral on the torus. $\psi$ does not depend on $t$ and therefore $\Psi_4(t)=0$. Moreover, multiplying by $\mu_i(v)\mu_i^{-1}(v)$ in the $i^{th}$ integral of $\Psi_3$ yields
$$\Psi_3(t)= \int_0^t \int_{\T^3} \langle \mathbf{L}(\mathbf{f}), \pa{\frac{\abs{v}^2-3m_i^{-1}}{2}m_i\mu_i}_{1\leq i\leq N}\rangle_{L^2_v\pa{\boldsymbol\mu^{-1/2}}}\partial_{x_k}\phi(x) \:dxdvds  = 0,$$
by definition of $\mbox{Ker}(\mathbf{L})$.
\par From the weak formulation $\eqref{eqpsi}$ it therefore remains
\begin{eqnarray*}
C\int_{\T^3} \cro{c(t+\eps)-c(t)}\phi(x)\:dx &=& \sum\limits_{i=1}^N\int_{t}^{t+\eps} \int_{\T^3\times\R^3}\pi_i(\mathbf{f})\frac{m_i\abs{v}^2-3}{2}v\cdot\nabla_x\phi(x)
\\&+& \sum\limits_{i=1}^N\int_{t}^{t+\eps} \int_{\T^3\times\R^3}\pi_i^\bot(\mathbf{f})\frac{m_i\abs{v}^2-3}{2}v\cdot\nabla_x\phi(x)
\\&+& \sum\limits_{i=1}^N\int_{t}^{t+\eps} \int_{\T^3\times\R^3}g_i(s,x,v)\frac{m_i\abs{v}^2-3}{2}\phi(x).
\end{eqnarray*}
As for $\mathbf{a}(t,x)$ we divide by $\eps$ and take the limit as $\eps$ goes to $0$. By oddity, the first integral on the right-hand side only gives terms with $\rho_ib(s,x)$. The last two terms are dealt with by multiplying by $\mu_i(v)^{-1/2}\mu_i(v)^{1/2}$ inside each integral and applying a Cauchy-Schwarz inequality. Note that again we also apply Poincar\'e inequality. This yields
\begin{equation*}
\begin{split}
&\abs{\int_{\T^3} \partial_t c(t,x)\phi(x)\:dx} 
\\&\quad\quad\quad\leq C \cro{\norm{b}_{L^2_{x}\pa{\boldsymbol\rho^{1/2}}}+\norm{\pi_\mathbf{L}^\bot(\mathbf{f})}_{L^2_{x,v}\pa{\boldsymbol\mu^{-1/2}}}+\norm{\mathbf{g}}_{L^2_{x,v}\pa{\boldsymbol\mu^{-1/2}}}}\norm{\nabla_x\phi}_{L^2_x}.
\end{split}
\end{equation*}

\par That estimate holds for all $\phi(x)$ in $H^1_x$ with null integral over $\T^3$. We copy the arguments made for $\mathbf{a}(t,x)$ or $b_J(t,x)$ and construct 
$$-\Delta_x\Phi_c(t,x) = \partial_tc(t,x)$$
and obtain by elliptic estimates
\begin{eqnarray*}
\norm{\partial_t\nabla_x\phi_c}_{L^2_x} &=& \norm{\nabla_x\Delta^{-1}\partial_t c}_{L^2_x} \leq  \norm{\Delta^{-1}\partial_t c}_{H^1_x} = \norm{\Phi_c}_{H^1_x}
\\ &\leq& C\norm{\partial_tc(t,x)}_{\pa{H^1_x}^*} 
\\&\leq& C \cro{\norm{b}_{L^2_{x}\pa{\boldsymbol\rho^{1/2}}}+\norm{\pi_\mathbf{L}^\bot(\mathbf{f})}_{L^2_{x,v}\pa{\boldsymbol\mu^{-1/2}}}+\norm{\mathbf{g}}_{L^2_{x,v}\pa{\boldsymbol\mu^{-1/2}}}}.
\end{eqnarray*}

\par Combining this estimate with $\eqref{Psi4cstart}$ and using Young's inequality with any $\eps_c>0$
\begin{equation}\label{Psi4c}
\abs{\Psi_4(t)} \leq \eps_c\int_0^t\norm{b}^2_{L^2_{x}\pa{\boldsymbol\rho^{1/2}}}\:ds + C_5(\eps_c)\int_0^t \cro{\norm{\pi_\mathbf{L}^\bot(\mathbf{f})}^2_{L^2_{x,v}\pa{\boldsymbol\mu^{-1/2}}}+\norm{\mathbf{g}}^2_{L^2_{x,v}\pa{\boldsymbol\mu^{-1/2}}}}\:ds.
\end{equation}

\bigskip
We now gather $\eqref{RHSc}$, $\eqref{Psi1all}$, $\eqref{Psi23c}$, $\eqref{Psi4c}$ and $\eqref{Psi5c}$ into $\eqref{eqpsi}$:
\begin{equation}\label{cfinal}
\begin{split}
\int_0^t \norm{c}^2_{L^2_{x}\pa{\boldsymbol\rho^{1/2}}}\:ds \leq & N^{(c)}_\mathbf{f}(t)-N^{(c)}_\mathbf{f}(0) + \eps_c\int_0^t \norm{b}^2_{L^2_{x}\pa{\boldsymbol\rho^{1/2}}}\:ds
\\&+ C_c(\eps_c)\int_0^t \cro{\norm{\pi_\mathbf{L}^\bot(\mathbf{f})}^2_{L^2_{x,v}\pa{\boldsymbol\mu^{-1/2}}}+\norm{\mathbf{g}}^2_{L^2_{x,v}\pa{\boldsymbol\mu^{-1/2}}}}\:ds.
\end{split}
\end{equation}


\bigskip
\textbf{Conclusion of the proof.} We gather together the estimates we derived for $\mathbf{a}$, $b$ and $c$. We compute the linear combination $\eqref{afinal} + \alpha \times \eqref{bfinal} + \beta \times \eqref{cfinal}$. For all $\eps_b >0$ and $\eps_c >0$ this implies
\begin{equation*}
\begin{split}
&\int_0^t \cro{\norm{\mathbf{a}}^2_{L^2_x\pa{\boldsymbol\rho^{1/2}}} + \alpha\norm{b}^2_{L^2_x\pa{\boldsymbol\rho^{1/2}}} + \beta \norm{c}^2_{L^2_x\pa{\boldsymbol\rho^{1/2}}}} \:ds 
\\&\leq N_\mathbf{f}(t)-N_\mathbf{f}(0) + C_\bot\int_0^t\cro{ \norm{\pi_\mathbf{L}^\bot(\mathbf{f})}^2_{L^2_{x,v}\pa{\boldsymbol\mu^{-1/2}}}+ \norm{\mathbf{g}}^2_{L^2_{x,v}\pa{\boldsymbol\mu^{-1/2}}}}\:ds
\\&\quad+\int_0^t \cro{\alpha\eps_b\norm{\mathbf{a}}^2_{L^2_x\pa{\boldsymbol\rho^{1/2}}} + \pa{C_{a,b}+\beta\eps_c}\norm{b}^2_{L^2_x\pa{\boldsymbol\rho^{1/2}}} + \alpha C_{b,c}(\eps_b) \norm{c}^2_{L^2_x\pa{\boldsymbol\rho^{1/2}}}} \:ds.
\end{split}
\end{equation*}

\par We first choose $\alpha > C_{a,b}$, then $\eps_b$ such that $\alpha\eps_b < 1$ and then $\beta > \alpha C_{b,c}(\eps_b) $. Finally, we fix $\eps_c$ small enough such that $C_{a,b}+\beta\eps_c < \alpha$ . With such choices we can absorb the last term on the right-hand side by the left-hand side. This concludes the proof of Lemma \ref{lem:controlfluidmicro} since
$$\norm{\pi_\mathbf{L}(\mathbf{f})}^2_{L^2_{x,v}\pa{\boldsymbol\mu^{-1/2}}} = \norm{\mathbf{a}}^2_{L^2_x\pa{\boldsymbol\rho^{1/2}}} + \norm{b}^2_{L^2_x\pa{\boldsymbol\rho^{1/2}}} +  \norm{c}^2_{L^2_x\pa{\boldsymbol\rho^{1/2}}}.$$

\end{proof}
\bigskip


\subsection{Generation of a $C^0$ semigroup on $L^2_{x,v}\pa{\boldsymbol\mu^{-1/2}}$}\label{subsec:semigroupL2}

We now have the tools to develop the hypocoercivity of $\mathbf{G}$ into a semigroup property.

\bigskip
\begin{proof}[Proof of Theorem \ref{theo:semigroupL2}]
Let $\mathbf{f_0}$ be in $L^2_{x,v}\pa{\boldsymbol\mu^{-1/2}}$ and consider the following equation
\begin{equation}\label{equationL2}
\partial_t \mathbf{f} = \mathbf{L}\pa{\mathbf{f}} - v\cdot\nabla_x \mathbf{f}
\end{equation}
with initial data $\mathbf{f_0}$.
\par Since the transport part $-v\cdot\nabla_x$ is skew-symmetric in $L^2_{x,v}\pa{\mathbf{\mu}_i^{-1/2}}$ (mere integration by part) and $\mathbf{L}$ is self-adjoint, $\mbox{Ker}(\mathbf{G})$ and $\pa{\mbox{Ker}(\mathbf{G})}^\bot$ are stable under $\eqref{equationL2}$. We therefore consider only the case $\mathbf{f_0}$ in $\pa{\mbox{Ker}(\mathbf{G})}^\bot$ and the associated solution stays in $\pa{\mbox{Ker}(\mathbf{G})}^\bot$ for all $t$.
\par Moreover, $\mathbf{L}$ has a spectral gap $\lambda_L$ and so by Theorem \ref{theo:spectralgapL}, if $\mathbf{f}=\pa{f_i}_{1\leq i \leq N}$ is a solution to $\eqref{equationL2}$ we have the following
\begin{eqnarray}
\frac{1}{2}\frac{d}{dt}\norm{\mathbf{f}}^2_{L^2_{x,v}\pa{\boldsymbol\mu^{-1/2}}} &=& \int_{\T^3 \times \R^3} \langle\partial_t \mathbf{f},\mathbf{f} \rangle_{\boldsymbol\mu^{-1/2}} \:dxdv \nonumber
\\&=&  -\sum\limits_{i=1}^N \int_{\T^3\times\R^3} v\cdot\nabla_x\pa{f_i(t,x,v)^2}\mu_i^{-1}(v)\:dxdv \nonumber
\\&&+ \int_{\T^3} \langle \mathbf{L}(\mathbf{f})(t,x,\cdot),\mathbf{f}(t,x,\cdot) \rangle_{L^2_v\pa{\mu^{-1/2}}}\:dx \nonumber
\\&\leq &  -\lambda_L \norm{\pi_\mathbf{L}^\bot(\mathbf{f})}^2_{L^2_{x,v}\pa{\boldsymbol\mu^{-1/2}}}. \label{ineqL2}
\end{eqnarray}
We remind that $\pi_\mathbf{L}^\bot=\mbox{Id}-\pi_\mathbf{L}$ where $\pi_\mathbf{L}$ is the orthogonal projection $\eqref{piL}$ onto $\mbox{Ker}(L)$ in $L^2_v\pa{\boldsymbol\mu^{-1/2}}$.  The norm is thus decreasing under the flow and it therefore follows that $\mathbf{G}$ generates a strongly continuous semigroup on $L^2_v\pa{\boldsymbol\mu^{-1/2}}$, we refer the reader to \cite{Ka} (general theory) or \cite{Uk}\cite{UkYa} (for the special case of single species Boltzmann equation).

\bigskip
Let $\mathbf{f} = S_{\mathbf{G}}(t)\mathbf{f_0}$ and define $\mathbf{\tilde{f}}(t,x,v) = e^{\lambda t}\mathbf{f}(t,x,v)$ for $\lambda > 0$ to be defined later. $\mathbf{\tilde{f}}$ satisfies the conservation laws and is solution in $L^2_{x,v}\pa{\boldsymbol\mu^{-1/2}}$ to the following equation
$$\partial_t \mathbf{\tilde{f}} = \mathbf{G}(\mathbf{\tilde{f}}) + \lambda\mathbf{\tilde{f}}.$$

\par As for $\eqref{ineqL2}$ we obtain the following estimate
\begin{equation}\label{ineqL2tilde}
\norm{\mathbf{\tilde{f}}}^2_{L^2_{x,v}\pa{\boldsymbol\mu^{-1/2}}} \leq \norm{\mathbf{f_0}}^2_{L^2_{x,v}\pa{\boldsymbol\mu^{-1/2}}} - 2\lambda_L\int_0^t \norm{\pi_\mathbf{L}^\bot(\mathbf{\tilde{f}})}^2_{L^2_{x,v}\pa{\boldsymbol\mu^{-1/2}}}+ 2\lambda\int_0^t \norm{\mathbf{\tilde{f}}}^2_{L^2_{x,v}\pa{\boldsymbol\mu^{-1/2}}}.
\end{equation}

\par Along with the latter estimate, we have the following control given by Lemma \ref{lem:controlfluidmicro} with $\mathbf{g}=\lambda\mathbf{\tilde{f}}$
\begin{equation}\label{controlpiLpiLbot}
\begin{split}
\int_0^t\norm{\pi_\mathbf{L}(\mathbf{\tilde{f}})}^2_{L^2_{x,v}\pa{\boldsymbol\mu^{-1/2}}}\:ds \leq & N_\mathbf{\tilde{f}}(t)-N_\mathbf{\tilde{f}}(0) + C_\bot \int_0^t\norm{\pi^\bot_\mathbf{L}(\mathbf{\tilde{f}})}^2_{L^2_{x,v}\pa{\boldsymbol\mu^{-1/2}}}\:ds
\\&+ C_\bot\lambda^2 \int_0^t\norm{\mathbf{\tilde{f}}}^2_{L^2_{x,v}\pa{\boldsymbol\mu^{-1/2}}}\:ds
\end{split}
\end{equation}
where $C_\bot>0$ is independent of $\mathbf{f}$ and $\abs{N_\mathbf{\tilde{f}}(s)} \leq C\norm{\mathbf{\tilde{f}}(s)}^2_{L^2_{x,v}\pa{\boldsymbol\mu^{-1/2}}}$, then $\eps \times \eqref{controlpiLpiLbot}+\eqref{ineqL2tilde}$ yields
\begin{eqnarray*}
&&\cro{\norm{\mathbf{\tilde{f}}}^2_{L^2_{x,v}\pa{\boldsymbol\mu^{-1/2}}}-\eps N_\mathbf{\tilde{f}}(t)} + C_{\eps}\int_0^t \pa{\norm{\pi_\mathbf{L}(\mathbf{\tilde{f}})}^2_{L^2_{x,v}\pa{\boldsymbol\mu^{-1/2}}} + \norm{\pi_\mathbf{L}^\bot (\mathbf{\tilde{f}})}^2_{L^2_{x,v}\pa{\boldsymbol\mu^{-1/2}}}}
\\&&\quad\quad\quad \leq \norm{\mathbf{f_0}}^2_{L^2_{x,v}\pa{\boldsymbol\mu^{-1/2}}}-\eps N_{\tilde{f}}(0) + \pa{2\lambda + \eps C_\bot \lambda^2}\int_0^t \norm{\mathbf{\tilde{f}}}^2_{L^2_{x,v}\pa{\boldsymbol\mu^{-1/2}}}\:ds.
\end{eqnarray*}
where $C_\eps = \min\br{2\lambda_L-\eps C_\bot, \eps}$. By the control on $\abs{N_\mathbf{\tilde{f}}(s)}$ and the fact that 
$$\norm{\pi_\mathbf{L}(\mathbf{\tilde{f}})}^2_{L^2_{x,v}\pa{\boldsymbol\mu^{-1/2}}} + \norm{\pi_\mathbf{L}^\bot (\mathbf{\tilde{f}})}^2_{L^2_{x,v}\pa{\boldsymbol\mu^{-1/2}}} = \norm{\mathbf{\tilde{f}}}^2_{L^2_{x,v}\pa{\boldsymbol\mu^{-1/2}}}$$
we can choose $\eps$ small enough such that $C_\eps >0$ and then $\lambda$ small enough such that $\pa{2\lambda + \eps C_\bot \lambda^2} < C_\eps$. Such choices imply that $\norm{\mathbf{\tilde{f}}}^2_{L^2_{x,v}\pa{\boldsymbol\mu^{-1/2}}}$ is uniformly bounded in time by $C\norm{\mathbf{f_0}}^2_{L^2_{x,v}\pa{\boldsymbol\mu^{-1/2}}}$.
\par By definition of $\mathbf{\tilde{f}}$, this shows an exponential decay for $\mathbf{f}$ and concludes the proof of Theorem \ref{theo:semigroupL2}.

\end{proof}
\bigskip

\section{$L^\infty$ theory for the linear part with maxwellian weight}\label{sec:Linftytheory}

As explained in the introduction, the $L^2$ setting is not algebraic for the nonlinear operator $\mathbf{Q}$. We therefore need to work in an $L^\infty$ framework. We first give a pointwise control on the linear operator $\mathbf{K}$ in Subsection \ref{subsec:keyestimatesoperator} and then we prove that the linear part of the perturbed equation $\eqref{perturbedmultiBE}$ generates a strongly continuous semigroup in $L^\infty_{x,v}\pa{\langle v \rangle^\beta\boldsymbol\mu^{-1/2}}$ in Subsection \ref{subsec:Linftylinear}.

\bigskip


\subsection{Pointwise estimate on $\mathbf{K}$}\label{subsec:keyestimatesoperator}

We recall that $\mathbf{L}$ can be written under the following form
$$\mathbf{L} = -\boldsymbol\nu(v) + \mathbf{K},$$
where $\boldsymbol\nu = \pa{\nu_i}_{1\leq i\leq N}$ is a multiplicative operator satisfying $\eqref{nu0nu1}$:
$$\forall v \in \R^3,\quad \nu_i^{(0)}(1+\abs{v}^\gamma)\leq \nu_i(v)\leq \nu_i^{(1)}(1+\abs{v}^\gamma),$$
with $\nu_i^{(0)},\:\nu_i^{(1)} >0$.

\par In the case of single-species Boltzmann equation, the operator $\mathbf{K}$ can be written as a kernel operator (\cite{Gr2} or \cite{CIP} Section 7.2) and we give here a similar property where the different exponential decay rates, due to the different masses, are explicitely taken into account. These explicit bounds will be strongly needed for the $L^\infty$ theory.

\bigskip
\begin{lemma}\label{lem:kernelK}
Let $\mathbf{f}$ be in $L^2_v\pa{\boldsymbol\mu^{-1/2}}$. Then for all $i$ in $\br{1,\dots,N}$ there exists $\mathbf{k}^{(i)}$ such that
$$K_i(\mathbf{f})(v) = \int_{\R^3}\langle \mathbf{k}^{(i)}(v,v_*), \mathbf{f}(v_*)\rangle\:dv_*.$$
Moreover there exist $m$, $C_K >0$ such that for all $i$ in $\br{1,\dots,N}$ and for all $1\leq j \leq N$
\begin{equation}\label{pointwiseki}
\abs{k^{(i)}_j(v,v_*)} \leq C_K\sqrt{\frac{\mu_i(v)}{\mu_j(v_*)}} \cro{\abs{v-v_*}^\gamma + \abs{v-v_*}^{\gamma-2}}e^{-m\abs{v-v_*}^2-m\frac{\abs{\abs{v}^2-\abs{v_*}^2}^2}{\abs{v-v_*}^2}}. 
\end{equation}
The constants $m$ and $C_K$ are explicit and depend only on $\pa{m_i}_{1\leq i \leq N}$ and the collision kernel $B$.
\end{lemma}
\bigskip

\begin{proof}[Proof of Lemma \ref{lem:kernelK}]
By definition, $\mathbf{K} = \pa{K_i}_{1\leq i \leq N}$ with
\begin{equation}\label{Kilemma}
K_i(\mathbf{f})(v) = \sum\limits_{j=1}^N \int_{\mathbb{S}^2\times\R^3}B_{ij}\pa{\abs{v-v_*},\cos\theta}\cro{\mu_j^{'*}f_i' + \mu_i'f_j^{'*} - \mu_i f_j^*}\:d\sigma dv_*.
\end{equation}
We used the identity $\mu_i(v)\mu_j(v_*) = \mu_i(v')\mu_j(v'_*)$ that is a consequence of the conservation of energy during an elastic collision.

\bigskip
\textbf{Step 1: A kernel form.}
The third term in the integral is already in the desired form. The first two terms require a new representation of the collision kernel where the integrand parameters will be $v'$ and $v'_*$ instead of $v_*$  and $\sigma$. Such a representation has been obtained in the case of a single-species Boltzmann equation and is called the Carleman representation \cite{Ca2}. We derive below the Carleman representation associated with the multi-species Boltzmann operator. We follow the methods used in \cite{Ca2}\cite{CIP} Section 7.2 and \cite{GPV}. However, the existence of different masses generates an asymmetry between $v'$ and $v'_*$ as we shall see.

\bigskip
The laws of elastic collisions gives
$$v'= V + \frac{m_j}{m_i+m_j}\abs{v-v_*}\sigma \quad\mbox{and}\quad v'_* = V - \frac{m_i}{m_i+m_j}\abs{v-v_*}\sigma$$
where $V$ is the center of mass of the particles $i$ and $j$:
$$V = \frac{m_i}{m_i+m_j}v + \frac{m_j}{m_i+m_j}v_*.$$
We can also express 
$$v = V +\frac{m_j}{m_i+m_j}\pa{v-v_*} \quad\mbox{and}\quad v_* = V - \frac{m_i}{m_i+m_j}\pa{v-v_*}.$$
Note that 
\begin{eqnarray}
\abs{v-v_*} &=& \abs{v'-v'_*},\label{v-v*}
\\ \abs{v-v'} &\leq& \frac{2m_j}{m_i+m_j}\abs{v'-v'_*}, \label{v-v'}
\\ \abs{v-v'_*} &\leq& \abs{v'-v'_*}. \label{v-v'*}
\end{eqnarray}

\bigskip
The points $v$, $v_*$, $v'$ and $v'_*$ therefore belong to the plane defined by $V$ and $\mbox{Span}(\sigma, v-v_*)$. We have the following geometric configuration, which gives a perfect circle in the case of equal masses.

\begin{figure}[!h]
\begin{center}
\includegraphics[scale=0.5]{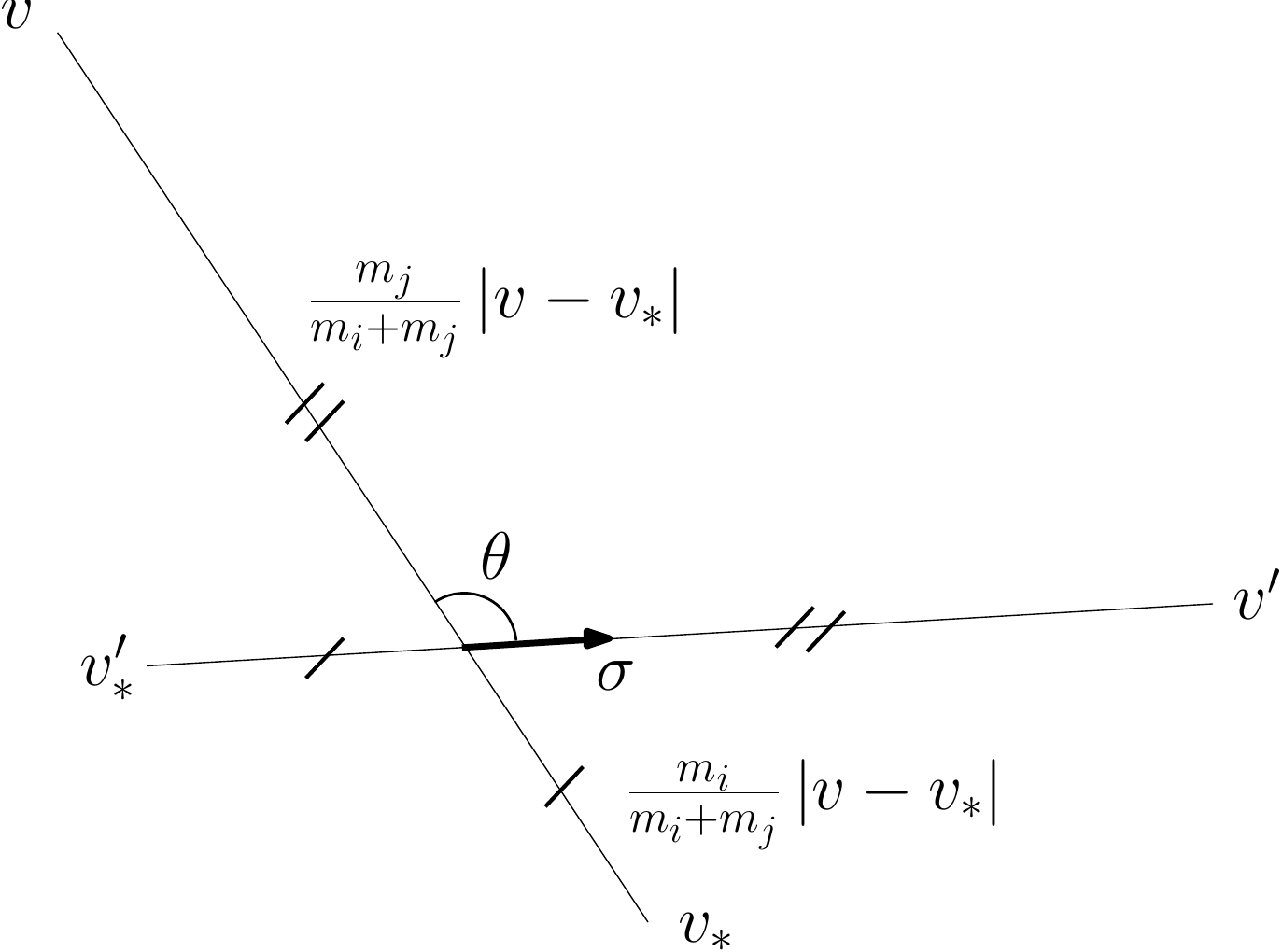}
\end{center}
\caption{\footnotesize Relation between pre-collisional and post-collisional velocities}
\label{fig:rolling}
\end{figure}

Geometrically, $m_j^{-1}(v-V)$, $m_j^{-1}(v'-V)$, $m_i^{-1}(v_*-V)$ and $m_i^{-1}(v'_*-V)$ are on the same circle of diameter $\left|m_i^{-1}(v'^*-V)-m_j(v'-V)\right|=\frac{2}{m_i+m_j}|v-v_*|$. Therefore,
$$\left\langle \frac{1}{m_i}(v'_*-V) - \frac{1}{m_j}(v-V), \frac{1}{m_j}(v-V)-\frac{1}{m_j}(v'-V)   \right\rangle = 0.$$
Using the laws of elasticity to see that
$$V = \frac{m_i}{m_i+m_j}v' + \frac{m_j}{m_i+m_j}v'_*$$
we end up with the following orthogonal property (that is also easily checked by direct computations)
\begin{equation}\label{orthogonal}
\left\langle v'_* - \pa{\frac{m_i+m_j}{2m_j}v - \frac{m_i-m_j}{2m_j}v'} , v-v' \right\rangle =0.
\end{equation}

\par We can now apply the change of variables $(v_*,\sigma) \mapsto (v',v'_*)$, where $v'$ evolves in $\R^3$ and $v'_*$ in $E^{ij}_{vv'}$. $E^{ij}_{vv'}$ is the hyperplane that passes through 
\begin{equation}\label{VE}
V_E(v,v') = \frac{m_i+m_j}{2m_j}v - \frac{m_i-m_j}{2m_j}v'
\end{equation}
 and is orthogonal to $v-v'$; we denote $dE(v'_*)$ the Lebesgue measure on it. Note that $v_* = V(v',v'_*)$ is now a function of $v'$ and $v'_*$:
$$V(v',v'_*) = v'_* + m_im_j^{-1}v' - m_im_j^{-1}v.$$
Up to the translation and dilatation (generating a constant $C_{ij}>0$ only depending on $m_i$ and $m_j$) from $v$ to the origin of $E^{ij}_{vv'}$, this change of variables works as derived in \cite{GPV}. Our operator thus reads
\begin{equation} \label{carlemanv'*functionv'}
\begin{split}
&\int_{\R^3\times \mathbb{S}^{2}}  B(v-v_*,\sigma)f'g^{'*}\:dv_*d\sigma 
\\&\quad\quad\quad= C_{ij}\int_{\R^3}\frac{1}{\abs{v-v'}}\pa{\int_{E^{ij}_{vv'}}\frac{B\left(v-V(v',v'_*),\frac{v'_*-v'}{\abs{v'_*-v'}}\right)}{\abs{v'_*-v'}}g^{'*}\:dE(v'_*)}f'\:dv'.
\end{split}
\end{equation}

\bigskip
We can also give a Carleman representation where we first integrate against $v'_*$. In the case $m_i=m_j$ the orthogonal property $\eqref{orthogonal}$ is entirely symmetric in $v'$ and $v'_*$ and we reach the same representation $\eqref{carlemanv'*functionv'}$ with the role of $v'$ and $v'_*$ swapped. This is the classical case of a single-species Boltzmann operator.
\par In the case $m_i \neq m_j$, $\eqref{orthogonal}$ is equivalent to
$$\abs{v'}^2 - 2\left\langle v' , \frac{m_i}{m_i-m_j}v -\frac{m_j}{m_i-m_j}v'_*\right\rangle = \left\langle v,\frac{2m_j}{m_i-m_j}v'_* - \frac{m_i+m_j}{m_i-m_j}v \right\rangle$$
which is itself equivalent to
\begin{equation}\label{tildeEijvv'*}
\abs{v'+\pa{\frac{m_j}{m_i-m_j}v'_* - \frac{m_i}{m_i-m_j}v}}^2= \abs{\frac{m_j}{m_i-m_j}v'_* - \frac{m_j}{m_i-m_j}v}^2.
\end{equation}
\par The same change of variables as before but $(v_*,\sigma) \mapsto (v'_*,v')$ instead of $(v_*,\sigma) \mapsto (v',v'_*)$ thus yields
\begin{equation} \label{carlemanv'functionv'*}
\begin{split}
&\int_{\R^3\times \mathbb{S}^{2}}  B(v-v_*,\sigma)f'g^{'*}\:dv_*d\sigma 
\\&\quad\quad\quad= C_{ij}\int_{\R^3}\frac{1}{\abs{v-v'_*}}\pa{\int_{\tilde{E}^{ij}_{vv'_*}}\frac{B\left(v-V(v',v'_*),\frac{v'_*-v'}{\abs{v'_*-v'}}\right)}{\abs{v'_*-v'}}f'\:dE(v')}g^{'*}\:dv'_*,
\end{split}
\end{equation}
where $\tilde{E}^{ij}_{vv'_*}$ stands for $E^{ij}_{vv'_*}$ if $m_i=m_j$ or for the sphere defined by $\eqref{tildeEijvv'*}$; and $dE$ is the Lebesgue measure on it.

\bigskip
We therefore conclude gathering $\eqref{Kilemma}$, $\eqref{carlemanv'*functionv'}$ and $\eqref{carlemanv'functionv'*}$ with a relabelling of the integrated variables,
\begin{equation}\label{Kikernel}
\begin{split}
K_i(\mathbf{f})(v) &= \sum\limits_{j=1}^NC_{ij}\int_{\R^3} \pa{\frac{1}{\abs{v-v_*}}\int_{\tilde{E}^{ij}_{vv_*}}\frac{B_{ij}\left(v-V(u,v_*),\frac{v_*-u}{\abs{u-v_*}}\right)}{\abs{u-v_*}}\mu_i(u)\:dE(u)}f_j^*\:dv_*
\\& + \sum\limits_{j=1}^NC_{ji}\int_{\R^3} \pa{\frac{1}{\abs{v-v_*}}\int_{E^{ij}_{vv_*}}\frac{B_{ij}\left(v-V(v_*,u),\frac{u-v_*}{\abs{u-v_*}}\right)}{\abs{u-v_*}}\mu_j(u)\:dE(u)}f_i^*\:dv_*
\\&- \sum\limits_{j=1}^N\int_{\R^3}B_{ij}\pa{\abs{v-v_*},\cos\theta}\mu_i(v) f_j^*\:dv_*.
\end{split}
\end{equation}
This concludes the fact that $K_i$ is a kernel operator.

\bigskip
\textbf{Step 2: Pointwise estimate.}
It remains to show the pointwise estimate $\eqref{pointwiseki}$. The assumptions on $B_{ij}$ imply that
$$\abs{B_{ij}\left(v-V(v_*,u),\frac{u-v_*}{\abs{u-v_*}}\right)} \leq C \abs{v-V(v_*,u)}^\gamma,$$
where $C$ denotes any positive constant independent of $v$ and $v_*$. We shall bound each of the three terms in $\eqref{Kikernel}$ separately.

\bigskip
From elastic collision laws $\eqref{v-v*}$, for $u$ in $E^{ij}_{vv_*}$ one has $\abs{v-V(v_*,u)} = \abs{u-v_*}$, and hence
$$\abs{\int_{E^{ij}_{vv_*}}\frac{B_{ij}\left(v-V(v_*,u),\frac{u-v_*}{\abs{u-v_*}}\right)}{\abs{u-v_*}}\mu_j(u)\:dE(u)} \leq C\int_{E^{ij}_{vv_*}}\frac{1}{\abs{u-v_*}^{1-\gamma}}e^{-m_j\frac{\abs{u}^2}{2}}\:dE(u).$$
We can further bound, since $\eqref{v-v'}$ is valid on $E^{ij}_{vv_*}$,
$$\abs{u-v_*} \geq  \frac{m_i+m_j}{2m_j}\abs{v-v_*},$$
and get
$$\abs{\int_{E^{ij}_{vv_*}}\frac{B_{ij}\left(v-V(v_*,u),\frac{u-v_*}{\abs{u-v_*}}\right)}{\abs{u-v_*}}\mu_j(u)\:dE(u)} \leq \frac{C}{\abs{v-v_*}^{1-\gamma}}\int_{E_{vv*}}e^{-m_j\frac{\abs{u}^2}{2}}\:dE(u).$$

\par To estimate the integral over $E^{ij}_{vv_*}$ we make the change of variables 
$$u=V_E(v,v_*)+w$$
 with $V_E(v,v_*)$ the origin $\eqref{VE}$ of $E^{ij}_{vv_*}$ and $w$ in $\pa{\mbox{Span}(v-v_*)}^\bot$. Using $\langle v,w\rangle = \langle v_*,w\rangle$ we compute
\begin{eqnarray*}
\abs{u}^2=\abs{V_E(v,v_*)+w}^2 &=& \abs{w+\frac{1}{2}(v+v_*) + \frac{m_i}{2m_j}\pa{v-v_*}}^2
\\& =& \abs{w+\frac{1}{2}(v+v_*)}^2 + \frac{m_i^2}{4m_j^2}\abs{v-v_*}^2 +\frac{m_i}{2m_j}\pa{\abs{v}^2-\abs{v_*}^2}.
\end{eqnarray*}
Now we decompose $v+v_* = V^{\bot} + V^{\parallel}$ where $V^{\parallel}$ is the projection onto $\mbox{Span}(v-v_*)$ and $V^{\bot}$ is the orthogonal part. This implies
$$\abs{u}^2 = \abs{w+\frac{1}{2}V^\bot}^2 + \frac{1}{4}\abs{V^\parallel}^2 + \frac{m_i^2}{4m_j^2}\abs{v-v_*}^2 +\frac{m_i}{2m_j}\pa{\abs{v}^2-\abs{v_*}^2}.$$
By definition,
$$\abs{V^\parallel}^2 = \frac{\langle v+v_*,v-v_*\rangle^2}{\abs{v-v_*}^2}= \frac{\abs{\abs{v}^2-\abs{v_*}^2}^2}{\abs{v-v_*}^2}$$
and therefore the following holds
\begin{equation}\label{1stineqki}
\begin{split}
&\abs{\frac{1}{\abs{v-v_*}}\int_{E^{ij}_{vv_*}}\frac{B_{ij}\left(v-V(v_*,u),\frac{u-v_*}{\abs{u-v_*}}\right)}{\abs{u-v_*}}\mu_j(u)\:dE(u)}
\\&\quad \leq  \frac{C}{\abs{v-v_*}^{2-\gamma}}e^{-\frac{m_i^2}{8m_j}\abs{v-v_*}^2-\frac{m_j}{8}\frac{\abs{\abs{v}^2-\abs{v_*}^2}^2}{\abs{v-v_*}^2}}\sqrt{\frac{\mu_i(v)}{\mu_i(v_*)}}\cro{\int_{\pa{v-v_*}^\bot}e^{-\frac{m_j}{2}\abs{w+\frac{1}{2}V^\bot}^2}\:dE(w)}.
\end{split}
\end{equation}
The space $(v-v_*)^\bot$ is invariant by translation of vector $-2^{-1}V^\bot$ and the exponential term inside the integral only depends on the norm and therefore the integral term is a constant not depending on $v$ or $v_*$.

\bigskip
We now turn to the term involving $\tilde{E}^{ij}_{vv'}$ which is a bit more technical. In the case $m_i=m_j$ then $\tilde{E}^{ij}_{vv'} = E^{ij}_{vv'}$. We therefore have the bound $\eqref{1stineqki}$ to which we use $\mu_i(v)\mu_i^{-1}(v) = C_{ij}\mu_j(v)\mu_i^{-1}(v)$ since $m_i=m_j$.
\par Assume now that $m_i\neq m_j$. As for $E^{ij}_{vv_*}$, the elastic collision properties $\eqref{v-v*}$ and $\eqref{v-v'*}$ give for all $v_*$ in $\R^3$ and $u$ in $\tilde{E}^{ij}_{vv_*}$
$$\abs{\int_{\tilde{E}^{ij}_{vv_*}}\frac{B_{ij}\left(v-V(u,v_*),\frac{v_*-u}{\abs{u-v_*}}\right)}{\abs{u-v_*}}\mu_i(u)\:dE(u)} \leq \frac{C}{\abs{v-v_*}^{1-\gamma}}\int_{\tilde{E}^{ij}_{vv_*}}e^{-m_i\frac{\abs{u}^2}{2}}\:dE(u).$$
Since $\tilde{E}^{ij}_{vv_*}$ is the sphere of radius 
$$R_{vv_*} = \frac{m_j}{\abs{m_i-m_j}}\abs{v-v_*}$$
and centered at
$$O_{vv_*} = \frac{m_i}{m_i-m_j}v -\frac{m_j}{m_i-m_j}v_*.$$
We make a change of variables to end up on $\mathbb{S}^2$:
\begin{equation}\label{tildeES2}
\begin{split}
&\abs{\frac{1}{\abs{v-v_*}}\int_{\tilde{E}^{ij}_{vv_*}}\frac{B_{ij}\left(v-V(u,v_*),\frac{v_*-u}{\abs{u-v_*}}\right)}{\abs{u-v_*}}\mu_i(u)\:dE(u)} 
\\&\quad\quad\quad\quad\quad\leq C\abs{v-v_*}^\gamma\int_{\mathbb{S}^2}e^{-\frac{m_i}{2}\abs{R_{vv_*}u + O_{vv_*}}^2}\:d\sigma(u).
\end{split}
\end{equation}

Decomposing the norm inside the integral and using Cauchy-Schwarz inequality yields
\begin{equation}\label{normUstart}
\begin{split}
-\frac{m_i}{2}\abs{R_{vv_*}u + O_{vv_*}}^2 \leq& -\frac{m_im_j^2}{2(m_i-m_j)^2}\abs{v-v_*}^2 - \frac{m_i}{2(m_i-m_j)^2}\abs{m_iv-m_jv_*}^2 
\\&+\frac{m_im_j}{\pa{m_i-m_j}^2}\abs{v-v_*}\abs{m_iv-m_jv_*}
\end{split}
\end{equation}

The idea is to express everything in terms of $\abs{v-v_*}$ and $\frac{\abs{v}^2-\abs{v_*}^2}{\abs{v-v_*}}$. We recall that we defined $v+v_* = V^\bot+V^\parallel$ with $V^\bot$ orthogonal to $\mbox{Span}(v-v_*)$ and $V^\parallel = \frac{\langle v+v_*, v-v_*\rangle}{\abs{v-v_*}}(v-v_*)$. We first use the identity
$$\abs{v-v_*}\abs{m_iv-m_jv_*} = \frac{1}{4}\cro{\abs{(1+m_i)v -(1+m_j)v_*}^2-\abs{(1-m_i)v -(1-m_j)v_*}^2}$$
and then the following equality that holds for all $a$ and $b$,
\begin{eqnarray}
\abs{av-bv_*}^2 &=& \abs{\frac{a-b}{2}(v+v_*)+\frac{a+b}{2}(v-v_*)}^2 \nonumber
\\ &=& \frac{(a-b)^2}{4}\abs{V^\bot}^2 + \frac{(a-b)^2}{4}\frac{\abs{\abs{v}^2-\abs{v_*}^2}^2}{\abs{v-v_*}^2} \label{avbv*}
\\&\quad& + \frac{(a-b)(a+b)}{2}\pa{\abs{v}^2-\abs{v_*}^2} + \frac{(a+b)^2}{4}\abs{v-v_*}^2.\nonumber
\end{eqnarray}
Direct computations from $\eqref{normUstart}$ then yield
\begin{equation*}
\begin{split}
-\frac{m_i}{2}\abs{R_{vv_*}u + O_{vv_*}}^2 \leq& -\frac{m_i}{8}\abs{V^\bot}^2 -\frac{m_i}{8}\frac{\abs{\abs{v}^2-\abs{v_*}^2}^2}{\abs{v-v_*}^2} - \frac{m_i}{4}\pa{\abs{v}^2-\abs{v_*}^2}
\\& -\frac{m_i}{8}\abs{v-v_*}^2.
\end{split}
\end{equation*}

\par Taking $(a,b)=(1,0)$ and $(a,b)=(0,1)$ in $\eqref{avbv*}$ we have
\begin{equation*}
\begin{split}
\frac{m_i}{4}\abs{v}^2 - \frac{m_j}{4}\abs{v_*}^2 =& \frac{m_i-m_j}{16}\abs{V^\bot}^2 + \frac{m_i-m_j}{16}\frac{\abs{\abs{v}^2-\abs{v_*}^2}^2}{\abs{v-v_*}^2}
\\& + \frac{m_i+m_j}{8}\pa{\abs{v}^2-\abs{v_*}^2} + \frac{m_i-m_j}{16}\abs{v-v_*}^2.
\end{split}
\end{equation*}
At last we obtain
\begin{equation}\label{normU}
\begin{split}
-\frac{m_i}{2}\abs{R_{vv_*}u + O_{vv_*}}^2 \leq& -\frac{m_i}{4}\abs{v}^2 + \frac{m_j}{4}\abs{v_*}^2 - \frac{m_i+m_j}{16}\abs{V^\bot}^2 
\\&+ U\pa{\frac{\abs{v}^2-\abs{v_*}}{\abs{v-v_*}},\abs{v-v_*}}
\end{split}
\end{equation}
where $U(x,y)$ is a quadratic form defined by
$$U(x,v) = -\frac{m_i+m_j}{16}x^2 -\frac{m_i-m_j}{8}xy - \frac{m_i+m_j}{16}y^2.$$
The latter quadratic form is associated with the symmetric matrix
$$\left(\begin{array}{cc}  \disp{-\frac{m_i+m_j}{16}} & \disp{-\frac{m_i-m_j}{16}} \\ \disp{-\frac{m_i-m_j}{16}}& \disp{-\frac{m_i+m_j}{16}} \end{array}\right)$$
which has a negative trace and determinant $m_im_j/64 >0$. It therefore is a negative definite symmetric matrix and thus, denoting by $-\lambda(m_i,m_j)<0$ its largest eigenvalue we have
$$\forall (x,x)\in \R^2, \quad U(x,v)\leq -\lambda(m_i,m_j)\cro{x^2+y^2}.$$
Plugging the latter into $\eqref{normU}$ and going back to the integral of interest $\eqref{tildeES2}$ we get
\begin{equation}\label{2ndineqki}
\begin{split}
&\abs{\frac{1}{\abs{v-v_*}}\int_{\tilde{E}^{ij}_{vv_*}}\frac{B_{ij}\left(v-V(u,v_*),\frac{v_*-u}{\abs{u-v_*}}\right)}{\abs{u-v_*}}\mu_i(u)\:dE(u)} 
\\&\quad\quad\quad\quad\quad\leq C\abs{v-v_*}^\gamma e^{-\lambda(m_i,m_j)\abs{v-v_*}^2-\lambda(m_i,m_j)\frac{\abs{\abs{v}^2-\abs{v_*}^2}^2}{\abs{v-v_*}^2}}\sqrt{\frac{\mu_i(v)}{\mu_j(v_*)}}.
\end{split}
\end{equation}

\bigskip
To conclude we turn to the last integral term in $\eqref{Kikernel}$ which is easily bounded by
\begin{equation*}
\begin{split}
\abs{B_{ij}\pa{\abs{v-v_*},\cos\theta}\mu_i(v)} &\leq C \abs{v-v_*}^\gamma \mu_i(v)
\\&\leq C \abs{v-v_*}^\gamma e^{-\frac{1}{4}\pa{m_i\abs{v}^2+m_j\abs{v_*}^2}}\sqrt{\frac{\mu_i(v)}{\mu_j(v_*)}}.
\end{split}
\end{equation*}
Using Cauchy-Schwartz
\begin{equation*}
\begin{split}
\abs{v-v_*}^2 + \frac{\abs{\abs{v}^2-\abs{v_*}^2}^2}{\abs{v-v_*}^2}&=\abs{v-v_*}^2 + \frac{\abs{\langle v-v_*,v+v_* \rangle}^2}{\abs{v-v_*}^2} 
\\&\leq \abs{v-v_*}^2+\abs{v+v_*}^2 = 2\pa{\abs{v}^2+\abs{v_*}^2},
\end{split}
\end{equation*}
this implies
\begin{equation}\label{3rdineqki}
\abs{B_{ij}\pa{\abs{v-v_*},\cos\theta}\mu_i(v)} \leq C\abs{v-v_*}^\gamma e^{-\frac{m_{ij}}{8}\abs{v-v_*}^2 - \frac{m_{ij}}{8}\frac{\abs{\abs{v}^2-\abs{v_*}^2}^2}{\abs{v-v_*}^2}}\sqrt{\frac{\mu_i(v)}{\mu_j(v_*)}},
\end{equation}
where $m_{ij}=\min\br{m_i,m_j}$.

\bigskip
Gathering $\eqref{Kikernel}$-$\eqref{1stineqki}$-$\eqref{2ndineqki}$-$\eqref{3rdineqki}$ gives the desired estimate on $k_j^{(i)}$.
\end{proof}
\bigskip

The pointwise estimate on $k^{(i)}_j$ can be transferred into a decay of the $L^1_v$-norm with a relatively important weight. This has been proved in \cite[Lemma 7]{Gu6} for the right-hand side of $\eqref{pointwiseki}$ with $m=1/8$. The case of general $m$ is identical and leads to

\bigskip
\begin{lemma}\label{lem:L1ki}
Let $\beta >0$ and $\theta$ in $[0,1/(32m))$. There exists $C_{\theta,\beta} >0$ and $\eps_{\theta,\beta}>0$ such that for all $i$, $j$ in $\br{1,\dots,N}$ and all $\eps$ in $[0,\eps_{\theta,\beta})$,
$$ \int_{\R^3} \abs{k^{(i)}_j(v,v_*)}e^{\eps m \abs{v-v_*}^2+\eps m \frac{\abs{\abs{v}^2-\abs{v_*}^2}^2}{\abs{v-v_*}^2}}\frac{\langle v \rangle^\beta e^{\theta\abs{v}^2}\mu_i(v)^{-1/2}}{\langle v_* \rangle^\beta e^{\theta\abs{v_*}^2}\mu_j(v_*)^{-1/2}}\:dv_* \leq \frac{C_{\beta,\theta}}{1+\abs{v}}.$$
\end{lemma}
\bigskip

From Lemma \ref{lem:kernelK} and \ref{lem:L1ki} we conclude that $\mathbf{K}$ is a bounded operator on $L^\infty_{v}\pa{\langle v\rangle^\beta\boldsymbol\mu^{-1/2}}$.

\bigskip


\subsection{Semigroup generated by the linear part}\label{subsec:Linftylinear}

\bigskip
Following ideas developed in \cite{Gu6} in the case of bounded domains, the $L^2$ theory could be used to construct a $L^\infty$ one by using the flow of characteristics to transfer pointwise estimates at $x-vt$ into integral in the space variable. Such a method is the core of the $L^\infty$ theory thanks to the following lemma.

\bigskip
\begin{lemma}\label{lem:L2Linfty}
Let $\beta>3/2$ and let $(H1) - (H4)$ hold for the collision kernel. Assume that there exist $T_0 >0$ and $\lambda$, $C_{T_0}>0$ such that for all $\mathbf{f}(t,x,v)$ in $L^\infty_{x,v}\pa{\langle v \rangle^\beta\boldsymbol\mu^{-1/2}}$ solution to
\begin{equation}\label{lineqLinfty}
\partial_t \mathbf{f} + v\cdot\nabla_x\mathbf{f} = \mathbf{L}(\mathbf{f})
\end{equation}
with initial data $\mathbf{f_0}$, the following holds for all $t$ in $[0,T_0]$
$$\norm{\mathbf{f}(t)}_{L^\infty_{x,v}\pa{\langle v \rangle^\beta\boldsymbol\mu^{-1/2}}} \leq e^{\lambda (T_0-2t)}\norm{\mathbf{f_0}}_{L^\infty_{x,v}\pa{\langle v \rangle^\beta\boldsymbol\mu^{-1/2}}} + C_{T_0}\int_0^{t} \norm{\mathbf{f}(s)}_{L^2_{x,v}\pa{\boldsymbol\mu^{-1/2}}}\:ds.$$
Then for all $0 < \tilde{\lambda}< \min\br{\lambda,\lambda_G}$, defined in Theorem \ref{theo:semigroupL2}, there exists $C = C\pa{\beta,\tilde{\lambda}}>0$ such that for all $\mathbf{f}$ solution to $\eqref{lineqLinfty}$ in $L^\infty_{x,v}\pa{\langle v \rangle^\beta\boldsymbol\mu^{-1/2}}$  satisfying $\Pi_{\mathbf{G}}(\mathbf{f})=0$,
$$\forall t \geq 0, \quad \norm{\mathbf{f}(t)}_{L^\infty_{x,v}\pa{\langle v \rangle^\beta\boldsymbol\mu^{-1/2}}} \leq C e^{-\tilde{\lambda} t}\norm{\mathbf{f_0}}_{L^\infty_{x,v}\pa{\langle v \rangle^\beta\boldsymbol\mu^{-1/2}}}.$$
\end{lemma}
\bigskip

\begin{proof}[Proof of Lemma \ref{lem:L2Linfty}]
To shorten the computations we use the following notation $\mathbf{w}_\beta(v) = \langle v \rangle^\beta\boldsymbol\mu^{-1/2}.$
\par Let $\mathbf{f}$ be a solution to $\eqref{lineqLinfty}$ in $L^\infty_{x,v}\pa{\mathbf{w}_\beta}$ associated with the initial data $\mathbf{f_0}$. Taking $n$ in $\N$ we can apply the assumption of the lemma to $\mathbf{\tilde{f}}(t,x,v) = \mathbf{f}(t+nT_0,x,v)$. This yields, with a change of variables $t\mapsto t-nT_0$,
\begin{equation*}
\begin{split}
\norm{\mathbf{f}((n+1)T_0)}_{L^\infty_{x,v}\pa{\mathbf{w}_\beta}} \leq& e^{-\lambda T_0}\norm{\mathbf{f}(nT_0)}_{L^\infty_{x,v}\pa{\mathbf{w}_\beta}} + C_{T_0}\int_{nT_0}^{(n+1)T_0} \norm{\mathbf{f}(s)}_{L^2_{x,v}\pa{\boldsymbol\mu^{-1/2}}}\:ds.
\end{split}
\end{equation*}
We can iterate the process for $\mathbf{f}(nT_0)$ as long as $n \neq 0$. We thus obtain
\begin{equation}\label{controlfn+1T0}
\begin{split}
\norm{\mathbf{f}((n+1)T_0)}_{L^\infty_{x,v}\pa{\mathbf{w}_\beta}} \leq& e^{-(n+1)\lambda T_0}\norm{\mathbf{f_0}}_{L^\infty_{x,v}\pa{\mathbf{w}_\beta}} 
\\&+ C_{T_0}\sum\limits_{k=0}^ne^{-k\lambda T_0}\int_{(n-k)T_0}^{(n+1-k)T_0} \norm{\mathbf{f}(s)}_{L^2_{x,v}\pa{\boldsymbol\mu^{-1/2}}}\:ds.
\end{split}
\end{equation}

\bigskip
We see that multiplying and dividing by $\langle v \rangle^\beta$ gives
$$\norm{\mathbf{f}}^2_{L^2_{x,v}\pa{\boldsymbol\mu^{-1/2}}} = \sum\limits_{i=1}^N\int_{\T^3\times\R^3}f_i^2\mu_i^{-1}\:dxdv \leq \abs{\T^3}\pa{\int_{\R^3}\frac{dv}{\pa{1+ \abs{v}^2}^\beta}}\norm{\mathbf{f}}^2_{L^\infty_{x,v}\pa{\mathbf{w}_\beta}}.$$
Since $\beta >3/2$, the integral is finite and $\mathbf{f}$ also belongs to $L^2_{x,v}\pa{\boldsymbol\mu^{-1/2}}$. By Theorem \ref{theo:semigroupL2} it follows that $\mathbf{f}(t) = S_{\mathbf{G}}(t)\pa{ \mathbf{f_0}}$ and thus if $\Pi_\mathbf{G}(\mathbf{f})=0$ we have the following exponential decay
$$\forall t\geq 0, \quad\norm{\mathbf{f}(t)}_{L^2_{x,v}\pa{\boldsymbol\mu^{-1/2}}} \leq C_G e^{-\lambda_G t}\norm{\mathbf{f_0}}_{L^2_{x,v}\pa{\boldsymbol\mu^{-1/2}}} \leq C_{G,\beta}e^{-\lambda_G t}\norm{\mathbf{f_0}}_{L^\infty_{x,v}\pa{\mathbf{w}_\beta}}.$$
Plugging the latter into $\eqref{controlfn+1T0}$ and taking $0<\tilde{\lambda} <  \lambda_1 \leq  \min\br{\lambda,\lambda_G}$
\begin{equation*}
\begin{split}
&\norm{\mathbf{f}((n+1) T_0)}_{L^\infty_{x,v}\pa{\mathbf{w}_\beta}} 
\\&\quad\quad\quad\leq \cro{e^{-(n+1)\lambda T_0} + C_{\beta,G}\pa{\sum\limits_{k=0}^ne^{-k\lambda_1 T_0}\int_{(n-k)T_0}^{(n+1-k)T_0} e^{-\lambda_1 s}\:ds}}\norm{\mathbf{f_0}}_{L^\infty_{x,v}\pa{\mathbf{w}_\beta}}
\\&\quad\quad\quad\leq \cro{e^{-(n+1)\lambda T_0} + \frac{C_{\beta,G}e^{\lambda_1 T_0}}{\lambda_1}(n+1)e^{-(n+1)\lambda_1T_0}}\norm{\mathbf{f_0}}_{L^\infty_{x,v}\pa{\mathbf{w}_\beta}}
\\&\quad\quad\quad\leq C_{T_0,\tilde{\lambda}}e^{-(n+1)\tilde{\lambda}T_0}\norm{\mathbf{f_0}}_{L^\infty_{x,v}\pa{\mathbf{w}_\beta}},
\end{split}
\end{equation*}
where we used $(n+1)e^{-(n+1)\lambda_1T_0} \leq C e^{-(n+1)\tilde{\lambda}T_0}$.

\bigskip
At last, for $t\geq 0$ there exists $n$ in $\N$ such that $nT_0 \leq t \leq (n+1) T_0$. Using the inequality satisfied by $\sup\limits_{0\leq t \leq T_0}\norm{\mathbf{f}(t-nT_0,x,v)}_{L^\infty_{x,v}\pa{\mathbf{w}_\beta}}$, same computations as above gives
\begin{eqnarray*}
\norm{\mathbf{f}(t)}_{L^\infty_{x,v}\pa{\mathbf{w}_\beta}} &\leq& C \norm{\mathbf{f}((n+1)T_0)}_{L^\infty_{x,v}\pa{\mathbf{w}_\beta}} \leq Ce^{-(n+1)\tilde{\lambda}T_0}\norm{\mathbf{f_0}}_{L^\infty_{x,v}\pa{\mathbf{w}_\beta}}
\\&\leq& Ce^{-\tilde{\lambda} t}\norm{\mathbf{f_0}}_{L^\infty_{x,v}\pa{\mathbf{w}_\beta}},
\end{eqnarray*}
where $C$ is any positive constants depending on $T_0$. This concludes the proof.
\end{proof}
\bigskip

We now state the theorem about the linear perturbed equation.

\bigskip
\begin{theorem}\label{theo:semigroupLinfty}
Let $\beta > 3/2$ and let assumptions $(H1) - (H4)$ hold for the collision kernel. The linear perturbed operator $\mathbf{G}=\mathbf{L}-v\cdot\nabla_x$ generates a semigroup $S_{\mathbf{G}}(t)$ on $L^\infty_{x,v}\pa{\langle v\rangle^\beta \boldsymbol\mu^{-1/2}}$. Moreover, there exists $\lambda_\infty$ and $C_\infty>0$ such that 
$$\forall t \geq 0, \quad \norm{S_{\mathbf{G}}(t)\pa{\mbox{Id}-\Pi_{\mathbf{G}}}}_{L^\infty_{x,v}\pa{\langle v\rangle^\beta \boldsymbol\mu^{-1/2}}} \leq C_\infty e^{-\lambda_\infty t},$$
where $\Pi_{\mathbf{G}}$ is the orthogonal projection onto $\mbox{Ker}(\mathbf{G})$ in $L^2_{x,v}\pa{\boldsymbol\mu^{-1/2}}$.
\\The constants $C_\infty$ and $\lambda_\infty$ are explicit and depend on $\beta$, $N$ and the collision kernel.
\end{theorem}
\bigskip

\begin{proof}[Proof of Theorem \ref{theo:semigroupLinfty}]
As before, we use the shorthand notations $\mathbf{w}_\beta = \langle v \rangle^\beta\boldsymbol\mu^{-1/2}$ and $w_{\beta i} = \langle v \rangle^\beta\mu^{-1/2}_i$.
\par Let $\mathbf{f_0}$ be in $L^\infty_{x,v}\pa{\mathbf{w}_\beta}$ with $\beta >3/2$. 
If $\mathbf{f}$ is solution of $\eqref{lineqLinfty}$:
$$\partial_t \mathbf{f} = \mathbf{G}(\mathbf{f})$$
in $L^\infty_{x,v}\pa{\mathbf{w}_\beta}$ with initial data $\mathbf{f_0}$ then because $\beta >3/2$ we have that $\mathbf{f}$ belongs to $L^2_{x,v}\pa{\boldsymbol\mu^{-1/2}}$ and $\mathbf{f}(t) = S_{\mathbf{G}}(t)\mathbf{f_0}$ in this space. This implies first that $\mathbf{f}$ has to be unique and second that $\mbox{Ker}(\mathbf{G})$ and $\pa{\mbox{Ker}(\mathbf{G})}^\bot$ are stable under the flow of the equation $\eqref{lineqLinfty}$. It suffices to consider $\mathbf{f_0}$ such that $\Pi_\mathbf{G}(\mathbf{f_0})=0$ and to prove existence and exponential decay of solutions to $\eqref{lineqLinfty}$ in $L^\infty_{x,v}\pa{\mathbf{w}_\beta}$ with initial data $\mathbf{f_0}$.

\bigskip
We recall that $\boldsymbol\nu(v) = (\nu_i(v))_{1\leq i\leq N}$ is a multiplicative operator and so the existence of solutions to equation $\eqref{lineqLinfty}$ is equivalent to the existence of a fixed point to its Duhamel's form along the characteristics of the free transport equation. These characteristic trajectories are straight lines of constant speed. We thus need to have existence and exponential decay of a fixed point $\mathbf{f} = \pa{f_i}_{1\leq i \leq N}$ to the following problem for all $i$ in $\br{1,\dots,N}$:
$$f_i(t,x,v) = e^{-\nu_i(v)t}f_{0i}(x-vt,v) + \int_0^t e^{-\nu_i(v)(t-s)}K_i\pa{\mathbf{f}(s,x-(t-s)v,\cdot)}(v)\:ds.$$
Thanks to Lemma \ref{lem:kernelK}, each operator $K_i$ is a kernel operator and we thus have for all $i$ in $\br{1,\dots,N}$,
\begin{equation*}
\begin{split}
f_i(t,x,v) =& e^{-\nu_i(v)t}f_{0i}(x-vt,v)
\\&+\sum\limits_{j=1}^N \int_0^t\int_{\R^3}e^{-\nu_i(v)(t-s)}k^{(i)}_j(v,v_*)f_j(s,x-(t-s)v,v_*)\:dv_*ds.
\end{split}
\end{equation*}

Iterating this Duhamel's form we end up with the following formulation
\begin{equation}\label{iteratedDuhamel}
f_i(t,x,v) = D^{(i)}_1(\mathbf{f_0})(t,x,v) + D^{(i)}_2(\mathbf{f_0})(t,x,v) + D^{(i)}_3(\mathbf{f})(t,x,v) 
\end{equation}
where we define
\begin{eqnarray}
&&D^{(i)}_1(\mathbf{f_0}) = e^{-\nu_i(v)t}f_{0i}(x-vt,v), \label{D1}
\\&&D^{(i)}_2(\mathbf{f_0}) = \sum\limits_{j=1}^N\int_0^t\int_{\R^3}e^{-\nu_i(v)(t-s)}e^{-\nu_j(v_*)s}k^{(i)}_j(v,v_*) \label{D2}
\\&&\quad\quad\quad\quad\quad\times f_{0j}(x-(t-s)v -sv_*,v_*)\:dv_*ds, \nonumber
\\&&D^{(i)}_3(\mathbf{f}) = \sum\limits_{j=1}^N\sum\limits_{l=1}^N \int_0^t\int_0^s\int_{\R^3}\int_{\R^3} e^{-\nu_i(v)(t-s)}e^{-\nu_j(v_*)(s-s_1)}k^{(i)}_j(v,v_*)k^{(j)}_l(v_*,v_{**}) \nonumber
\\&&\quad\quad\quad\quad\times f_l(s_1,x-(t-s)v-(s-s_1)v_*,v_{**})\:dv_{**}dv_*ds_1ds. \label{D3}
\end{eqnarray}

\par Thanks to this Duhamel's formulation, the existence of a fixed point to $\eqref{iteratedDuhamel}$ in $L^\infty_tL^\infty_{x,v}\pa{\mathbf{w}_\beta}$ follows from a contraction argument. The computations required to prove such a contraction property follow exactly the ones leading to the exponential decay of the latter fixed point. We therefore solely prove that if $\mathbf{f}$ satisfies $\eqref{iteratedDuhamel}$ then $\mathbf{f}$ decreases exponentially in $L^\infty_{x,v}\pa{\mathbf{w}_\beta}$.

\bigskip
We shall bound each of the terms $\eqref{D1}$, $\eqref{D2}$ and $\eqref{D3}$ separately. From $\eqref{nu0nu1}$, for all $i$ there exists $\nu_0^{(i)} = \min_{v\in\R^3}\br{\nu_i(v)}> 0$. We define by $\nu_0>0$ the minimum of the $\nu_0^{(i)}$ and every positive constant independent of $i$ and $\mathbf{f}$ will be denoted by $C_k$.
\par The first term $\eqref{D1}$ is straightforwardly bounded.
\begin{equation}\label{D1final}
\norm{D^{(i)}_1(\mathbf{f_0})(t)}_{L^\infty_{x,v}\pa{w_{\beta i}}} \leq e^{-\nu_0 t}\norm{f_{0i}}_{L^\infty_{x,v}\pa{w_{\beta i}}}.
\end{equation}

\par In the second term $\eqref{D2}$ we multiply and divide inside the $v_*$ integral by $w_{\beta j}(v_*)$ and take the supremum $\T^3\times\R^3$.
\begin{equation*}
\begin{split}
\abs{w_{\beta i}(v)D^{(i)}_2(\mathbf{f_0})(t)} \leq & Cte^{-\nu_0 t}
\\&\times\sum\limits_{j=1}^N\pa{\int_{\R^3} \abs{k^{(i)}_j(v,v_*)} \frac{\langle v \rangle^\beta \mu_i^{-1/2}}{\langle v_* \rangle^\beta \mu_{j*}^{-1/2}}\:dv_*}\norm{f_{0j}}_{L^\infty_{x,v}\pa{w_{\beta j}}}.
\end{split}
\end{equation*}
Applying Lemma \ref{lem:L1ki} with $\theta=\eps=0$, the integral term is  bounded uniformly in $i$, $j$ and $v$. Hence
\begin{equation}\label{D2final}
\norm{D^{(i)}_2(\mathbf{f_0})(t)}_{L^\infty_{x,v}\pa{w_{\beta i}}} \leq C_2 t e^{-\nu_0 t} \norm{\mathbf{f_0}}_{L^\infty_{x,v}\pa{\mathbf{w}_\beta}}.
\end{equation}
\par The third and last term $\eqref{D3}$ is more involved analytically and requires to consider the cases $\abs{v} \geq R$ and $\abs{v} \leq R$, for $R$ to be chosen later, separately.

\bigskip
\textbf{Step 1: $\mathbf{\abs{v}\geq R}$.}
We multiply and divide by $w_{\beta l}(v_{**})$ inside the $v_*$ integral of $\eqref{D3}$ and take the supremum in space and velocity for $f_l$. The exponential factor can be bounded by
$$e^{-\nu_i(v)(t-s)-\nu_j(v_*)(s-s_1)} \leq e^{-\frac{\nu_0}{2}t}e^{-\frac{\nu_0}{2}(t-s)}e^{-\frac{\nu_0}{2}(t-s_1)}e^{\frac{\nu_0}{2}s}.$$
Hence, for all $t$, $x$ and $v$,
\begin{equation}\label{genericbound}
\begin{split}
&\abs{w_{\beta i}(v)D^{(i)}_3\pa{\mathbf{f}}(t,x,v)} 
\\&\quad\quad\quad\leq e^{-\frac{\nu_0}{2}t}\sum\limits_{1\leq j,l\leq N}\int_0^t\int_0^se^{-\frac{\nu_0}{2}(t-s_1)}\pa{e^{\frac{\nu_0}{2}s}\norm{f_l}_{L^\infty_{x,v}\pa{w_{\beta l}}}}
\\&\quad\quad\quad\quad\times\cro{\int_{\R^3}\abs{k^{(i)}_j(v,v_*)}\frac{w_{\beta i}(v)}{w_{\beta i}(v_*)}\pa{\int_{\R^3}\abs{k^{(j)}_l(v_*,v_{**})}\frac{w_{\beta i}(v_*)}{w_{\beta l}(v_{**})}\:dv_{**}}dv_*}\:ds_1ds.
\end{split}
\end{equation}
We use Lemma \ref{lem:L1ki} twice to bound the term inside bracket independently of $j$, $l$ and $v$ by
$$\frac{C_\beta^2}{1+\abs{v}}\leq \frac{C_\beta^2}{1+R}.$$
We conclude
\begin{equation}\label{D3geqR}
\sup\limits_{\abs{v}\geq R}\abs{w_{\beta i}(v)D^{(i)}_3\pa{\mathbf{f}}(t,x,v)} \leq \frac{C_3}{1+R}\:e^{-\frac{\nu_0}{2}t}\sup\limits_{0\leq s \leq t}\cro{e^{\frac{\nu_0}{2}}\norm{\mathbf{f}}_{L^\infty_{x,v}\pa{\mathbf{w}_\beta}}}.
\end{equation}

\bigskip
\textbf{Step 2: $\mathbf{\abs{v}\leq R}$.}
In order for the change of variables $y = x-(t-s)v - (s-s_1)v_*$ in the $v_*$ integral to be well-defined we need $s-s_1$ bounded from below. Moreover, in order to make the $L^2$-norm appearing we would need to have $k^{(i)}_j(v,v_*)$ uniformly bounded which is not the case. We therefore need to approximate it uniformly by compactly supported functions, which is possible on compact domains. We take $\eta >0$ and divide $\eqref{D3}$ into four parts
\begin{equation}\label{D3leqNdecomposition}
\begin{split}
D^{(i)}_3\pa{\mathbf{f}} = & \int_0^t\int_{s-\eta}^s\int_{\R^3\times\R^3} d^{(i)}_3 + \int_0^t\int_0^{s-\eta}\int_{\abs{v_*}\geq 2R}\int_{\R^3} d^{(i)}_3 
\\&+ \int_0^t\int_0^{s-\eta}\int_{\abs{v_*}\leq 2R}\int_{\abs{v_{**}}\geq 3R} d^{(i)}_3 + \int_0^t\int_0^{s-\eta}\int_{\abs{v_*}\leq 2R}\int_{\abs{v_{**}}\leq 3R} d^{(i)}_3,
\end{split}
\end{equation} 
where, using
$$e^{-\nu_i(v)(t-s)-\nu_j(v_*)(s-s_1)} \leq e^{-\nu_0(t-s_1)},$$
we have the following bound
\begin{equation*}
\begin{split}
d^{(i)}_3 \leq e^{-\nu_0(t-s_1)}\sum\limits_{1\leq j,l\leq N}  & k^{(i)}_j(v,v_*)k^{(j)}_l(v_*,v_{**})
\\&\times \abs{f_l(s_1,x-(t-s)v-(s-s_1)v_*,v_{**})}\:dv_{**}dv_*ds_1ds.
\end{split}
\end{equation*}

\par The first integral in $\eqref{D3leqNdecomposition}$ is dealt with by using Lemma \ref{lem:L1ki} twice, as for $\eqref{genericbound}$.  We get
\begin{equation}\label{di31}
\begin{split}
\abs{w_{\beta i}\int_0^t\int_{s-\eta}^s\int_{\R^3\times\R^3} d^{(i)}_3} &\leq  C e^{-\frac{\nu_0}{2}t}\pa{\int_0^t\int_{s-\eta}^s e^{-\frac{\nu_0}{2}(t-s_1)\:}ds_1ds}\sup\limits_{0\leq s \leq t}\cro{e^{\frac{\nu_0}{2}}\norm{\mathbf{f}}_{L^\infty_{x,v}\pa{\mathbf{w}_\beta}}}
\\&\leq \eta C e^{-\frac{\nu_0}{2}t}\sup\limits_{0\leq s \leq t}\cro{e^{\frac{\nu_0}{2}}\norm{\mathbf{f}}_{L^\infty_{x,v}\pa{\mathbf{w}_\beta}}}.
\end{split}
\end{equation}
\par For the second and third terms in $\eqref{D3leqNdecomposition}$ we remark that for $\abs{v}\leq R$ we always have either $\abs{v-v_*} \geq R$ or $\abs{v_*-v_{**}}\geq R$  in the domain of integration and therefore we have for any $\eps >0$ either one of the following inequalities 
\begin{eqnarray*}
\abs{k^{(i)}_j(v,v_*)} &\leq& e^{-m\eps R^2}\abs{k^{(i)}_j(v,v_*)e^{m\eps\abs{v-v_*}^2}}
\\\abs{k^{(j)}_l(v_*,v_{**})} &\leq& e^{-m\eps R^2}\abs{k^{(j)}_l(v_*,v_{**})e^{m\eps \abs{v_*-v_{**}}^2}}.
\end{eqnarray*}
Now we take $\eps$ small enough to apply Lemma \ref{lem:L1ki} as before but with the first inequality above for $\abs{v_*}\geq 2R$ or the second inequality above for $\abs{v_*}\leq 2R$ and $\abs{v_{**}}\geq 3R$. Exactly the same computations as $\eqref{genericbound}$ before yields
\begin{eqnarray}
\abs{w_{\beta i}\int_0^t\int_0^{s-\eta}\int_{\abs{v_*}\geq 2R}\int_{\R^3} d^{(i)}_3} &\leq&  Ce^{-m\eps R^2} e^{-\frac{\nu_0}{2}t}\sup\limits_{0\leq s \leq t}\cro{e^{\frac{\nu_0}{2}}\norm{\mathbf{f}}_{L^\infty_{x,v}\pa{\mathbf{w}_\beta}}} \label{di32}
\\\abs{w_{\beta i}\int_0^t\int_0^{s-\eta}\int_{\abs{v_*}\geq 2R}\int_{\abs{v_{**}}\geq 3} d^{(i)}_3} &\leq&  Ce^{-m\eps R^2} e^{-\frac{\nu_0}{2}t}\sup\limits_{0\leq s \leq t}\cro{e^{\frac{\nu_0}{2}}\norm{\mathbf{f}}_{L^\infty_{x,v}\pa{\mathbf{w}_\beta}}}\label{di33}
\end{eqnarray} 

\par At last, the last term in $\eqref{D3leqNdecomposition}$ deals with a set included in the compact support 
$$\Omega = \br{(v,v_*,v_{**})\in\R^3,\quad\abs{v}\leq 3R,\:\abs{v_*}\leq 2R,\;\abs{v_{**}}\leq 3R}.$$
As discussed earlier, Lemma \ref{lem:kernelK} shows that $k^{(i)}_j(v,v_*)$ has a possible blow-up in $\abs{v-v_*}^\gamma$. However, since $\Omega$ is compact we can approximate $k^{(i)}_j(v,v_*)$, for all $i$ and $j$, by a smooth and compactly supported function $k^{(i)}_{R,j}(v,v_*)$ in the following uniform sense
\begin{equation}\label{approxkij}
\sup\limits_{\abs{v}\leq 3R} \int_{\abs{v_*}\leq 3R}\abs{k^{(i)}_{j}(v,v_*) -k^{(i)}_{R,j}(v,v_*)}\frac{w_{\beta i(v)}}{w_{\beta i}(v_*)}\:dv_*\leq \frac{1}{R}.
\end{equation}
Thanks to the following equality
\begin{equation*}
\begin{split}
k^{(i)}_j(v,v_*)k^{(j)}_l(v_*,v_{**}) =& \pa{k^{(i)}_j(v,v_*)-k^{(i)}_{R,j}(v,v_{*})}k^{(j)}_l(v_*,v_{**})
\\&+ \pa{k^{(j)}_l(v_*,v_{**})-k^{(j)}_{R,l}(v_*,v_{**})}k^{(i)}_{R,j}(v,v_{*})
\\&+ k^{(i)}_{R,j}(v,v_*)k^{(j)}_{R,l}(v_{*},v_{**})
\end{split}
\end{equation*}
the last term in $\eqref{D3leqNdecomposition}$ is bounded by
\begin{equation*}
\begin{split}
&\abs{w_{\beta i} \int_0^t\int_0^{s-\eta}\int_{\abs{v_*}\leq 2R}\int_{\abs{v_{**}}\leq 3R} d^{(i)}_3} 
\\&\leq \frac{C}{R} e^{-\frac{\nu_0}{2}t}\sup\limits_{0\leq s \leq t}\cro{e^{\frac{\nu_0}{2}}\norm{\mathbf{f}}_{L^\infty_{x,v}\pa{\mathbf{w}_\beta}}}\sup\limits_{1\leq i,j,l \leq N}\pa{\sup\limits_{\abs{v_*}\leq 2R}\int_{\abs{v_{**}}\leq 3R} \abs{k^{(j)}_l(v_*,v_{**})}\frac{w_{\beta i}(v_*)}{w_{\beta l}(v_{**})}\:dv_{**} }
\\&\quad +\frac{C}{R} e^{-\frac{\nu_0}{2}t}\sup\limits_{0\leq s \leq t}\cro{e^{\frac{\nu_0}{2}}\norm{\mathbf{f}}_{L^\infty_{x,v}\pa{\mathbf{w}_\beta}}}\sup\limits_{1\leq i,j \leq N}\pa{\sup\limits_{\abs{v}\leq R}\int_{\abs{v_{*}}\leq 2R} \abs{k^{(i)}_{R,j}(v,v_{*})}\frac{w_{\beta i(v)}}{w_{\beta i}(v_*)}\:dv_{*} }
\\&\quad + \sum\limits_{1\leq j,l\leq N}\int_0^t\int_0^{s-\eta}e^{-\nu_0(t-s_1)}\int_{\overset{\abs{v_*}\leq 2R}{\abs{v_{**}}\leq 3R}} \abs{k^{(i)}_{R,j}(v,v_*)k^{(j)}_{R,l}(v_{*},v_{**})}\abs{f_l(s_1,y(v_*),v_{**})}
\end{split}
\end{equation*}
where we made the usual controls $\eqref{genericbound}$ and used $\eqref{approxkij}$. We also defined $y(v_*) = x-(t-s)v-(s-s_1)v_*$. The first two terms are dealt with using Lemma \ref{lem:L1ki} while we can bound $k^{(i)}_{R,j}k^{(j)}_{R,l}$ by a constant $C_R$ depending only on $R$ (note that all constants only depending on $R$ will be denoted by $C_R$). This yields
\begin{equation*}
\begin{split}
&\abs{w_{\beta i} \int_0^t\int_0^{s-\eta}\int_{\abs{v_*}
\leq 2R}\int_{\abs{v_{**}}\leq 3R} d^{(i)}_3} 
\\&\leq \frac{C}{R} e^{-\frac{\nu_0}{2}t}\sup\limits_{0\leq s \leq t}\cro{e^{\frac{\nu_0}{2}}\norm{\mathbf{f}}_{L^\infty_{x,v}\pa{\mathbf{w}_\beta}}} + C_R\sum\limits_{l=1}^{N}\int_0^t\int_0^{s-\eta}\int_{\overset{\abs{v_*}\leq 2R}{\abs{v_{**}}\leq 3R}} \abs{f_l(s_1,y(v_*),v_{**})}.
\end{split}
\end{equation*}
We first integrate over $v_*$. We make the change of variables $y=y(v_*)$ which has a jacobian $\abs{s-s_1}^{-3} \leq \eta^{-3}$. Since we are on the periodic box, $y$ has to be understood as the class of equivalence of $y(v_*)$ and is therefore not one-to-one. However, $v_*$ being bounded by $2R$ we cover $\T^3$ only finitely many times (depending on $R$). Hence,
\begin{equation*}
\begin{split}
&\abs{w_{\beta i} \int_0^t\int_0^{s-\eta}\int_{\abs{v_*}
\leq 2R}\int_{\abs{v_{**}}\leq 3R} d^{(i)}_3} 
\\&\leq \frac{C}{R} e^{-\frac{\nu_0}{2}t}\sup\limits_{0\leq s \leq t}\cro{e^{\frac{\nu_0}{2}}\norm{\mathbf{f}}_{L^\infty_{x,v}\pa{\mathbf{w}_\beta}}} + \frac{C_R}{\eta^3}\sum\limits_{l=1}^{N}\int_0^t\int_0^{s-\eta}\int_{\T^3\times\R^3}\abs{f_l(s_1,y,v_{**})}.
\end{split}
\end{equation*}
Finally, a Cauchy-Schwarz inequality against  $\mu^{-1/2}_l(v_{**})/\mu^{-1/2}_l(v_{**})$ yields the following estimate
\begin{equation}\label{di34}
\begin{split}
&\abs{w_{\beta i} \int_0^t\int_0^{s-\eta}\int_{\abs{v_*}
\leq 2R}\int_{\abs{v_{**}}\leq 3R} d^{(i)}_3} 
\\&\leq \frac{C}{R} e^{-\frac{\nu_0}{2}t}\sup\limits_{0\leq s \leq t}\cro{e^{\frac{\nu_0}{2}}\norm{\mathbf{f}}_{L^\infty_{x,v}\pa{\mathbf{w}_\beta}}} + \frac{C_R}{\eta^3}t\int_0^t\norm{\mathbf{f}(s_1)}_{L^2_{x,v}\pa{\boldsymbol\mu^{-1/2}}}\:ds_1.
\end{split}
\end{equation}
\par Plugging $\eqref{di31}$, $\eqref{di32}$, $\eqref{di33}$ and $\eqref{di34}$ into $\eqref{D3leqNdecomposition}$ gives the final estimate
\begin{equation}\label{D3leqR}
\begin{split}
\sup\limits_{\abs{v}\leq R}\abs{w_{\beta i}(v)D^{(i)}_3\pa{\mathbf{f}}} \leq& C_4e^{-\frac{\nu_0}{2}t}\pa{\eta +e^{-m\eps R^2} +\frac{1}{R}}\sup\limits_{0\leq s \leq t}\cro{e^{\frac{\nu_0}{2}}\norm{\mathbf{f}}_{L^\infty_{x,v}\pa{\mathbf{w}_\beta}}} 
\\&+ C_{R,\eta}t\int_0^t\norm{\mathbf{f}(s)}_{L^2_{x,v}\pa{\boldsymbol\mu^{-1/2}}}\:ds.
\end{split}
\end{equation}

\bigskip
We can now conclude the proof by gathering $\eqref{iteratedDuhamel}$, $\eqref{D1final}$, $\eqref{D2final}$, $\eqref{D3geqR}$ and $\eqref{D3leqR}$. We get that for all $i$ in $\br{1,\dots,N}$
\begin{equation}\label{Linftyfinal}
\begin{split}
e^{\frac{\nu_0}{2}t}\norm{f_i(t)}_{L^\infty_{x,v}\pa{w_{\beta i}}} \leq& \pa{1+C_2t}e^{-\frac{\nu_0}{2}t}\norm{\mathbf{f_0}}_{L^\infty_{x,v}\pa{\mathbf{w}_\beta}} + C_{R,\eta}t\int_0^t\norm{\mathbf{f}(s)}_{L^2_{x,v}\pa{\boldsymbol\mu^{-1/2}}}\:ds
\\&+ C_5\pa{\eta +e^{-m\eps R^2} +\frac{1}{R}}\sup\limits_{0\leq s \leq t}\cro{e^{\frac{\nu_0}{2}}\norm{\mathbf{f}}_{L^\infty_{x,v}\pa{\mathbf{w}_\beta}}}.
\end{split}
\end{equation}
We remind the reader that $C_2$ and $C_5$ are independent of $\eta$, $R$ and $t$; moreover $\eps>0$ is fixed. We choose $R$ large enough and $\eta$ small enough such that
$$C_5\pa{\eta +e^{-m\eps R^2} +\frac{1}{R}}\leq \frac{1}{2}$$
and $T_0 > 0$ such that
$$2(1+C_2T_0)e^{-\nu_0 T_0} = e^{-\frac{\nu_0}{2}T_0}.$$
Such choices with $\eqref{Linftyfinal}$ yields that for all $t$ in $[0,T_0]$,
$$\norm{\mathbf{f}(t)}_{L^\infty_{x,v}\pa{\langle v \rangle^\beta\boldsymbol\mu^{-1/2}}} \leq e^{\frac{\nu_0}{2} (T_0-2t)}\norm{\mathbf{f_0}}_{L^\infty_{x,v}\pa{\langle v \rangle^\beta\boldsymbol\mu^{-1/2}}} + CT_0\int_0^{t} \norm{\mathbf{f}(s)}_{L^2_{x,v}\pa{\boldsymbol\mu^{-1/2}}}\:ds.$$
Lemma \ref{lem:L2Linfty} then concludes the proof of Theorem \ref{theo:semigroupLinfty}.
\end{proof}
\bigskip

\section{The full nonlinear equation in a perturbative regime}\label{sec:fullcauchy}

This section is devoted to the proof of Theorem \ref{theo:fullcauchy}. We divide our study in three steps. Subsection \ref{subsec:existenceexpodecay} deals with the existence of a solution with exponential decay to the perturbed multi-species Boltzmann equation that reads
\begin{equation}\label{perturbedmultiBEcauchy}
\partial_t \mathbf{f}+ v\cdot\nabla_x \mathbf{f} = \mathbf{L}\pa{\mathbf{f}} +\mathbf{Q}\pa{\mathbf{f}}.
\end{equation}
Then Subsection \ref{subsec:uniqueness} proves the uniqueness of such solutions and, at last, Subsection \ref{subsec:positivity} shows the positivity of the latter.

\bigskip


\subsection{Existence of a solution that decays exponentially}\label{subsec:existenceexpodecay}

We refer to the definition of $\Pi_\mathbf{G}$ $\eqref{PiG}$ and recall that $\Pi_\mathbf{G}(\mathbf{f})=0$ is a convenient way to say that $\mathbf{f}$ satisfies the conservation laws $\eqref{conservationlaws}$ with $\theta_\infty=1$ and $u_\infty=0$.
\par This subsection is dedicated to the proof of the following proposition.

\bigskip
\begin{prop}\label{prop:existenceexpodecay}
Let assumptions $(H1) - (H4)$ hold for the collision kernel, and let $k>k_0$, where $k_0$ is the smallest integer such that $C_{k_0} <1$ where $C_k$ was given by $\eqref{Ck}$. There exists $\eta_k$, $C_k$ and $\lambda_k >0$ such that for any $\mathbf{f_0}$ in $L^1_vL^\infty_x\pa{\langle v \rangle^k}$ satisfying $\Pi_{\mathbf{G}}(\mathbf{f_0})=0$, if
$$\norm{\mathbf{f_0}}\leq \eta_k$$
then there exists $\mathbf{f}$ in $L^1_vL^\infty_x\pa{\langle v \rangle^k}$ with $\Pi_\mathbf{G}(\mathbf{f})=0$ solution to $\eqref{perturbedmultiBEcauchy}$ with initial data $\mathbf{f_0}$ such that
$$\forall t\geq 0, \quad \norm{\mathbf{f}}_{L^1_vL^\infty_x\pa{\langle v \rangle^k}} \leq C_k e^{-\lambda_k t}\norm{\mathbf{f_0}}_{L^1_vL^\infty_x\pa{\langle v \rangle^k}}.$$
The constants are explicit and only depend on $N$, $k$ and the collision kernels.
\end{prop}
\bigskip


\subsubsection{Decomposition of the perturbed equation and toolbox}

As explained in the introduction, the main strategy is to find a decomposition of the perturbed Boltzmann equation $\eqref{perturbedmultiBEcauchy}$ into a system of differential equations where we could make use of the $L^\infty$ semigroup theory developed in Section \ref{sec:Linftytheory}. More precisely, one would like to solve a somewhat simpler equation in $L^1_vL^\infty_x\pa{\langle v \rangle^k}$ and that the remainder part has regularising properties and could thus be handled in the more regular space $L^\infty_{x,v}\pa{\langle v \rangle^\beta\boldsymbol\mu^{-1/2}}$. Then the exponential decay of $S_{\mathbf{G}}(t)$ in the more regular space could be carried up to the bigger space. 
\par Remark that
$$L^\infty_{x,v}\pa{\langle v \rangle^\beta\boldsymbol\mu^{-1/2}} \subset L^1_vL^\infty_x\pa{\langle v \rangle^k}.$$

\bigskip
We propose here a decomposition of the mutli-species linear operator $\mathbf{G} = \mathbf{L} -v\cdot\nabla_x$ that follows the idea used in \cite{GMM} for the single-species Boltzmann operator.
We define for $\delta \in (0,1)$ to be chosen later the truncation function $\Theta(v,v^*,\sigma) \in C^{\infty}(\mathbb{R}^3 \times \mathbb{R}^3)$ bounded by one on the set
\begin{align*}
 \left\{|v|\leq\delta^{-1} \quad \textnormal{and} \quad 2\delta \leq |v-v^*| \leq \delta^{-1} \quad \textnormal{and} \quad |\cos\theta|\leq 1-2\delta\right\},
\end{align*}
and its support included in the set
\begin{align*}
 \left\{|v|\leq 2\delta^{-1} \quad \textnormal{and} \quad \delta \leq |v-v^*| \leq 2 \delta^{-1} \quad \textnormal{and} \quad |\cos\theta|\leq 1-\delta\right\}.
\end{align*}

Thus we can define the splitting
\begin{align*}
 \mathbf{G}=\mathbf{L}-v\cdot\nabla_x= \mathbf{A^{(\delta)}} + \mathbf{B^{(\delta)}} - \boldsymbol{\nu} - v \cdot \nabla_x,
\end{align*}
with the operators $\mathbf{A^{(\delta)}}=\left(A^{(\delta)}_i\right)_{1\leq i \leq N}$ and $\mathbf{B^{(\delta)}}=\left(B^{(\delta)}_i\right)_{1\leq i \leq N}$ defined as
\begin{align*}
A^{(\delta)}_i(\mathbf{f}(v))&=\sum_{j=1}^N C_{ij}^{\Phi}\int_{\mathbb{R}^3\times\mathbb{S}^2}\Theta_{\delta}\left(\mu_j'^*f_i' + \mu_i'f_j'^* - \mu_if_j^*\right)b_{ij}(\cos \theta)|v-v^*|^{\gamma}d\sigma dv^*,\\
B^{(\delta)}_i(\mathbf{f}(v))&=\sum_{j=1}^N C_{ij}^{\Phi}\int_{\mathbb{R}^3\times\mathbb{S}^2}\left(1-\Theta_{\delta}\right)\left(\mu_j'^*f_i' + \mu_i'f_j'^* - \mu_if_j^*\right)b_{ij}(\cos \theta)|v-v^*|^{\gamma}d\sigma dv^*.
\end{align*}
Our goal is to show that $\mathbf{A^{(\delta)}}$ has some regularizing effects and that $\mathbf{G_1}:=\mathbf{B^{(\delta)}}-\boldsymbol{\nu}-v\cdot\nabla_x$ acts like a small perturbation of $\boldsymbol{G_{\nu}}:=-\boldsymbol{\nu}-v\cdot\nabla_x$ and is thus hypodissipative.


\bigskip
\begin{lemma}\label{lem:controlA}
For any $k$ in $\N$, $\beta>0$ and $\delta$ in $(0,1)$, there exists $C_A>0$ such that for all $\mathbf{f}$ in $L^1_vL^\infty_x\pa{\langle v \rangle^k}$
$$\norm{\mathbf{A^{(\delta)}} \pa{\mathbf{f}}}_{L^\infty_{x,v}\pa{\langle v\rangle^\beta\boldsymbol{\mu^{-1/2}}}} \leq C_A\norm{\mathbf{f}}_{L^1_vL^\infty_x\pa{\langle v \rangle^k}}.$$
The constant $C_A$ is constructive and only depends on $k$, $\beta$, $\delta$, $N$ and the collision kernels.
\end{lemma}
\bigskip

\begin{proof}[Proof of Lemma $\ref{lem:controlA}$]
As we proved it in Lemma \ref{lem:kernelK}, the operator $\mathbf{A^{(\delta)}}$ can be written as a kernel operator thanks to Carleman representation:
$$\forall \: i\in \br{1,\dots,N},\quad A^{(\delta)}_i (\mathbf{f})(x,v) =\int_{\R^3} \langle \mathbf{k^{(i),(\delta)}_A}(v,v_*),\mathbf{f}(x,v_*)\rangle\:dv_*.$$
Moreover, by definition of  $\mathbf{A^{(\delta)}}$, its kernels $\mathbf{k^{(i),(\delta)}_A}$ are of compact support which implies the desired estimate.
\end{proof}
\bigskip

Thanks to the regularizing property above of the operator $\mathbf{A^{(\delta)}}$ we are looking for solutions to the perturbed Boltzmann equation
$$\partial_t \mathbf{f} = \mathbf{G}\pa{\mathbf{f}} + \mathbf{Q}(\mathbf{f})$$
in the form of $\mathbf{f}=\mathbf{f_1}+\mathbf{f_2}$ with $\mathbf{f_1}$ in $L^1_vL^\infty_x\pa{\langle v \rangle^k}$ and $\mathbf{f_2}$ in $L^\infty_{x,v}\pa{\langle v \rangle^\beta\boldsymbol\mu^{-1/2}}$ and $(\mathbf{f_1},\mathbf{f_2})$ satisfying the following system of equation

\begin{eqnarray}
\partial_t \mathbf{f_1} &=& \mathbf{G_1^{(\delta)}} \pa{\mathbf{f_1}} + \mathbf{Q}(\mathbf{f_1}+\mathbf{f_2}) \quad\mbox{and}\quad \mathbf{f_1}(0,x,v)=\mathbf{f_0}(x,v),\label{f1}
\\\partial_t \mathbf{f_2} &=& \mathbf{G}(\mathbf{f_2}) + \mathbf{A^{(\delta)}}(\mathbf{f_1}) \quad\mbox{and}\quad \mathbf{f_2}(0,x,v)=0\label{f2}.
\end{eqnarray}

\bigskip
The equation in the smaller space $\eqref{f2}$ will be treated thanks to the semigroup generated by $\mathbf{G}$ in $L^\infty\pa{\langle v \rangle^\beta\boldsymbol\mu^{-1/2}}$ whilst we expect an exponential decay for solutions in the larger space $\eqref{f1}$. Indeed, $\mathbf{B^{(\delta)}}$ can be controlled by the multiplicative operator $\boldsymbol\nu(v)$ thanks to the following lemma.

\bigskip
\begin{lemma}\label{lem:controlB}
Define
$$\mathbf{\bar{w_k}} = \pa{1+m_i^{k/2}\abs{v}^k}_{1\leq i \leq N} \quad\mbox{and}\quad \mathbf{\bar{w_k}}\boldsymbol\nu = \pa{(1+m_i^{k/2}\abs{v}^k)\nu_i(v)}_{1\leq i \leq N}.$$
There exists $k_0$ in $\N$ such that for any $k\geq k_0$ and $\delta$ in $(0,1)$ there exists $C_B(k,\delta)>0$ such that for all $\mathbf{f}$ in $L^1_vL^\infty_x\pa{\mathbf{\bar{w_k}}\boldsymbol\nu}$,
$$\norm{\mathbf{B^{(\delta)}}(\mathbf{f})}_{L^1_vL^\infty_x\pa{\mathbf{\bar{w_k}}}} \leq C_B(k,\delta)\norm{\mathbf{f}}_{L^1_vL^\infty_x\pa{\mathbf{\bar{w_k}}\boldsymbol\nu}}.$$
Moreover we have the following formula
$$C_B(k,\delta) = C_k + \eps_k(\delta)$$
where $\eps_k(\delta)$ is an explicit function that tends to $0$ as $\delta$ tends to $0$ and $C_k$ is defined by $\eqref{Ck}$ and  $k_0$ is the minimal integer such that $C_{k_0} <1$.
\end{lemma}
\bigskip

We make an important remark.
\begin{remark}\label{rem:kq*}
We emphasize here that for $k>k_0$ we have that $\lim_{\delta\to 0}C_B(k,\delta)=C_k <1$. Until the end of this article we fix $\delta_k >0$ such that
$C_B(k,\delta_k)<1$. For convenience we will drop the exponent and use the following notations: $\mathbf{B}=\mathbf{B^{(\delta_k)}}$, $\mathbf{A}=\mathbf{A^{(\delta_k)}}$, $\mathbf{G_1}=\mathbf{G_1^{(\delta_k)}}$ and finally $C_B = C_B(k,\delta_k)$. The equivalent of this result for the mono-species Boltzmann equation can be found in \cite[Lemma 4.4]{GMM} for $k>2$ which is recovered here when $m_i=m_j$ (note that our Lemma deals with more general collision kernels).
\par We also notice here that the weighted norm $\mathbf{\bar{w_k}}$ required for this sharp lemma is equivalent to $\langle v \rangle^k$.
\end{remark}
\bigskip

The proof of Lemma \ref{lem:controlB} relies on a Povzner-type inequality. Such inequalities are now common in the mono-species Boltzmann literature (for both elastic and inelastic collisions) \cite{Pov}\cite{MisWen}\cite{Bob}\cite{BobGamPan}\cite{GMM} and state that the integral on $\S^2$ of $\cro{\abs{v'}^k+\abs{v'_*}^k}$ can be controlled strictly by the integral on $\S^2$ of $C_k\cro{\abs{v}^k+\abs{v_*}^k}$ with $C_k=4/(k+2)<1$ (for hard sphere collision kernels) and a remainder term of lower order when $k>2$. As we shall see, the asymmetry brought by the difference of masses generates a larger constant $C_k$ that can still be less than $1$ if $k$ is large enough.
\par The method proposed here to prove such a Povzner inequality is inspired by \cite[Lemma 1 and Corollary 3]{BobGamPan}. The main idea is to consider kinetic energies $m_i\abs{v'}^2$ and $m_j\abs{v'_*}^2$ to exhibit the problematic term arising from $m_i-m_j$ which can be non-zero. We state our result, which covers the mono-species case when $m_i=m_j$.

\bigskip
\begin{prop}[Povzner-type inequality]\label{prop:povzner}
Let $i$ and $j$ in $\br{1,\dots,N}$. Then for all $k>2$,
$$\int_{\S^2} \cro{m_i^{k/2}\abs{v'}^k + m_j^{k/2}\abs{v'_*}^k}\:d\sigma \leq \frac{l_{b_{ij}}}{ b_{ij}^\infty} C_k \cro{m_i\abs{v}^2 + m_j\abs{v_*}^2}^{k/2}$$
where $C_k$ was defined by $\eqref{Ck}$ and $l_{b_{ij}}$, $b_{ij}^\infty$ by $\eqref{constantsbij}$.
\end{prop}
\bigskip

\begin{proof}[Proof of Proposition \ref{prop:povzner}]
By definition of $v'$ and $v'_*$ we can expand $\abs{v'}^2$ and $\abs{v'_*}^2$ as follows
\begin{eqnarray*}
m_i\abs{v'}^2 &=& E \frac{1 + a_{ij}+b_{ij}\langle e,\sigma \rangle}{2}
\\m_i\abs{v'_*}^2 &=& E \frac{1 - a_{ij} - b_{ij}\langle e,\sigma \rangle}{2}
\end{eqnarray*}
where we denoted by $e$ the direction of the vector $m_iv+m_jv_*$ and we defined 
\begin{equation}\label{Eab}
\begin{split}
E &= m_i\abs{v}^2 + m_j\abs{v_*}^2,
\\ a_{ij} &= \frac{1}{E}\frac{m_i-m_j}{m_i+m_j}\cro{m_i\frac{m_i-m_j}{m_i+m_j}\abs{v}^2+m_j\frac{m_j-m_i}{m_i+m_j}\abs{v_*}^2+ 4\frac{m_im_j}{m_i+m_j}\langle v,v_*\rangle},
\\ b_{ij} &= \frac{1}{E}\frac{4m_im_j}{(m_i+m_j)^2}\abs{v-v_*}\abs{m_iv+m_jv_*}.
\end{split}
\end{equation}
We will drop the dependencies on $v$ and $v_*$
The first important property to notice is that for all $\sigma$ on $\S^2$, $m_i\abs{v'}^2$ and $m_j\abs{v'_*}^2$ are positive and this implies
\begin{equation}\label{absolute a b}
\abs{a_{ij}}\leq 1,\quad \abs{a_{ij}+b_{ij}} \leq 1 \quad\mbox{and}\quad \abs{a_{ij}-b_{ij}} \leq 1.
\end{equation}

\bigskip
Plugging these equalities inside the integral yields
\begin{eqnarray}
&&\int_{\S^2} \cro{m_i^{k/2}\abs{v'}^k + m_j^{k/2}\abs{v'_*}^k}\:d\sigma \nonumber
\\&&\quad\quad\quad= E^{k/2}\int_{\S^2}\cro{\pa{\frac{1 + a_{ij}+b_{ij}\langle e,\sigma \rangle}{2}}^{k/2} + \pa{\frac{1 - a_{ij} - b_{ij}\langle e,\sigma \rangle}{2}}^{k/2}} \:d\sigma \nonumber
\\&& \quad\quad\quad= 2\pi E^{k/2} \int_{-1}^1 \cro{\pa{\frac{1 + a_{ij}+b_{ij}z}{2}}^{k/2} + \pa{\frac{1 - a_{ij} - b_{ij}z}{2}}^{k/2}}\:dz \nonumber
\\&& \quad\quad\quad= \frac{8\pi}{k+2}E^{k/2}\cro{F_{k/2}(\abs{a_{ij}},b_{ij})+ F_{k/2}(-\abs{a_{ij}},b_{ij})} \label{Povnzer 1st}
\end{eqnarray}
where
$$F_p(a,b) = \frac{\pa{\frac{1+a+b}{2}}^{p+1}-\pa{\frac{1+a-b}{2}}^{p+1}}{b}.$$

When $\abs{a}\leq 1$, a mere study of the function $F_p(a,\cdot)$ shows that the latter function is increasing on $[0,1+a]$ if $p \geq 0$. Therefore, since $\abs{a_{ij}}\leq 1$, for $k\geq 2$ we can bound $F_{k/2}(\abs{a_{ij}},b_{ij})$ and $F_{k/2}(-\abs{a_{ij}},b_{ij})$ with their value at an upper bound on $b_{ij}$. Using $\eqref{absolute a b}$ we see that $0\leq b_{ij} \leq 1-\abs{a_{ij}}$. Bounding into $\eqref{Povnzer 1st}$, this gives
\begin{equation}\label{Povzner 2nd}
\int_{\S^2} \cro{m_i^{k/2}\abs{v'}^k + m_j^{k/2}\abs{v'_*}^k}\:d\sigma \leq \frac{8\pi}{k+2}E^{k/2}\frac{1-\abs{a_{ij}}^{k/2+1}+(1-\abs{a_{ij}})^{k/2+1}}{1-\abs{a_{ij}}}.
\end{equation}

\bigskip
To conclude the proof it suffices to see that the function
$$a \mapsto \frac{1-\abs{a}^{k/2+1}+(1-\abs{a})^{k/2+1}}{1-\abs{a}}$$
is increasing on $[0,1]$. Proposition \ref{prop:povzner} will follow if $\abs{a_{ij}}\leq \abs{m_i-m_j}/(m_i+m_j)$.
\par Going back to the definition of $a_{ij}$ and decomposing $v_*$ as $v_* = \lambda v + v^\bot$ with $v^\bot$ orthogonal to $v$ we see that
\begin{equation*}
\abs{a_{ij}} = \frac{1}{E}\frac{\abs{m_i-m_j}}{m_i+m_j}\abs{\pa{\frac{m_i^2+m_im_j(4\lambda - \lambda^2-1)+\lambda^2m_j^2}{m_i+m_j}}\abs{v}^2 + m_j\frac{m_i-m_j}{m_i+m_j}\abs{v^\bot}^2}.
\end{equation*}
But then, direct computations show first
$$\abs{m_j\frac{m_i-m_j}{m_i+m_j}} \leq m_j$$
and second
\begin{equation*}
\begin{split}
&\abs{m_i^2+m_im_j(4\lambda - \lambda^2-1)+\lambda^2m_j^2}^2-(m_i+m_j)^2(m_i+\lambda^2 m_j)^2 
\\&\quad\quad\quad= -4m_im_j(1-\lambda)^2(m_i +\lambda m_j)^2 \leq 0.
\end{split}
\end{equation*}
Hence
$$\abs{a_{ij}} \leq \frac{\abs{m_i-m_j}}{m_i+m_j} \frac{(m_i+\lambda^2m_j)\abs{v}^2+m_j\abs{v^\bot}^2}{E}$$
which terminates the proof of the proposition.
\end{proof}
\bigskip
Now we can prove the estimate on $\mathbf{B^{(\delta)}}$.
\begin{proof}[Proof of Lemma \ref{lem:controlB}]
We use the definition $\mathbf{\bar{w_k}} = (1+m_i^{k/2}\abs{v}^k)_{1\leq i \leq N}$. Moreover, as we will drastically bound $\mathbf{B^{(\delta)}}(\mathbf{f})$ by the absolute value inside the integral in $v$, it is enough to show Lemma \ref{lem:controlB} only for $f=f(v)$.
\par With the multi-species Povzner inequality (Proposition \ref{prop:povzner}, the proof follows closely the proof of \cite[Lemma 4.4]{GMM} with appropriate characteristic functions that fits the invariance of the elastic collisions $\eqref{elasticcollision}$.

\bigskip
First we bound the truncation function from above by cutting the integral in the following way
\begin{equation*}
\begin{split}
&\norm{\mathbf{B^{(\delta)}}(\mathbf{f})}_{L^1_v(\mathbf{\bar{w_k}})}
\\&\:\leq \sum_{i,j=1}^N C_{ij}^{\Phi} \int_{\R^6 \times\S^2}\left(1-\Theta_{\delta}\right)\left[\mu_j^{'^*}|f_i'| + \mu_i'|f_j{'^*}| + \mu_i|f_j^*|\right]b_{ij}(\cos \theta)|v -v_*|^\gamma\bar{w_k}_idv_*d\sigma
 \\&\:\leq \sum_{i,j=1}^N C_{ij}^{\Phi}\int_{\{|\cos\theta| \in [1-\delta, 1]\}} b_{ij}(\cos \theta)|v -v_*|^\gamma\mu_j^*|f_i|(\bar{w_k}_i' + \bar{w_k}_j^{'*} + \bar{w_k}_j^*) dvdv_*d\sigma 
\\&\:\quad+\sum_{i,j=1}^N C_{ij}^{\Phi}\int_{|v-v_*|\leq \delta} b_{ij}(\cos \theta)|v -v_*|^\gamma\mu_j^*|f_i|(\bar{w_k}_i' + \bar{w_k}_j^{'*} + \bar{w_k}_j^*) dvdv_*d\sigma
\\&\:\quad+\sum_{i,j=1}^N C_{ij}^{\Phi} \int_{\left\{|v|\geq \delta^{-1}~\textrm{or} ~|v-v_*|\geq \delta^{-1}\right\}} \left[\mu_j^{'^*}|f_i'| + \mu_i'|f_j^{'^*}| + \mu_i |f_j^*|\right]b_{ij}(\cos \theta)|v -v_*|^\gamma\bar{w_k}_i.
\end{split}
\end{equation*}
Note that we used the change of variables $(v,v_*, \sigma) \to (v',v'^*, v-v_*/\abs{v-v_*})$ for $\mu_j^{'*}f_i'$. Then for $\mu_i'f_j^{'*}$ we used first  $(v,v_*,\sigma) \to (v_*,v,-\sigma)$ which sends $(v'_{ij},v_{ij}^{'*})$ to $(v^{'*}_{ji},v_{ji}^{'})$ and then relabelling $i$ and $j$ we come back to the first term $\mu_j^{'*}f_i'$.
\par Defining the characteristic function $\chi_A$ on the set
$$A = \br{\sqrt{m_i|v|^2 + m_j|v_*|^2}\geq \min\br{\sqrt{m_i},\sqrt{m_j}}\delta^{-1} ~~\textrm{or}~~|v-v_*|\geq \delta^{-1}}$$
we can bound $b(\cos \theta)$ by its supremum $b_\infty$ and use the equivalence between $\nu_i$ and $1+\abs{v}^\gamma$ to get
\begin{equation}\label{control B start}
\begin{split}
&\norm{\mathbf{B^{(\delta)}}(\mathbf{f})}_{L^1_v(\mathbf{\bar{w_k}})} 
\\&\quad\leq \delta C(k)\norm{\mathbf{f}}_{L^1_v(\mathbf{\bar{w_k}\boldsymbol{\nu}})}
\\&\quad\quad+ \sum_{i,j=1}^N C_{ij}^{\Phi}\int_{\R^3 \times \R^3 \times \S^2}\chi_A\left[\mu_j'^*|f_i'| + \mu_i'|f_j^{'*}| + \mu_i|f_j^*|\right]b_{ij}(\cos \theta)|v-v_*|^\gamma\bar{w_k}_i\:dvdv_*d\sigma
\end{split}
\end{equation}
where $C(k)$ will denote any positive constant independent on $\delta$ and $\mathbf{f}$.
\bigskip

We shall deal with the second term on the right-hand side of $\eqref{control B start}$ thanks to the Povzner inequality. Indeed, the set $A$ is invariant by the changes of variables already mentioned (remember that when changing $v$ to $v_*$ we also change $i$ and $j$) and therefore

\begin{equation}\label{control chiA}
\begin{split}
&\sum_{i,j=1}^N C_{ij}^{\Phi} \int_{\R^3 \times \R^3 \times \S^2}\chi_A\left[\mu_j'^*|f_i'| + \mu_i'|f_j^{'*}| + \mu_i|f_j^*|\right]b_{ij}(\cos \theta)|v-v_*|^\gamma\bar{w_k}_i\:dvdv_*d\sigma 
\\&\quad\quad= \sum_{i,j=1}^N C_{ij}^{\Phi}\int_{\R^3 \times \R^3 \times \S^2}\chi_A b_{ij}(\cos \theta)|v -v_*|^\gamma\mu_j^*|f_i|(\bar{w_k}_i^{*'} + \bar{w_k}_i' + \bar{w_k}_i^*) \:dvdv_*d\sigma
\\&\quad\quad\leq \sum_{i,j=1}^N C_{ij}^{\Phi}b_{ij}^\infty \int_{\R^3 \times \R^3}\chi_A |v -v_*|^\gamma \mu_j^*|f_i|\pa{\int_{\S^2}\cro{\bar{w_k}_i' + \bar{w_k}_j^{'*} - \bar{w_k}_j^*-\bar{w_k}_i}d\sigma} dvdv_*
\\&\quad\quad\quad +8\pi  \sum_{i,j=1}^N C_{ij}^{\Phi}b_{ij}^\infty\int_{\R^3 \times \R^3}\chi_A |v -v_*|^\gamma\mu_j^*|f_i|\bar{w_k}_i^* \:dvdv_* 
\\&\quad\quad\quad +4\pi \sum_{i,j=1}^N C_{ij}^{\Phi}b_{ij}^\infty \int_{\R^3 \times \R^3}\chi_A |v -v_*|^\gamma\mu_j^*|f_i|\bar{w_k}_i \:dvdv_*
\end{split}
\end{equation}

We can use Proposition \ref{prop:povzner} for the first term on the right-hand side of the inequality. Indeed,
\begin{equation*}
\begin{split}
&\int_{\S^2}\cro{\bar{w_k}_j' + \bar{w_k}_j^{'*} - \bar{w_k}_j^*-\bar{w_k}_i}\:d\sigma
\\ &\quad\quad\quad\leq \frac{l_{b_{ij}}}{b_{ij}^\infty}C_k \pa{m_i\abs{v}^2+m_j\abs{v_*}^2}^{k/2} -4\pi m_i^{k/2}\abs{v}^k -4\pi m_j^{k/2}\abs{v_*}^k
\\&\quad\quad\quad\leq 2^{k/2}\frac{l_{b_{ij}}}{b_{ij}^\infty}C_k\cro{(m_i\abs{v}^2)^{k/2-1/2}(m_j\abs{v_*}^2)^{1/2}+(m_i\abs{v}^2)^{1/2}(m_j\abs{v_*}^2)^{k/2-1/2}} 
\\&\quad\quad\quad\quad - 4\pi\pa{1-\frac{l_{b_{ij}}}{4\pi b_{ij}^\infty}C_k}\cro{m_i^{k/2}\abs{v}^k + m_j^{k/2}\abs{v_*}^k}
\end{split}
\end{equation*}
For $k\geq k_0$ we have that $C_k<1$, hence $\frac{l_{b_{ij}}}{4\pi b_{ij}^\infty}C_k <1$.  We can thus plug this back into $\eqref{control chiA}$ we find, recalling that $\bar{w_k}_i = 1 +m_i^{k/2}\abs{v}^k$
\begin{equation*}
\begin{split}
&\sum_{i,j=1}^N C_{ij}^{\Phi} \int_{\R^3 \times \R^3 \times \S^2}\chi_A\left[\mu_j'^*|f_i'| + \mu_i'|f_j^{'*}| + \mu_i|f_j^*|\right]b_{ij}(\cos \theta)|v-v_*|^\gamma\bar{w_k}_i\:dvdv_*d\sigma
\\&\quad\quad\leq C(k)\sum_{i,j=1}^N \int_{\R^3 \times \R^3}\chi_A |v -v_*|^\gamma \mu_j^*|f_i|\cro{\abs{v}^{k-1}\abs{v_*}+\abs{v}\abs{v_*}^{k-1}} \:dvdv_*
\\&\quad\quad\quad + 12\pi \sum_{i,j=1}^N C_{ij}^{\Phi}b_{ij}^\infty\int_{\R^3 \times \R^3}\chi_A |v -v_*|^\gamma\mu_j^*|f_i|\:dvdv_* +  
\\&\quad\quad\quad +8\pi \sum_{i,j=1}^N C_{ij}^{\Phi}b_{ij}^\infty \int_{\R^3 \times \R^3}\chi_A |v -v_*|^\gamma\mu_j^*|f_i|m_j^{k/2}\abs{v_*}^k \:dvdv_* 
\\&\quad\quad\quad + C_k \sum_{i,j=1}^N C_{ij}^{\Phi} l_{b_{ij}}\int_{\R^3 \times \R^3}\chi_A |v -v_*|^\gamma\mu_j^*|f_i|m_i^{k/2}\abs{v}^k \:dvdv_*
\end{split}
\end{equation*}
From here we can use that
$$\chi_A(v,v_*) \leq 2\max_{i,j}\br{m_i,m_j}\delta(m_i\abs{v}^2+m_j\abs{v^*}^2)$$
and the fact that $\gamma +1<k_0\leq k$ to bound the first, second and third term on the right-hand side by $\delta C(k)\norm{\mathbf{f}}_{L^1_v(\mathbf{\bar{w_k}})}$. And finally, we exactly have the definition of $\nu_i(v)$ in the last term on the right-hand side. This gives
\begin{equation}\label{control B last}
\begin{split}
&\sum_{i,j=1}^N C_{ij}^{\Phi}\int_{\R^3 \times \R^3 \times \S^2}\chi_A\left[\mu_j'^*|f_i'| + \mu_i'|f_j^{'*}| + \mu_i|f_j^*|\right]b_{ij}(\cos \theta)|v-v_*|^\gamma\bar{w_k}_i\:dvdv_*d\sigma 
\\&\quad\quad\quad\leq C_k\sum\limits_{i=1}^N\norm{f_i}_{L^1_v(\bar{w_k}_i\nu_i)}+\delta C(k)\norm{\mathbf{f}}_{L^1_v(\mathbf{\bar{w_k}})}.
\end{split}
\end{equation}
Combining $\eqref{control B start}$ and $\eqref{control B last}$ yields the desired estimate.
\end{proof}
\bigskip

We conclude this subsection with a control on the nonlinear term. 
\bigskip
\begin{lemma}\label{lem:controlQ}
Define $\mathbf{\tilde{Q}}(\mathbf{f},\mathbf{g})$ by 
\begin{equation}\label{tildeQ}
\forall\: 1\leq i \leq N, \quad \tilde{Q}_i(\mathbf{f},\mathbf{g}) = \frac{1}{2}\sum\limits_{j=1}^N \left(Q_{ij}(f_i,g_j) + Q_{ij}(g_i,f_j)\right).
\end{equation}
Then for all $\mathbf{f}$, $\mathbf{g}$  such that $\mathbf{\tilde{Q}}(\mathbf{f},\mathbf{g})$ is well-defined, the latter belongs to $\cro{\mbox{Ker}(L)}^\bot$:
\begin{equation*}
\pi_\mathbf{L}\pa{\mathbf{\tilde{Q}}(\mathbf{f},\mathbf{g})}=0.
\end{equation*}
Moreover, there exists $C_Q >0$ such that for all $i$ in $\br{1,\dots,N}$ and every $\mathbf{f}$ and $\mathbf{g}$,
\begin{equation*}
\begin{split}
\norm{\tilde{Q}_i(\mathbf{f},\mathbf{g})}_{L^1_vL^\infty_x\pa{\langle v \rangle^k}} \leq C_Q \Big[&\norm{f_i}_{L^1_vL^\infty_x\pa{\langle v \rangle^k}}\norm{\mathbf{g}}_{L^1_vL^\infty_x\pa{\langle v \rangle^k\boldsymbol\nu}} 
\\&+ \norm{f_i}_{L^1_vL^\infty_x\pa{\nu_i\langle v \rangle^k}}\norm{\mathbf{g}}_{L^1_vL^\infty_x\pa{\langle v \rangle^k}}\Big],
\end{split}
\end{equation*}
The constant $C_Q$ is explicit and depends only on $k$, $N$ and the kernel of the collision operator.
\end{lemma}

\bigskip
\begin{proof}[Proof of Lemma $\ref{lem:controlQ}$]
The orthogonality property is well-known for the single-species Boltzmann operator \cite[Appendix A.2]{Bri1} and \cite{BGS} and follows from the same methods as to prove $\eqref{symmetry property Qij}$ and $\eqref{invariantsQij}$.
\par The estimate also follows standard computations from the mono-species case, we adapt them to the case of multi-species for the sake of completeness. Since we are dealing with hard potential kernels, we can decompose the bilinear operator $Q_{ij}(f_i,g_j)$, for any $i$, $j$ in $\br{1,\dots,N}$, as
\begin{equation*}
\begin{split}
Q_{ij}(f_i,g_j) =&\int_{\R^3\times \mathbb{S}^{2}}B_{ij}\left(|v - v_*|,\mbox{cos}\:\theta\right)f_i'g_j^{'*}\:dv_*d\sigma 
\\&- \int_{\R^3\times \mathbb{S}^{2}}B_{ij}\left(|v - v_*|,\mbox{cos}\:\theta\right)f_ig_j^*\:dv_*d\sigma.
\end{split}
\end{equation*}
By Minkowski integral inequality we have for all $q$ in $[1,\infty)$,
\begin{equation*}
\begin{split}
\int_{\R^3}\langle v \rangle^k \cro{\int_{\T^3}\left|Q_{ij}(f_i,g_j)\right|^qdx}^{1/q}dv \leq& \int_{\S^2\times\R^3\times\R^3}\langle v \rangle^k\cro{\int_{\T^3}\abs{B_{ij}f_i'g_j^{'*}}^qdx}^{1/q}d\sigma dv_* dv
\\&+ \int_{\S^2\times\R^3\times\R^3}\langle v \rangle^k\cro{\int_{\T^3}\abs{B_{ij}f_ig_j^*}^qdx}^{1/q}d\sigma dv_* dv.
\end{split}
\end{equation*}

Since the function $(v,v_*) \mapsto (v',v'_*)$ is its own inverse and does not change the value of $B_{ij}\pa{\abs{v-v_*},\cos\theta}$, we make the latter change of variables in the first integral and we obtain
\begin{equation*}
\begin{split}
&\int_{\R^3}\langle v \rangle^k \cro{\int_{\T^3}\left|Q_{ij}(f_i,g_j)\right|^q dx}^{1/q}dv 
\\&\quad\quad\quad\leq \int_{\S^2\times\R^3\times\R^3}\pa{\langle v \rangle^k+\langle v' \rangle^k}\cro{\int_{\T^3}\abs{B_{ij}f_ig_j^*}^qdx}^{1/q}d\sigma dv_* dv
\\&\quad\quad\quad\leq C_{ij}\int_{\S^2\times\R^3\times\R^3}\langle v \rangle^k\langle v_* \rangle^k\abs{v-v_*}^\gamma\cro{\int_{\T^3}\abs{f_ig_j^*}^qdx}^{1/q}d\sigma dv_* dv.
\end{split}
\end{equation*}
The constant $C_{ij}>0$ will stand for any constant depending only on $m_i$, $m_j$, the integral over the sphere of $b_{ij}$ and $C^\Phi_{ij}$ (see assumptions on the kernel $B_{ij}$). Finally we use the fact that $\abs{v-v_*}^\gamma \leq \langle v \rangle^\gamma + \langle v^* \rangle^\gamma$.
\begin{equation*}
\begin{split}
&\int_{\R^3}\langle v \rangle^k \cro{\int_{\T^3}\left|Q_{ij}(f_i,g_j)\right|^qdx}^{1/q}dv 
\\&\quad\quad\quad\leq C_{ij}\int_{\S^2\times\R^3\times\R^3}\pa{\langle v \rangle^{k+\gamma}\langle v_* \rangle^{k}+\langle v \rangle^k\langle v_* \rangle^{k+\gamma}}\cro{\int_{\T^3}\abs{f_ig_j^*}^qdx}^{1/q}d\sigma dv_* dv.
\end{split}
\end{equation*}
We take the limit as $q$ tends to infinity and conclude
\begin{equation*}
\begin{split}
\norm{Q_{ij}(f_i,g_j)}_{L^1_vL^\infty_x\pa{\langle v \rangle^k}}\leq C_{ij}\Big[&\norm{f_i}_{L^1_vL^\infty_x\pa{\langle v \rangle^{k}}}\norm{g_j}_{L^1_vL^\infty_x\pa{\langle v \rangle^{k+\gamma}}}
\\&+ \norm{f_i}_{L^1_vL^\infty_x\pa{\langle v \rangle^{k+\gamma}}}\norm{g_j}_{L^1_vL^\infty_x\pa{\langle v \rangle^{k}}}\Big].
\end{split}
\end{equation*}
We remind $\eqref{nu0nu1}$ which states that $\nu_i(v)\sim \langle v \rangle^\gamma$ and the lemma follows after summing over $j$, $C_Q$ being the maximum of all the $C_{ij}$.
\end{proof}
\bigskip


\subsubsection{Study of the equations in $L^1_vL^\infty_x\pa{\langle v \rangle^k}$}\label{subsubsec:cauchyE}

We start with the well-posedness of the system $\eqref{f1}$ in $L^1_vL^\infty_x\pa{\langle v \rangle^k}$. 
\bigskip
\begin{prop}\label{prop:cauchyE}
Let $k>k_0$. Let $\mathbf{f_0}$ be in $L^1_vL^\infty_x\pa{\langle v \rangle^k}$ and $\mathbf{g}=\mathbf{g}(t,x,v)$ be in $L^\infty_tL^1_vL^\infty_x\pa{\boldsymbol\nu\langle v \rangle^k}$. There exist $\eta_1$, $\lambda_1 >0$ such that if
$$\norm{\mathbf{f_0}}_{L^1_vL^\infty_x\pa{\langle v \rangle^k}}\leq \eta_1 \quad\mbox{and}\quad \exists C,\:\lambda>0\quad \norm{\mathbf{g}(t)}_{L^1_vL^\infty_x\pa{\boldsymbol\nu\langle v \rangle^k}}\leq C\norm{\mathbf{f_0}}_{L^1_vL^\infty_x\pa{\langle v \rangle^k}}e^{-\lambda t}$$
then there exists a function $\mathbf{f_1}$ in $L^\infty_tL^1_vL^\infty_x\pa{\langle v \rangle^k}$ such that
$$\partial_t \mathbf{f_1} = \mathbf{G_1}\pa{\mathbf{f_1}} + \mathbf{Q}\pa{\mathbf{f_1}+\mathbf{g}} \quad\mbox{and}\quad \mathbf{f_1}(0,x,v)= \mathbf{f_0}(x,v).$$
Moreover, any solution $\mathbf{f_1}$ satisfies
$$\forall t\geq 0, \quad \norm{\mathbf{f_1}(t)}_{L^1_vL^\infty_x\pa{\langle v \rangle^k}} \leq C_1 e^{-\lambda_1 t}\norm{\mathbf{f_0}}_{L^1_vL^\infty_x\pa{\langle v \rangle^k}}.$$
The constants $C_1$, $\delta_1$, $\eta_1$ and $\lambda_1$ are independent of $\mathbf{f_0}$ and $\mathbf{g}$ and depends on $N$, $k$ and the collision kernel.
\end{prop}
\bigskip

\begin{proof}[Proof of Proposition \ref{prop:cauchyE}]
We start by showing the exponential decay and then prove existence. As a matter of fact, we saw in Lemma \ref{lem:controlB} that the natural weight to estimate $\mathbf{B}$ is $\mathbf{\bar{w_k}} = 1+\mathbf{m}^{k/2}\abs{v}^k$ which is equivalent to $\langle v \rangle^k$. We will therefore rather work in $L^1_vL^\infty_x\pa{\mathbf{\bar{w_k}}}$ which just modifies the definition for $C_1$, $\delta_1$ and $\eta_1$.

\bigskip
\textbf{Step 1: \textit{a priori} exponential decay.}
Suppose that $\mathbf{f_1}$ is a solution to the differential equation in $L^1_vL^\infty_x\pa{\mathbf{\bar{w_k}}}$ with initial data $\mathbf{f_0}$.
\par We recall that for $q$ in $[1,\infty)$,
$$\norm{\mathbf{f_1}}_{L^1_vL^q_x\left(\mathbf{\bar{w_k}}\right)} = \sum_{i=1}^N \int_{\R^3}\pa{1+m_i^{k/2}\abs{v}^k} \left(\int_{\T^3}\abs{f_{1i}}^q \:dx\right)^{1/q}\:dv.$$
Therefore we can compute for all $i$ in $\br{1,\dots,N}$

\begin{equation*}
\begin{split}
&\frac{d}{dt}\norm{f_{1i}}_{L^1_vL^q_x\pa{1+m_i^{k/2}\abs{v}^k}}
\\&\quad\quad\quad=\int_{\R^3}\pa{1+m_i^{k/2}\abs{v}^k} \norm{f_{1i}}^{1-q}_{L^q_x}\left(\int_{\T^3}\mbox{sgn}(f_{i1})\abs{f_{1i}}^{q-1}\partial_t f_{1i} \:dx\right)\:dv.
\end{split}
\end{equation*}

Observing that
$$\partial_t f_{1i} = -v\cdot\nabla_xf_{1i} -\nu_i(v)f_{1i}+ B_{i}\pa{\mathbf{f_1}} + Q_i\pa{\mathbf{f_1}+\mathbf{g}},$$
that the transport gives null contribution
$$\int_{\T^3}\mbox{sgn}(f_{1i})\abs{f_{1i}}^{q-1} v\cdot\nabla_x f_{1i} \:dx = \frac{1}{q}v\cdot\int_{\T^3} \nabla_x \left(\abs{f_{1i}}^q\right) \:dx =0,$$
that the multiplicative part gives a negative contribution,
$$-\int_{\T^3}\nu_i(v)f_{1i}\:\mbox{sgn}(f_{1i})\abs{f_{1i}}^{q-1} \:dx \leq -\nu_i(v)\norm{f_{1i}}^{q}_{L^q_x}$$
and that by H\"older inequality with $q$ and $q/(q-1)$,
\begin{equation}\label{holder}
\abs{\int_{\T^3}\mbox{sgn}(f_{1i})\abs{f_{1i}}^{q-1}g_i\:dx} \leq \norm{f_{1i}}_{L^q_x}^{q-1}\norm{g_i}_{L^q_x},
\end{equation}
we deduce 
\begin{equation*}
\begin{split}
\frac{d}{dt}\norm{f_{i1}}_{L^1_vL^q_x\left(1+m_i^{k/2}\abs{v}^k\right)} \leq &-\norm{f_{1i}}_{L^1_vL^q_x\pa{\nu_i(1+m_i^{k/2}\abs{v}^k)}}+\norm{B_i\pa{\mathbf{f_1}}}_{L^1_vL^q_x\pa{1+m_i^{k/2}\abs{v}^k}} 
\\&+ \norm{Q_i\pa{\mathbf{f_1}+\mathbf{g}}}_{L^1_vL^q_x\pa{1+m_i^{k/2}\abs{v}^k)}}.
\end{split}
\end{equation*}

\par First sum over $i$ in $\br{1,\dots,N}$ and then let $q$ tend to infinity (on the torus the limit is thus the $L^\infty$-norm). This yields for all $t\geq 0$.
\begin{equation}\label{importanteq}
\begin{split}
\frac{d}{dt}\norm{\mathbf{f_1}}_{L^1_vL^\infty_x\left(\mathbf{\bar{w_k}}\right)} \leq &-\norm{\mathbf{f_1}}_{L^1_vL^\infty_x(\boldsymbol\nu\mathbf{\bar{w_k}})}+\norm{\mathbf{B}\pa{\mathbf{f_1}}}_{L^1_vL^\infty_x\pa{\mathbf{\bar{w_k}}}} 
\\&+ \norm{\mathbf{Q}\pa{\mathbf{f_1}+\mathbf{g}}}_{L^1_vL^\infty_x\pa{\mathbf{\bar{w_k}}}}.
\end{split}
\end{equation}
\par We use Lemma \ref{lem:controlB} to control $\mathbf{B}$, recalling that $0<C_B<1$, and the control of $\mathbf{Q}$ given in Lemma \ref{lem:controlQ} for $\mathbf{Q}$ (of course, since $\mathbf{\bar{w_k}} \sim \langle v \rangle^k$ the Lemma still holds with a different $C_Q$). We get that for all $t\geq 0$,
\begin{equation*}
\begin{split}
\frac{d}{dt}\norm{\mathbf{f_1}}_{L^1_vL^\infty_x\left(\mathbf{\bar{w_k}}\right)} \leq& -\cro{1-C_B-2C_Q\pa{\norm{\mathbf{f_1}}_{L^1_vL^\infty_x(\mathbf{\bar{w_k}})}+2\|\mathbf{g}\|_{L^1_vL^\infty_x(\mathbf{\bar{w_k}})}}}\norm{\mathbf{f_1}}_{L^1_vL^\infty_x(\boldsymbol\nu\mathbf{\bar{w_k}})} 
\\&+ C_Q\norm{\mathbf{g}(t)}_{L^1_vL^\infty_x(\boldsymbol\nu\mathbf{\bar{w_k}})}^2.
\end{split}
\end{equation*}
Since $C_B<1$, if $\norm{\mathbf{f_1}(0)}_{L^1_vL^\infty_x(\mathbf{\bar{w_k}})}$ is sufficiently small and thanks to the exponential decay of $\norm{\mathbf{g}(t)}_{L^\infty_tL^1_vL^\infty_x\pa{\boldsymbol\nu\mathbf{\bar{w_k}}}}$, a direct application of Gr\"onwall lemma yields the desired exponential decay.

\bigskip
\textbf{Step 2: existence.}
Let $\mathbf{f^{(0)}}=0$ and consider the following iterative scheme
$$\partial_t \mathbf{f^{(n+1)}} + v\cdot\nabla_x \mathbf{f^{(n+1)}} = -\boldsymbol\nu(v)\pa{\mathbf{f^{(n+1)}}} + \mathbf{B}\pa{\mathbf{f^{(n)}}} + \mathbf{\tilde{Q}}\pa{\mathbf{f^{(n)}}+\mathbf{g}}$$
with the initial data $\mathbf{f^{(n+1)}}(0,x,v)=\mathbf{f_0}$.
\par For each $n$ in $\N$, $\mathbf{f^{(n+1)}}$ is well-defined by induction since we have the explicit Duhamel formula along the characteristics for all $i$ in $\br{1,\dots,N}$
$$f^{(n+1)}_i(t,x,v) = e^{-\nu_i(v)t}f_{0i} +\int_0^t e^{-\nu_i(v)(t-s)}\cro{B_i\pa{\mathbf{f^{(n)}}} + Q_i\pa{\mathbf{f^{(n)}}+\mathbf{g}}}(x-sv,v)\:ds. $$
\par We are about to show that $\pa{\mathbf{f^{(n)}}}_{n\in\N}$ is a Cauchy sequence in $L^\infty_tL^1_vL^\infty_x\pa{\mathbf{\bar{w_k}}}$.

\bigskip
Direct computations on the nonlinear operator gives
\begin{equation*}
\begin{split}
\partial_t \pa{\mathbf{f^{(n+1)}}-\mathbf{f^{(n)}}} = &-\boldsymbol\nu(v) \pa{\mathbf{f^{(n+1)}}-\mathbf{f^{(n)}}} + \mathbf{B}\pa{\mathbf{f^{(n)}}-\mathbf{f^{(n-1)}}}
\\&+\mathbf{\tilde{Q}}\pa{\mathbf{f^{(n)}}-\mathbf{f^{(n-1)}},\mathbf{f^{(n-1)}}+\mathbf{g}} + \mathbf{\tilde{Q}}\pa{\mathbf{f^{(n)}}+\mathbf{g},\mathbf{f^{(n)}}-\mathbf{f^{(n-1)}}},
\end{split}
\end{equation*}
where we remind that $\mathbf{\tilde{Q}}$ was defined by $\eqref{tildeQ}$ and that $\mathbf{\tilde{Q}}(\mathbf{a},\mathbf{a}) - \mathbf{\tilde{Q}}(\mathbf{b},\mathbf{b})= \mathbf{\tilde{Q}}(\mathbf{a-b},\mathbf{b}) + \mathbf{\tilde{Q}}(\mathbf{a},\mathbf{a-b})$ .
\par Taking the $L^1_vL^\infty_x\pa{\mathbf{\bar{w_k}}}$-norm of $\pa{\mathbf{f^{(n+1)}}-\mathbf{f^{(n)}}}$ and summing over $i$ from $1$ to $N$ gives for all $t\geq 0$
\begin{equation*}
\begin{split}
&\norm{\mathbf{f^{(n+1)}}(t)-\mathbf{f^{(n)}}(t)}_{L^1_vL^\infty_x\pa{\mathbf{\bar{w_k}}}}
\\&\quad\quad\quad\leq\sum\limits_{i=1}^N \int_0^tds\int_{\R^3}dv\:  e^{-\nu_i(v)(t-s)}\pa{1+m_i^{k/2}\abs{v}^k}\norm{\Delta_{ni}\pa{\mathbf{f^{(n)}}-\mathbf{f^{(n-1)}}}}_{L^\infty_x}.
\end{split}
\end{equation*}
where we defined
\begin{equation*}
\begin{split}
&\boldsymbol\Delta_n\pa{\mathbf{f^{(n)}}-\mathbf{f^{(n-1)}}}
\\&\quad= \mathbf{B}\pa{\mathbf{f^{(n)}}-\mathbf{f^{(n-1)}}} + \mathbf{\tilde{Q}}\pa{\mathbf{f^{(n)}}-\mathbf{f^{(n-1)}},\mathbf{f^{(n-1)}}+\mathbf{g}} + \mathbf{\tilde{Q}}\pa{\mathbf{f^{(n)}}+\mathbf{g},\mathbf{f^{(n)}}-\mathbf{f^{(n-1)}}}.
\end{split}
\end{equation*}
As $\nu_i(v)\geq \nu_0$ for all $i$ and $v$ we further get
\begin{equation}\label{existence1}
\begin{split}
&\norm{\mathbf{f^{(n+1)}}(t)-\mathbf{f^{(n)}}(t)}_{L^1_vL^\infty_x\pa{\mathbf{\bar{w_k}}}}\leq \int_0^t e^{-\nu_0(t-s)}\norm{\boldsymbol\Delta_n\pa{\mathbf{f^{(n)}}-\mathbf{f^{(n-1)}}}}_{L^1_vL^\infty_x\pa{\mathbf{\bar{w_k}}}}ds
\\&\quad\leq \cro{C_B + C_Q\pa{\norm{\mathbf{f^{(n)}}}_{L^\infty_tL^1_vL^\infty_x\pa{\mathbf{\bar{w_k}}}}+\norm{\mathbf{f^{(n-1)}}}_{L^\infty_tL^1_vL^\infty_x\pa{\mathbf{\bar{w_k}}}}+ 2\norm{\mathbf{g}}_{L^\infty_tL^1_vL^\infty_x\pa{\mathbf{\bar{w_k}}}}}}
\\&\quad\quad\quad\times\int_0^t e^{-\nu_0(t-s)}\norm{\mathbf{f^{(n)}}(s)-\mathbf{f^{(n-1)}}(s)}_{L^1_vL^\infty_x\pa{\boldsymbol\nu\mathbf{\bar{w_k}}}}\:ds
\\&\quad\quad+C_Q\cro{\int_0^t e^{-\nu_0(t-s)}\pa{\norm{\mathbf{f^{(n)}}}_{L^1_vL^\infty_x\pa{\boldsymbol\nu\mathbf{\bar{w_k}}}}+\norm{\mathbf{f^{(n-1)}}}_{L^1_vL^\infty_x\pa{\boldsymbol\nu\mathbf{\bar{w_k}}}}}\:ds}
\\&\quad\quad\quad \times\sup\limits_{s\in [0,t]}\norm{\mathbf{f^{(n)}}(s)-\mathbf{f^{(n-1)}}(s)}_{L^1_vL^\infty_x\pa{\mathbf{\bar{w_k}}}}.
\end{split}
\end{equation}
where, as above, we used Lemma \ref{lem:controlB} and the estimate of Lemma \ref{lem:controlQ}.

\bigskip
Let us look at the terms inside the time integrals. To this end, we take the $L^1_tL^1_vL^\infty_x\pa{\boldsymbol\nu\mathbf{\bar{w_k}}}$-norm of $\pa{\mathbf{f^{(n+1)}}-\mathbf{f^{(n)}}}$ and we sum over $i$.
\begin{equation*}
\begin{split}
&\int_0^t\norm{\mathbf{f^{(n+1)}}(s)-\mathbf{f^{(n)}}(s)}_{L^1_vL^\infty_x\pa{\langle v \rangle^k\boldsymbol\nu}}\:ds
\\&\quad\quad\quad\leq\sum\limits_{i=1}^N \int_0^t\int_0^s\int_{\R^3}e^{-\nu_i(v)(s-s_1)}\nu_i(v)\bar{w}_{ki}(v)\norm{\boldsymbol\Delta_n\pa{\mathbf{f^{(n)}}-\mathbf{f^{(n-1)}}}}_{L^\infty_x}(s_1)\:ds_1ds.
\end{split}
\end{equation*}
We exchange the integration domains in $s$ and $s_1$, which implies
\begin{equation*}
\begin{split}
&\int_0^t\norm{\mathbf{f^{(n+1)}}(s)-\mathbf{f^{(n)}}(s)}_{L^1_vL^\infty_x\pa{\mathbf{\bar{w_k}}\boldsymbol\nu}}\:ds
\\&\leq\sum\limits_{i=1}^N \int_0^t\int_{\R^3}\pa{\int_{s_1}^te^{-\nu_i(v)(s-s_1)}\nu_i(v)\:ds}\bar{w}_{ki}(v)\norm{\boldsymbol\Delta_n\pa{\mathbf{f^{(n)}}-\mathbf{f^{(n-1)}}}}_{L^\infty_x}(s_1)\:ds_1.
\end{split}
\end{equation*}
Since the integral in $s$ is bounded by $1$, we use Lemma \ref{lem:controlB} and Lemma \ref{lem:controlQ} again and obtain
\begin{equation}\label{existence2}
\begin{split}
&\int_0^t\norm{\mathbf{f^{(n+1)}}(s)-\mathbf{f^{(n)}}(s)}_{L^1_vL^\infty_x\pa{\mathbf{\bar{w_k}}\boldsymbol\nu}}\:ds
\\&\quad\leq \cro{C_B + C_Q\pa{\norm{\mathbf{f^{(n)}}}_{L^\infty_tL^1_vL^\infty_x\pa{\mathbf{\bar{w_k}}}}+\norm{\mathbf{f^{(n-1)}}}_{L^\infty_tL^1_vL^\infty_x\pa{\mathbf{\bar{w_k}}}}+ 2\norm{\mathbf{g}}_{L^\infty_tL^1_vL^\infty_x\pa{\mathbf{\bar{w_k}}}}}}
\\&\quad\quad\times\int_0^t \norm{\mathbf{f^{(n)}}(s_1)-\mathbf{f^{(n-1)}}(s_1)}_{L^1_vL^\infty_x\pa{\mathbf{\bar{w_k}}\boldsymbol\nu}}\:ds_1
\\&\quad\quad+C_Q\cro{\int_0^t\pa{\norm{\mathbf{f^{(n)}}}_{L^1_vL^\infty_x\pa{\mathbf{\bar{w_k}}\boldsymbol\nu}}+\norm{\mathbf{f^{(n-1)}}}_{L^1_vL^\infty_x\pa{\mathbf{\bar{w_k}}\boldsymbol\nu}}}\:ds_1}
\\&\quad\quad\quad \times\sup\limits_{s\in [0,t]}\norm{\mathbf{f^{(n)}}(s)-\mathbf{f^{(n-1)}}(s)}_{L^1_vL^\infty_x\pa{\mathbf{\bar{w_k}}}}
\end{split}
\end{equation}

\bigskip
We now conclude the proof of existence. Indeed, exact same computations but subtracting $e^{-\boldsymbol\nu(v)t}\mathbf{f_0}$ instead of $\mathbf{f^{(n)}}$  lead to $\eqref{existence1}$ and $\eqref{existence2}$ with $\mathbf{f^{(n-1)}}$ replaced by $0$. Therefore, since $C_B <1$ it follows that for $\norm{\mathbf{f_0}}_{L^1_vL^\infty_x\pa{\mathbf{\bar{w_k}}}}$ and $\norm{\mathbf{g}}_{L^\infty_tL^1_vL^\infty_x\pa{\mathbf{\bar{w_k}}\boldsymbol\nu}}$ sufficiently small we have that there exists $C>0$ such that for all $n$ in $\N$ and all $t\geq 0$,
$$\norm{\mathbf{f^{(n)}}(t)}_{L^1_vL^\infty_x\pa{\mathbf{\bar{w_k}}}}\leq C\norm{\mathbf{f_0}}_{L^1_vL^\infty_x\pa{\mathbf{\bar{w_k}}}}$$
and
$$\int_0^t\norm{\mathbf{f^{(n)}}(s)}_{L^1_vL^\infty_x\pa{\mathbf{\bar{w_k}}\boldsymbol\nu}}\:ds\leq C\int_0^t\norm{\mathbf{f^{(1)}}}_{L^1_vL^\infty_x\pa{\mathbf{\bar{w_k}}\boldsymbol\nu}}\:ds \leq C \norm{\mathbf{f_0}}_{L^1_vL^\infty_x\pa{\mathbf{\bar{w_k}}}}. $$
\par Therefore, denoting by $C$ any positive constant independent of $\mathbf{f^{(n)}}$ and $\mathbf{g}$, adding $\eqref{existence1}$ and $\eqref{existence2}$ yields
\begin{equation*}
\begin{split}
&\norm{\mathbf{f^{(n+1)}}(t)-\mathbf{f^{(n)}}(t)}_{L^1_vL^\infty_x\pa{\mathbf{\bar{w_k}}}} + \int_0^t\norm{\mathbf{f^{(n+1)}}(s)-\mathbf{f^{(n)}}(s)}_{L^1_vL^\infty_x\pa{\mathbf{\bar{w_k}}\boldsymbol\nu}}\:ds 
\\&\quad\leq C\eta_1 \sup\limits_{s\in [0,t]}\norm{\mathbf{f^{(n)}}(s)-\mathbf{f^{(n-1)}}(s)}_{L^1_vL^\infty_x\pa{\mathbf{\bar{w_k}}}}
\\&\quad\quad + \cro{C_B + C\eta_1}\int_0^t\norm{\mathbf{f^{(n)}}(s)-\mathbf{f^{(n-1)}}(s)}_{L^1_vL^\infty_x\pa{\mathbf{\bar{w_k}}\boldsymbol\nu}}\:ds.
\end{split}
\end{equation*}

\par Since $C_B <1$, choosing $\eta_1$ such that $C_B + C\eta_1 <1$ implies that $\pa{\mathbf{f}^{(n)}}_{n\in\N}$ is a Cauchy sequence in $L^\infty_tL^1_vL^\infty_x\pa{\mathbf{\bar{w_k}}}$. Hence, $\pa{\mathbf{f}^{(n)}}_{n\in\N}$ converges to a function $\mathbf{f_1}$ in $L^\infty_tL^1_vL^\infty_x\pa{\mathbf{\bar{w_k}}}$ and since $k>k_0>\gamma$ we can take the limit inside the iterative scheme and $\mathbf{f_1}$ is thus a solution of our differential equation.
\end{proof}
\bigskip


\subsubsection{Study of the equations in $L^\infty_{x,v}\pa{\langle v \rangle^\beta\boldsymbol\mu^{-1/2}}$}

We turn to the system $\eqref{f2}$ in $L^\infty_{x,v}\pa{\langle v \rangle^\beta\boldsymbol\mu^{-1/2}}$ with $\beta >3/2$ so that Theorem \ref{theo:semigroupLinfty} holds.

\bigskip
\begin{prop}\label{prop:cauchyLinfty}
Let $k>k_0$, $\beta >3/2$ and let assumptions $(H1) - (H4)$ hold for the collision kernel. Let $\mathbf{g}=\mathbf{g}(t,x,v)$ be in $L^\infty_tL^1_vL^\infty_x\pa{\langle v \rangle^k}$. Then there exists a unique function $\mathbf{f_2}$ in $L^\infty_tL^\infty_{x,v}\pa{\langle v \rangle^\beta\boldsymbol\mu^{-1/2}}$ such that
$$\partial_t \mathbf{f_2} = \mathbf{G}\pa{\mathbf{f_2}} + \mathbf{A}\pa{\mathbf{g}} \quad\mbox{and}\quad \mathbf{f_2}(0,x,v)= 0.$$
Moreover, if $\:\Pi_\mathbf{G}\pa{\mathbf{f_2}+\mathbf{g}}=0$ and if
$$\exists\: \lambda_g,\:\eta_g>0,\:\forall t\geq 0, \: \norm{\mathbf{g}(t)}_{L^1_vL^\infty_x\pa{\langle v \rangle^k}}\leq \eta_ge^{-\lambda_g t},$$
then  for any $0<\lambda_2<\min\br{\lambda_g,\:\lambda_\infty}$, with $\lambda_\infty$ defined in Theorem \ref{theo:semigroupLinfty}, there exist $C_2>0$ such that
$$\forall t\geq 0, \quad \norm{\mathbf{f_2}(t)}_{L^\infty_{x,v}\pa{\langle v \rangle^\beta\boldsymbol\mu^{-1/2}}} \leq C_2\eta_g e^{-\lambda_2 t}.$$
The constant $C_2$ only depends on $\lambda_2$.
\end{prop}
\bigskip

\begin{proof}[Proof of Proposition \ref{prop:cauchyLinfty}]
Thanks to the regularising property of $\mathbf{A}$, Lemma \ref{lem:controlA}, $\mathbf{A}\pa{\mathbf{g}}$ belongs to $L^\infty_tL^\infty_{x,v}\pa{\langle v \rangle^\beta\boldsymbol\mu^{-1/2}}$. Theorem \ref{theo:semigroupLinfty} implies that there is indeed a unique $\mathbf{f_2}$ solution to the differential system, given by
$$\mathbf{f_2} = \int_0^tS_{\mathbf{G}}(t-s)\cro{\mathbf{A}\pa{\mathbf{g}}(s)}\:ds,$$
where $S_{\mathbf{G}}(t)$ is the semigroup generated by $\mathbf{G}$ in $L^\infty_{x,v}\pa{\langle v \rangle^\beta\boldsymbol\mu^{-1/2}}$.

\bigskip
Suppose now that $\Pi_\mathbf{G}\pa{\mathbf{f_2}+\mathbf{g}}=0$ and that there exists $ \eta_2>0$ such that $\norm{\mathbf{g}(t)}_{L^1_vL^\infty_x\pa{\langle v \rangle^k}}\leq \eta_2e^{-\lambda t}$.
\par Using the definition of $\Pi_{\mathbf{G}}$ $\eqref{PiG}$, the projection part of $\mathbf{f_2}$ is straightforwardly bounded for all $t\geq 0$:
\begin{equation}\label{PiGf2}
\begin{split}
\norm{\Pi_\mathbf{G}\pa{\mathbf{f_2}}(t)}_{L^\infty_{x,v}\pa{\langle v \rangle^\beta\boldsymbol\mu^{-1/2}}} &= \norm{\Pi_\mathbf{G}\pa{\mathbf{g}}(t)}_{L^\infty_{x,v}\pa{\langle v \rangle^\beta\boldsymbol\mu^{-1/2}}}\leq C_{\Pi_G} \norm{\mathbf{g}}_{L^1_vL^\infty_x\pa{\langle v \rangle^k}}
\\&\leq C_{\Pi_G}\eta_g\:e^{-\lambda_g t}.
\end{split}
\end{equation}

\par Applying $\Pi_{\mathbf{G}}^\bot = \mbox{Id}-\Pi_\mathbf{G}$ to the equation satisfied by $\mathbf{f_2}$ we get, thanks to the fact the definition of $\Pi_{\mathbf{G}}$ $\eqref{PiG}$ which is independent of $t$,
$$\partial_t \cro{\Pi_\mathbf{G}^\bot\pa{\mathbf{f_2}}} = \mathbf{G}\cro{\Pi_\mathbf{G}^\bot\pa{\mathbf{f_2}}} + \Pi_\mathbf{G}^\bot\pa{\mathbf{A}\pa{\mathbf{g}}}.$$
This yields
$$\Pi_\mathbf{G}^\bot\pa{\mathbf{f_2}} = \int_0^tS_{\mathbf{G}}(t-s)\cro{ \Pi_\mathbf{G}^\bot\pa{\mathbf{A}\pa{\mathbf{g}}}(s)}\:ds.$$
We now use the exponential decay of $S_{\mathbf{G}}(t)$ on $\pa{\mbox{Ker}(\mathbf{G})}^\bot$, see Theorem \ref{theo:semigroupLinfty}.
$$\norm{\Pi_\mathbf{G}^\bot\pa{\mathbf{f_2}}}_{L^\infty_{x,v}\pa{\langle v \rangle^\beta\boldsymbol\mu^{-1/2}}} \leq C_\infty \int_0^t e^{-\lambda_\infty (t-s)}\norm{\Pi_\mathbf{G}^\bot\pa{\mathbf{A}\pa{\mathbf{g}}}(s)}_{L^\infty_{x,v}\pa{\langle v \rangle^\beta\boldsymbol\mu^{-1/2}}}\:ds.$$
Using the definition of $\Pi_{\mathbf{G}}$ $\eqref{PiG}$ and then the regularising property of $\mathbf{A}$ Lemma \ref{lem:controlA} we further bound, for a fixed $\lambda_2 < \min\br{\lambda_\infty,\:\lambda_g}$,
\begin{eqnarray}
\norm{\Pi_\mathbf{G}^\bot\pa{\mathbf{f_2}}}_{L^\infty_{x,v}\pa{\langle v \rangle^\beta\boldsymbol\mu^{-1/2}}} &\leq& C_\infty C_{\Pi_G}C_AC_g\eta_g \int_0^t e^{-\lambda_\infty(t-s)}e^{-\lambda_g s}\:ds \nonumber
\\&\leq& C_GC_\infty C_{\Pi_G}C_AC_g\eta_g \:  te^{-\min\br{\lambda_g,\lambda_\infty} t}\nonumber
\\&\leq& C_2(\lambda_2)\eta_g e^{-\lambda_2 t}.\label{PiGbotf2}
\end{eqnarray}
\par Gathering $\eqref{PiGf2}$ and $\eqref{PiGbotf2}$ yields the desired exponential decay.
\end{proof}
\bigskip


\subsubsection{Proof of Proposition \ref{prop:existenceexpodecay}}

Take $\mathbf{f_0}$ in $L^1_vL^\infty_x\pa{\langle v \rangle^k}$ such that $\Pi_\mathbf{G}(\mathbf{f_0})=0$.
\par The existence will be proved by an iterative scheme. We start with $\mathbf{f^{(0)}_1}=\mathbf{f^{(0)}_2}=0$ and we approximate the system of equation $\eqref{f1}-\eqref{f2}$ as follows.
\begin{eqnarray*}
\partial_t \mathbf{f^{(n+1)}_1}&=& \mathbf{G_1} \pa{\mathbf{f^{(n+1)}_1}} + \mathbf{Q}\pa{\mathbf{f^{(n+1)}_1}+\mathbf{f^{(n)}_2}}
\\\partial_t \mathbf{f^{(n+1)}_2} &=& \mathbf{G}\pa{\mathbf{f^{(n+1)}_2}} + \mathbf{A^{(\delta)}}\pa{\mathbf{f^{(n+1)}_1}},
\end{eqnarray*}
with the following initial data
$$\mathbf{f^{(n+1)}_1}(0,x,v)=\mathbf{f_0}(x,v) \quad\mbox{and}\quad \mathbf{f^{(n+1)}_2}(0,x,v)=0.$$
\par Assume that $(1+C_1C_2)\norm{\mathbf{f_0}}\leq \eta_1$, where $C_1,\:\eta_1$ were defined in Proposition \ref{prop:cauchyE} and $C_2$ was defined in Proposition \ref{prop:cauchyLinfty}. Thanks to Proposition \ref{prop:cauchyE} and Proposition \ref{prop:cauchyLinfty}, an induction proves first that $\pa{\mathbf{f^{(n)}_1}}_{n\in\N}$ and $\pa{\mathbf{f^{(n)}_2}}_{n\in\N}$ are well-defined sequences and second that for all $n$ in $\N$ and all $t\geq 0$
\begin{eqnarray}
\norm{\mathbf{f^{(n)}_1}(t)}_{L^1_vL^\infty_x\pa{\langle v \rangle^k}} &\leq& e^{-\lambda_1 t}\norm{\mathbf{f_0}}_{L^1_vL^\infty_x\pa{\langle v \rangle^k}} \label{expodecayfn1}
\\\norm{\mathbf{f^{(n)}_2}(t)}_{L^\infty_{x,v}\pa{\langle v \rangle^\beta\boldsymbol\mu^{-1/2}}} &\leq& C_1C_2 e^{-\lambda_2 t}\norm{\mathbf{f_0}}_{L^1_vL^\infty_x\pa{\langle v \rangle^k}},\label{expodecayfn2}
\end{eqnarray}
with $\lambda_2 <\min\br{\lambda_1,\lambda_\infty}$. Indeed, if we constructed $\mathbf{f^{(n)}_1}$ and $\mathbf{f^{(n)}_2}$ satisfying the exponential decay above then we can construct $\mathbf{f^{(n+1)}_1}$ with Proposition \ref{prop:cauchyE} and $\mathbf{g} = \mathbf{f^{(n)}_2}$, which has the required exponential decay $\eqref{expodecayfn1}$, and then construct $\mathbf{f^{(n+1)}_2}$ with Proposition \ref{prop:cauchyLinfty} and $\mathbf{g} = \mathbf{f^{(n+1)}_1}$. Finally, we have the following equality
$$\partial_t\pa{\mathbf{f^{(n+1)}_1}+\mathbf{f^{(n+1)}_2}} = \mathbf{G}\pa{\mathbf{f^{(n+1)}_1}+\mathbf{f^{(n+1)}_2}} + \mathbf{Q}\pa{\mathbf{f^{(n+1)}_1}+\mathbf{f^{(n)}_2}}.$$
Thanks to orthogonality property of $\mathbf{Q}$ in Lemma \ref{lem:controlQ} and the definition of $\Pi_\mathbf{G}$ $\eqref{PiG}$ we obtain that the projection is constant with time and thus
$$\Pi_\mathbf{G}\pa{\mathbf{f^{(n+1)}_1}+\mathbf{f^{(n+1)}_2}} = \Pi_\mathbf{G}(\mathbf{f_0})=0.$$
Applying Proposition \ref{prop:cauchyLinfty} we obtain the exponential decay $\eqref{expodecayfn2}$ for $\mathbf{f^{(n+1)}_2}$.

\bigskip
We recognize exactly the same iterative scheme for $\mathbf{f^{n+1}_1}$ as in the proof of Proposition \ref{prop:cauchyE} with $\mathbf{g}$ replaced by $\mathbf{f_2^{(n)}}$. Moreover, the uniform bound $\eqref{expodecayfn2}$ allows us to derive the same estimates as in the latter proof independently of $\mathbf{f^{(n)}_2}$. As a conclusion, $\pa{\mathbf{f^{(n)}_1}}_{n\in\N}$ is a Cauchy sequence in $L^\infty_tL^1_vL^\infty_x\pa{\langle v \rangle^k}$ and therefore converges strongly towards a function $\mathbf{f_1}$.
\par By $\eqref{expodecayfn2}$, the sequence $\pa{\mathbf{f^{(n)}_2}}_{n\in\N}$ is bounded in $L^\infty_tL^\infty_{x,v}\pa{\langle v \rangle^\beta\boldsymbol\mu^{-1/2}}$ and is therefore weakly-* compact and therefore converges, up to a subsequence, weakly-* towards $\mathbf{f_2}$ in $L^\infty_tL^\infty_{x,v}\pa{\langle v \rangle^\beta\boldsymbol\mu^{-1/2}}$.
\par Since the function inside the collision operator behaves like $\abs{v-v_*}^\gamma$ and that in our weighted spaces $k>k_0>\gamma$, we can take the weak limit inside the iterative scheme. This implies that $(\mathbf{f_1},\mathbf{f_2})$ is solution to the system $\eqref{f1}-\eqref{f2}$ and thus $\mathbf{f} = \mathbf{f_1}+\mathbf{f_2}$ is solution to the perturbed multi-species equation $\eqref{perturbedmultiBEcauchy}$. Moreover, taking the limit inside the exponential decays $\eqref{expodecayfn1}$ and $\eqref{expodecayfn2}$ yields the expected exponential decay for $\mathbf{f}$.

\bigskip


\subsection{Uniqueness of solutions in the perturbative regime}\label{subsec:uniqueness}

As said in Remark \ref{rem:mainresults}, we are solely interested in the uniqueness of solutions to the multi-species Boltzmann equation $\eqref{multiBE}$ in the perturbative setting. In other terms, uniqueness of solutions of the form $\mathbf{F}=\boldsymbol\mu + \mathbf{f}$ as long as $\mathbf{F_0}$ is close enough to the global equilibrium $\boldsymbol\mu$. This is equivalent to proving the uniqueness of solutions to the perturbed multi-species equation
\begin{equation}\label{perturbeduniqueness}
\partial_t \mathbf{f}  = \mathbf{G}(\mathbf{f}) + \mathbf{Q}(\mathbf{f})
\end{equation}
for $\mathbf{f_0}$ small.

\bigskip
\begin{prop}\label{prop:uniqueness}
Let $k>k_0$ and let assumptions $(H1) - (H4)$ hold for the collision kernel. There exists $\eta_k>0$ such that for any $\mathbf{f_0}$ in $L^1_vL^\infty_x\pa{\langle v \rangle^k}$; if $\norm{\mathbf{f_0}}_{L^1_vL^\infty_x\pa{\langle v \rangle^k}} \leq \eta_k$ then there exists at most one solution to the perturbed multi-species equation $\eqref{perturbeduniqueness}$.
\\The constant $\eta_k$ only depends on $k$, $N$ and the collision kernels.
\end{prop}
\bigskip

The uniqueness will follow from the study of the semigroup generated by $\mathbf{G}$ in a dissipative norm as well as a new \textit{a priori} stability estimate for solutions to $\eqref{perturbeduniqueness}$ in the latter norm. They are the purpose of the next two lemmas.

\bigskip
\begin{lemma}\label{lem:semigroupuniqueness}
Let $k>k_0$ and let assumptions $(H1) - (H4)$ hold for the collision kernel. The operator $\mathbf{G}$ generates a semigroup in $L^1_vL^\infty_x\pa{\langle v \rangle^k}$. Moreover, there exist $C_k$, $\lambda_k>0$ such that for all $\mathbf{f_0}$ in $L^1_vL^\infty_x\pa{\langle v \rangle^k}$ with $\Pi_\mathbf{G}(\mathbf{f_0})=0$
$$\forall t\geq 0, \quad \norm{S_\mathbf{G}(\mathbf{f})}_{L^1_vL^\infty_x\pa{\langle v \rangle^k}}\leq C_k e^{-\lambda_k t}\norm{\mathbf{f_0}}_{L^1_vL^\infty_x\pa{\langle v \rangle^k}}.$$
\end{lemma}
\bigskip

\begin{proof}[Proof of Lemma \ref{lem:semigroupuniqueness}]
From Proposition \ref{prop:existenceexpodecay} with a collision operator $\mathbf{Q}=0$ we have the existence of a solution to the equation
$$\partial_t \mathbf{f} = \mathbf{G}\pa{\mathbf{f}}$$
with initial data $\mathbf{f_0}$ in $L^1_vL^\infty_x\pa{\langle v \rangle^k}$. Moreover, that solution satisfies $\Pi_\mathbf{G}(\mathbf{f})=0$ and it decays exponentially fast with rate $\lambda_k$.
\par Let $\mathbf{g}$ be another solution to the linear equation then
$$\partial_t \pa{\mathbf{f}-\mathbf{g}} = \cro{-v\cdot\nabla_x-\boldsymbol\nu + \mathbf{B}+\mathbf{A}}\pa{\mathbf{f}-\mathbf{g}} .$$
Similar computations as to obtain $\eqref{importanteq}$ yield
\begin{equation*}
\begin{split}
\frac{d}{dt}\norm{\mathbf{f}-\mathbf{g}}_{L^1_vL^\infty_x\pa{\langle v \rangle^k}} \leq& -\norm{\mathbf{f}-\mathbf{g}}_{L^1_vL^\infty_x\pa{\langle v \rangle^k\boldsymbol\nu}} + \norm{\mathbf{A}(\mathbf{f}-\mathbf{g})}_{L^1_vL^\infty_x\pa{\langle v \rangle^k}} 
\\&+ \norm{\mathbf{B}(\mathbf{f}-\mathbf{g})}_{L^1_vL^\infty_x\pa{\langle v \rangle^k}}.
\end{split}
\end{equation*}
Using Lemma \ref{lem:controlA} and Lemma \ref{lem:controlB}, there exists $0<C_B<1$ such that
\begin{equation}\label{semigroupuniquenesseq}
\frac{d}{dt}\norm{\mathbf{f}-\mathbf{g}}_{L^1_vL^\infty_x\pa{\langle v \rangle^k}} \leq -(1-C_B)\norm{\mathbf{f}-\mathbf{g}}_{L^1_vL^\infty_x\pa{\langle v \rangle^k \boldsymbol\nu}} + C_A \norm{\mathbf{f}-\mathbf{g}}_{L^1_vL^\infty_x\pa{\langle v \rangle^k}}.
\end{equation}
Since $(1-C_B)>0$ we can further bound
$$\frac{d}{dt}\norm{\mathbf{f}-\mathbf{g}}_{L^1_vL^\infty_x\pa{\langle v \rangle^k}} \leq \cro{C_A-(1-C_B)}\norm{\mathbf{f}-\mathbf{g}}_{L^1_vL^\infty_x\pa{\langle v \rangle^k}}$$
and a Gr\"onwall lemma therefore yields $\mathbf{f}=\mathbf{g}$ if $\mathbf{g_0}=\mathbf{f_0}$.
\par We thus obtain existence and uniqueness of solution to the linear equation which means that $\mathbf{G}$ generates a semigroup in $L^1_vL^\infty_x\pa{\langle v \rangle^k}$. Moreover it has an exponential decay of rate $\lambda_k>0$ for functions in $\pa{\mbox{Ker}(\mathbf{G})}^\bot$.
\end{proof}
\bigskip

We now derive a stability estimate in an equivalent norm that catches the dissipativity of the linear operator.

\bigskip
\begin{lemma}\label{lem:estimateuniqueness}
Let $k>k_0$ and let assumptions $(H1) - (H4)$ hold for the collision kernel. For $\alpha >0$, we define
$$\norm{\mathbf{f}}_{\alpha,k} = \alpha \norm{\mathbf{f}}_{L^1_vL^\infty_x\pa{\langle v \rangle^k}} + \int_0^{+\infty} \norm{S_{\mathbf{G}}(s)\pa{\mathbf{f}}}_{L^1_vL^\infty_x\pa{\langle v \rangle^k}}\:ds.$$
There exist $\eta$, $\alpha$, $C_1$, $C_2$ and $\lambda>0$ such that $\norm{\cdot}_{\alpha,k} \sim \norm{\cdot}_{L^1_vL^\infty_x\pa{\langle v \rangle^k}}$ and for all $\mathbf{f_0}$ in $L^1_vL^\infty_x\pa{\langle v \rangle^k}$ with $\Pi_\mathbf{G}(\mathbf{f_0})=0$ and such that
$$\norm{\mathbf{f_0}}_{L^1_vL^\infty_x\pa{\langle v \rangle^k}}\leq \eta;$$
if $\mathbf{f}$ in $L^1_vL^\infty_x\pa{\langle v \rangle^k}$ with $\Pi_\mathbf{G}(\mathbf{f})=0$ is solution to the perturbed equation $\eqref{perturbeduniqueness}$ with initial data $\mathbf{f_0}$ then
$$\frac{d}{dt}\norm{\mathbf{f}}_{\alpha,k} \leq -\pa{C_1 - C_2\norm{\mathbf{f}}_{\alpha,k}}\norm{\mathbf{f}}_{\alpha,k,\boldsymbol\nu},$$
where the subscript $\boldsymbol\nu$ refers to the fact that the weight is multiplied by $\nu_i(v)$ on each coordinate.
\end{lemma}
\bigskip

\begin{proof}[Proof of Lemma \ref{lem:estimateuniqueness}]
Start with the new norm. Lemma \ref{lem:semigroupuniqueness} proved that for all $\mathbf{f_0}$ such that $\Pi_{\mathbf{G}}(\mathbf{f_0})=0$ and all $s\geq 0$, 
$$\norm{S_{\mathbf{G}}(s)\pa{\mathbf{f_0}}}_{L^1_vL^\infty_x\pa{\langle v \rangle^k}} \leq C_k e^{-\lambda_k s}\norm{\mathbf{f_0}}_{L^1_vL^\infty_x\pa{\langle v \rangle^k}}$$
and hence
$$\alpha \norm{\mathbf{f_0}}_{L^1_vL^\infty_x\pa{\langle v \rangle^k}} \leq \norm{\mathbf{f_0}}_{\alpha,k} \leq \pa{\alpha +\frac{C_k}{\lambda_k}}\norm{\mathbf{f_0}}_{L^1_vL^\infty_x\pa{\langle v \rangle^k}}.$$

\bigskip
Suppose that $\mathbf{f}$ is the solution described in Lemma \ref{lem:estimateuniqueness}. Same computations as to obtain $\eqref{importanteq}$ and $\eqref{semigroupuniquenesseq}$ yields
\begin{equation*}
\frac{d}{dt}\norm{\mathbf{f}}_{L^1_vL^\infty_x\pa{\langle v \rangle^k}} \leq -(1-C_B)\norm{\mathbf{f}}_{L^1_vL^\infty_x\pa{\langle v \rangle^k \boldsymbol\nu}} + C_A \norm{\mathbf{f}}_{L^1_vL^\infty_x\pa{\langle v \rangle^k}} + \norm{\mathbf{Q}(\mathbf{f})}_{L^1_vL^\infty_x\pa{\langle v \rangle^k}}.
\end{equation*}
To which we can apply Lemma \ref{lem:controlQ}:
\begin{equation}\label{estimateuniquenesseq1}
\frac{d}{dt}\norm{\mathbf{f}}_{L^1_vL^\infty_x\pa{\langle v \rangle^k}} \leq -\pa{1-C_B-C_Q\norm{\mathbf{f}}_{L^1_vL^\infty_x\pa{\langle v \rangle^k}}}\norm{\mathbf{f}}_{L^1_vL^\infty_x\pa{\langle v \rangle^k \boldsymbol\nu}} + C_A \norm{\mathbf{f}}_{L^1_vL^\infty_x\pa{\langle v \rangle^k}}.
\end{equation}

\par We now turn to the second term in the $\norm{\cdot}_{\alpha,k}$ norm. For $q$ in $[1,\infty)$ we denote $\Phi_q(\mathbf{F}) = \mbox{sgn}(\mathbf{F})\abs{\mathbf{F}}^{q-1}$, where it has to be understood component by component. We thus have
\begin{equation*}
\begin{split}
&\frac{d}{dt}\int_0^{+\infty}\norm{S_{\mathbf{G}}(s)\pa{\mathbf{f}(t)}}_{L^1_vL^q_x\pa{\langle v \rangle^k}}\:ds 
\\&\quad\quad\quad= \int_0^{+\infty}\int_{\R^3}\langle v \rangle^k\norm{S_\mathbf{G}(s)(\mathbf{f})}_{L^q_x}^{1-q}\pa{\int_{\T^3}\Phi_q\pa{S_\mathbf{G}(s)(\mathbf{f})} S_\mathbf{G}(s)\cro{\mathbf{G}(\mathbf{f})}\:dx}\:dvds
\\&\quad\quad\quad\quad+ \int_0^{+\infty}\int_{\R^3}\langle v \rangle^k\norm{S_\mathbf{G}(s)(\mathbf{f})}_{L^q_x}^{1-q}\pa{\int_{\T^3}\Phi_q\pa{S_\mathbf{G}(s)(\mathbf{f})}S_\mathbf{G}(s)\cro{\mathbf{Q}(\mathbf{f})}\:dx}\:dvds
\end{split}
\end{equation*}
First, by definition of $S_\mathbf{G}(s)$ we have that 
$$\Phi_q\pa{S_\mathbf{G}(s)(\mathbf{f}(t))} S_\mathbf{G}(s)\cro{\mathbf{G}(\mathbf{f}(t))} = \frac{d}{ds}\abs{S_\mathbf{G}(s)(\mathbf{f}(t))}^q.$$
Second, by H\"older inequality with $q$ and $q/(q-1)$ (see $\eqref{holder}$):
$$\int_{\T^3}\Phi_q\pa{S_\mathbf{G}(s)(\mathbf{f}(t))}S_\mathbf{G}(s)\cro{\mathbf{Q}(\mathbf{f}(t))}\:dx \leq \norm{S_\mathbf{G}(s)(\mathbf{f}(t))}_{L^q_x}^{q-1}\norm{S_\mathbf{G}(s)(\mathbf{Q}(\mathbf{f}(t)))}_{L^q_x}.$$
We therefore get
\begin{equation*}
\begin{split}
\frac{d}{dt}\int_0^{+\infty}\norm{S_{\mathbf{G}}(s)\pa{\mathbf{f}(t)}}_{L^1_vL^q_x\pa{\langle v \rangle^k}}\:ds \leq& \int_0^{+\infty} \frac{d}{ds}\norm{S_\mathbf{G}(\mathbf{f}(t))}_{L^1_vL^q_x\pa{\langle v \rangle^k}}\:ds 
\\&+ \int_0^{+\infty} \norm{S_\mathbf{G}(s)(\mathbf{Q}(\mathbf{f}(t)))}_{L^1_vL^q_x\pa{\langle v \rangle^k}}\:ds.
\end{split}
\end{equation*}
We make $q$ tend to infinity. Then we have $\Pi_\mathbf{G}(\mathbf{Q}(\mathbf{f}(t)))=0$ by Lemma \ref{lem:controlQ} so we are able to use the exponential decay of $S_\mathbf{G}(s)$ Lemma \ref{lem:semigroupuniqueness}. This yields
\begin{equation*}
\begin{split}
\frac{d}{dt}\int_0^{+\infty}\norm{S_{\mathbf{G}}(s)\pa{\mathbf{f}(t)}}_{L^1_vL^\infty_x\pa{\langle v \rangle^k}}\:ds \leq& -\norm{\mathbf{f}(t)}_{L^1_vL^\infty_x\pa{\langle v \rangle^k}}
\\&+ C_k\pa{\int_0^{+\infty}e^{-\lambda_k s}\:ds} \norm{\mathbf{Q}(\mathbf{f}(t))}_{L^1_vL^\infty_x\pa{\langle v \rangle^k}}.
\end{split}
\end{equation*}
With Lemma \ref{lem:controlQ} we control $\mathbf{Q}(\mathbf{f})$:
\begin{equation}\label{estimateuniquenesseq2}
\begin{split}
\frac{d}{dt}\int_0^{+\infty}\norm{S_{\mathbf{G}}(s)\pa{\mathbf{f}(t)}}_{L^1_vL^\infty_x\pa{\langle v \rangle^k}}\:ds \leq& -\norm{\mathbf{f}(t)}_{L^1_vL^\infty_x\pa{\langle v \rangle^k}} 
\\&+ \frac{C_kC_Q}{\lambda_k}\norm{\mathbf{f}}_{L^1_vL^\infty_x\pa{\langle v \rangle^k}}\norm{\mathbf{f}}_{L^1_vL^\infty_x\pa{\langle v \rangle^k\boldsymbol\nu}}.
\end{split}
\end{equation}

\bigskip
To conclude we add $\alpha \times \eqref{estimateuniquenesseq1} + \eqref{estimateuniquenesseq2}$,
\begin{equation}\label{finalestimate}
\begin{split}
\frac{d}{dt}\norm{\mathbf{f}}_{k,\alpha} \leq& -\cro{\alpha(1-C_B)-\pa{\alpha C_Q+\frac{C_kC_Q}{\lambda_k}}\norm{\mathbf{f}}_{L^1_vL^\infty_x\pa{\langle v \rangle^k}}}\norm{\mathbf{f}}_{L^1_vL^\infty_x\pa{\langle v \rangle^k\boldsymbol\nu}}
\\&+ \cro{\alpha C_A - 1}\norm{\mathbf{f}}_{L^1_vL^\infty_x\pa{\langle v \rangle^k}}.
\end{split}
\end{equation}
Choosing $\alpha$ such that $(\alpha C_A -1) <0$ yields the desired estimate.
\end{proof}
\bigskip

We now prove the uniqueness proposition.

\bigskip
\begin{proof}[Proof of Proposition \ref{prop:uniqueness}]
Let $\mathbf{f}$ and $\mathbf{g}$ in $L^1_vL^\infty_x\pa{\langle v \rangle^k}$, $k>k_0$, be two solutions of the perturbed equation with initial datum $\mathbf{f_0}$.
\par Thanks to Lemma \ref{lem:estimateuniqueness}, if $\norm{\mathbf{f_0}}_{L^1_vL^\infty_x\pa{\langle v \rangle^k}}$ is small enough we can deduce from the differential inequality that for all $t\geq 0$,
$$\frac{d}{dt}\norm{\mathbf{f}}_{\alpha,k} \leq -C_k\norm{\mathbf{f}}_{\alpha,k,\boldsymbol\nu},$$
and the same holds for $\mathbf{g}$ with the same constant $C_k>0$. We therefore have two estimates on $\mathbf{f}$ and $\mathbf{g}$. Either by integrating from $0$ to $t$:
\begin{equation}\label{boundnu}
\forall t\geq 0,\quad \int_0^t \norm{\mathbf{f}(s)}_{\alpha,k,\boldsymbol\nu}\:ds \leq C_k^{-1}\norm{\mathbf{f_0}}_{\alpha,k};
\end{equation}
or by Gr\"onwall lemma:
\begin{equation}\label{bound}
\forall t\geq 0,\quad \norm{\mathbf{f}(t)}_{\alpha,k}\leq e^{-C_k t}\norm{\mathbf{f_0}}_{\alpha,k}.
\end{equation}
The same estimate holds for $\mathbf{g}$.

\bigskip
Recalling the definition $\eqref{tildeQ}$ of the operator $\mathbf{\tilde{Q}}$, we find the differential equation satisfied by $\mathbf{f}-\mathbf{g}$:
$$\partial_t \pa{\mathbf{f}-\mathbf{g}}  = \mathbf{G}\pa{\mathbf{f}-\mathbf{g}} + \mathbf{\tilde{Q}}\pa{\mathbf{f}-\mathbf{g},\mathbf{f}} + \mathbf{\tilde{Q}}\pa{\mathbf{g},\mathbf{f}-\mathbf{g}}.$$
\par Using controls on $\mathbf{B}$ (Lemma \ref{lem:controlB}), $\mathbf{A}$ (Lemma \ref{lem:controlA}), $\mathbf{\tilde{Q}}$ (Lemma \ref{lem:controlQ}) and the semigoup property (Lemma \ref{lem:semigroupuniqueness}), exact same computations as for $\eqref{finalestimate}$, gives
\begin{equation*}
\begin{split}
\frac{d}{dt}\norm{\mathbf{f}-\mathbf{g}}_{k,\alpha} \leq& -\cro{\alpha(1-C_B)-\pa{\alpha C_Q +\frac{C_k C_Q}{\lambda_k}}\pa{\norm{\mathbf{f}}_{\alpha,k}+\norm{\mathbf{g}}_{\alpha,k}}}\norm{\mathbf{f}-\mathbf{g}}_{\alpha,k,\boldsymbol\nu} 
\\&+\cro{\alpha C_A -1 + C_Q\pa{\norm{\mathbf{f}}_{\alpha,k,\boldsymbol\nu}+\norm{\mathbf{g}}_{\alpha,k,\boldsymbol\nu}}}\norm{\mathbf{f}-\mathbf{g}}_{\alpha,k}.
\end{split}
\end{equation*}
Note that we used the equivalence of the $\norm{\cdot}_{\alpha,k}$ norm and our usual norm (see Lemma \ref{lem:estimateuniqueness}).
\par First, by $\eqref{bound}$ and $C_B<1$, if $\mathbf{f_0}$ is small enough then for all $t\geq 0$,
$$\alpha(1-C_B)-\pa{\alpha C_Q +\frac{C_k C_Q}{\lambda_k}}\pa{\norm{\mathbf{f}}_{\alpha,k}+\norm{\mathbf{g}}_{\alpha,k}} \leq 0.$$
Second we take $\alpha$ small enough so that $(\alpha C_A - 1)<0$. Hence, integrating the differential inequality from $0$ to $t$:
$$\norm{\mathbf{f}(t)-\mathbf{g}(t)}_{k,\alpha} \leq C_Q\cro{\int_0^t\left(\norm{\mathbf{f}(s)}_{\alpha,k,\boldsymbol\nu}+\norm{\mathbf{g(s)}}_{\alpha,k,\boldsymbol\nu}\right)\:ds}\sup\limits_{s\in [0,t]}\norm{\mathbf{f}(s)-\mathbf{g}(s)}_{k,\alpha}.$$
To conclude we use $\eqref{boundnu}$ to obtain
$$\forall t\geq 0,\quad \norm{\mathbf{f}(t)-\mathbf{g}(t)}_{k,\alpha} \leq \frac{2C_Q}{C_k}\norm{\mathbf{f_0}}_{\alpha,k}\pa{\sup\limits_{s\in [0,t]}\norm{\mathbf{f}(s)-\mathbf{g}(s)}_{k,\alpha}},$$
which implies $\mathbf{f}=\mathbf{g}$ if $\norm{\mathbf{f_0}}_{\alpha,k}$ is small enough.
\end{proof}
\bigskip


\subsection{Positivity of solutions}\label{subsec:positivity}

This last subsection is dedicated to the positivity of the solution to the multi-species Boltzmann equation
\begin{equation}\label{multiBEpositivity}
\partial_t \mathbf{F} + v \cdot\nabla_x\mathbf{F} = \mathbf{Q}\pa{\mathbf{F}}
\end{equation}
in the perturbative setting studied above.

\bigskip
\begin{prop}\label{prop:positivity}
Let $k>k_0$, let assumptions $(H1) - (H4)$ hold for the collision kernel, and let $\mathbf{f_0}$ be in $L^1_vL^\infty_x\pa{\langle v \rangle^k}$ with $\Pi_\mathbf{G}(\mathbf{f_0})=0$ and
$$\norm{\mathbf{f_0}}_{L^1_vL^\infty_x\pa{\langle v \rangle^k}}\leq \eta_k,$$
where $\eta_k>0$ is chosen such that Proposition \ref{prop:existenceexpodecay} and Proposition \ref{prop:uniqueness} hold and denote $\mathbf{f}$ the unique solution of the perturbed multi-species equation associated to $\mathbf{f_0}$.
\\ Suppose that $\mathbf{F_0}= \boldsymbol\mu + \mathbf{f_0}\geq 0$ then $\mathbf{F}=\boldsymbol\mu + \mathbf{f} \geq 0$.
\end{prop}
\bigskip

\begin{proof}[Proof of Proposition \ref{prop:positivity}]
Since we are working with the Grad's cutoff assumption we can decompose the nonlinear operator into
$$\mathbf{Q}(\mathbf{F}) = -\mathbf{Q_1}(\mathbf{F}) + \mathbf{Q_2}(\mathbf{F})$$
where
\begin{eqnarray*}
Q_{1i}(\mathbf{F}) &=& \sum\limits_{j=1}^N\int_{\R^3\times \mathbb{S}^{2}}B_{ij}\left(|v - v_*|,\mbox{cos}\:\theta\right) F_iF_j^*\:dv_*d\sigma
\\Q_{2i}(\mathbf{F})&=& \sum\limits_{j=1}^N\int_{\R^3\times \mathbb{S}^{2}}B_{ij}\left(|v - v_*|,\mbox{cos}\:\theta\right)F_i'F_j^{'*}\:dv_*d\sigma.
\end{eqnarray*}

\par Following the idea of \cite{Gu6}, we construct an interative scheme for the multi-species Boltzmann equation
$$\partial_t \mathbf{F^{(n+1)}} + v\cdot\nabla_x\mathbf{F^{(n+1)}} +\mathbf{\bar{Q}_1}(\mathbf{F^{(n+1)}},\mathbf{F^{(n)}}) = \mathbf{Q_2}(\mathbf{F^{(n)}}),$$
with the non-symmetriized bilinear form $\mathbf{\bar{Q}_1}$ defined as
\begin{eqnarray*}
\bar{Q}_{1i}(\mathbf{F},\mathbf{G}) &=& \sum\limits_{j=1}^N\int_{\R^3\times \mathbb{S}^{2}}B_{ij}\left(|v - v_*|,\mbox{cos}\:\theta\right) F_iG_j^*\:dv_*d\sigma
\\\bar{Q}_{2i}(\mathbf{F},\mathbf{G})&=& \sum\limits_{j=1}^N\int_{\R^3\times \mathbb{S}^{2}}B_{ij}\left(|v - v_*|,\mbox{cos}\:\theta\right)F_i'G_j^{'*}\:dv_*d\sigma.
\end{eqnarray*}

\par Defining $\mathbf{f^{(n)}}= \mathbf{F^{n}} -\boldsymbol\mu$ we have the following differential iterative scheme 
$$\partial_t \mathbf{f^{(n+1)}} + v\cdot\nabla_x\mathbf{f^{(n+1)}} = -\boldsymbol\nu(v)\pa{\mathbf{f^{(n+1)}}} + \mathbf{K}\pa{\mathbf{f^{(n)}}} + \mathbf{Q_2}\pa{\mathbf{f^{(n)}}}-\mathbf{\tilde{Q}_1}\pa{\mathbf{f^{(n+1)}},\mathbf{f^{(n)}}}.$$
As before, we can prove that $\pa{\mathbf{f^{(n)}}}_{n\in\N}$ is well-defined and converges in $L^1_vL^\infty_x\pa{\langle v \rangle^k}$ towards $\mathbf{f}$, the unique solution of the perturbed multi-species equation and thus the same holds for $\mathbf{F^{n}}$ converging towards $\mathbf{F}$ the unique perturbed solution of the original multi-species Boltzmann equation.

\bigskip
We prove that $\mathbf{F^{n}}\geq 0$ by an induction on $N$.
\\By definition we see that
$$\mathbf{\tilde{Q}_1}(\mathbf{F^{(n+1)}},\mathbf{F^{(n)}}) = q_1(\mathbf{F^{(n)}})\mathbf{F^{(n+1)}},$$
and thus applying the Duhamel formula along the characteristics gives
\begin{equation*}
\begin{split}
&\mathbf{F^{(n+1)}}(t,x,v)
\\&\: = \mbox{exp}\cro{-\int_0^tq_1(\mathbf{F^{(n)}})(s,x-(t-s)v,v)\:ds}\mathbf{F_0}(x-tv,v) 
\\&\:\quad+ \int_0^t\mbox{exp}\cro{-\int_s^tq_1(\mathbf{F^{(n)}})(s_1,x-(t-s_1)v,v)\:ds_1}\mathbf{Q_2}(\mathbf{F^{(n)}})(s,x-(t-s)v,v)\:ds.
\end{split}
\end{equation*}
By positivity of $\mathbf{F^{(n)}}$, all the terms on the right-hand side are positive and therefore $\mathbf{F^{n+1}}\geq 0$. Passing to the limit implies that $\mathbf{F}\geq 0$.
\end{proof}
\bigskip


\bigskip
\section{Ethical Statement}

Conflict of Interest: The authors declare that they have no conflict of interest.
\bigskip



%
\bibliographystyle{acm}
\bibliography{bibliography_multiBE}


\bigskip
\signmb
\signed

\end{document}